\documentclass[reqno]{amsart}
\RequirePackage[l2tabu, orthodox]{nag}

\usepackage{amssymb}
\usepackage[giveninits=true, style=numeric]{biblatex}
\usepackage{bbold}
\usepackage{bm}
\usepackage{color}
\usepackage{enumitem}
\usepackage{esint}
\usepackage{mathrsfs}
\usepackage{mathtools}
\usepackage{microtype}
\usepackage{calrsfs}
\usepackage{stmaryrd}
\usepackage{tcolorbox}
\tcbuselibrary{theorems}
\usepackage{hyperref}

\theoremstyle{plain}
\newtheorem{theorem}{Theorem}[section]
\newtheorem{corollary}[theorem]{Corollary}
\newtheorem{proposition}[theorem]{Proposition}
\newtheorem{lemma}[theorem]{Lemma}

\theoremstyle{definition}
\newtheorem{definition}[theorem]{Definition}

\theoremstyle{remark}
\newtheorem{example}[theorem]{Example}
\newtheorem{remark}[theorem]{Remark}
\newtheorem{notation}[theorem]{Notation}

\newcommand\set[1]{\left\{#1\right\}}
\newcommand{\IGNORE}[1]{}

\newcommand{\C}{\mathbb{C}}
\renewcommand{\H}{\mathcal{H}}
\newcommand{\N}{\mathbb{N}}
\newcommand{\R}{\mathbb{R}}

\newcommand{\BE}{{\bm{e}}}
\newcommand{\BV}{\bm{v}}
\newcommand{\BW}{\bm{w}}
\newcommand{\DL}{\mathscr{D}}
\newcommand{\GS}{\geqslant}
\newcommand{\LS}{\leqslant}
\newcommand{\NT}{\mathcal{N}}
\newcommand{\SL}{\mathscr{S}}
\newcommand{\TR}{\tau}
\newcommand {\supp} {\text{supp}}

\newcommand{\BMO}{\mathrm{BMO}}
\newcommand{\DOM}{\mathrm{dom}}
\newcommand{\DST}{\displaystyle}
\newcommand{\BVS}{\mathrm{BV}}
\newcommand{\HIL}{\mathcal{H}}
\newcommand{\LEB}{\mathcal{L}}
\newcommand{\LOC}{\mathrm{loc}}
\newcommand{\ONE}{\mathbb{1}}
\newcommand{\OTE}{E}
\newcommand{\TST}{\textstyle}

\newcommand{\DIAM}{\mathrm{diam}}
\newcommand{\DIST}{\mathrm{dist}}
\newcommand{\HALF}{{\TST\frac{1}{2}}}
\newcommand{\HAUS}{\mathcal{H}}
\newcommand{\MANI}{\mathcal{M}}
\newcommand{\NOTE}{\bar{E}}

\DeclareMathAlphabet{\pazocal}{OMS}{zplm}{m}{n}
\DeclareMathOperator{\DIV}{{\mathrm{div}}}
\DeclareMathOperator{\SPT}{{\mathrm{spt}}}

\DeclareMathOperator*{\ESUP}{{\mathrm{ess\,sup}}}

\DeclarePairedDelimiterX\NN[1]{\lvert}{\rvert}{\kern-2pt\delimsize\lvert #1 \delimsize\rvert\kern-2pt}

\numberwithin{equation}{section}

\title[Mullins-Sekerka in the plane]
      {Sharp Convergence to the Half-Space for Mullins-Sekerka in the Plane }

\author[W.\ Shi]{Wenhui Shi}
\author[M.\ G.\ Westdickenberg]{Maria G.\ Westdickenberg}
\author[M.\ Westdickenberg]{Michael Westdickenberg}

\address{%
  Wenhui Shi,
  Lehrstuhl für Angewandte Analysis,
  RWTH Aachen University,
  Pont\-driesch 14-16,
  D-52062 Aachen,
  Germany}
\email{shi@math1.rwth-aachen.de}

\address{%
  Maria G.\ Westdickenberg,
  Lehrstuhl für Angewandte Analysis,
  RWTH Aachen University,
  Pontdriesch 14-16,
  D-52062 Aachen,
  Germany}
\email{maria@math1.rwth-aachen.de}

\address{%
  Michael Westdickenberg,
  Lehrstuhl für Mathematik (Analysis),
  RWTH Aachen University,
  Templergraben 55,
  D-52062 Aachen,
  Germany}
\email{mwest@instmath.rwth-aachen.de}

\date{\today}


\begin{document}

\begin{abstract} We revisit the HED Method for the Mullins-Sekerka evolution in the plane. We identify a natural notion of distance, intrinsic to the interface itself. Using this distance, the energy, and the dissipation, we develop natural assumptions on the flow and, assuming existence of a solution satisfying these conditions, obtain not just the \emph{algebraic rate} (already previously derived by Chugreeva, Otto, and M.G.\ Westdickenberg) but also the \emph{sharp leading order constant} for the convergence to the flat limiting interface. To carry this out, we introduce a lemma that obtains optimal algebraic decay behavior (rate and prefactor) from given variational information.
\end{abstract}

\maketitle

\setcounter{tocdepth}{1}


\section{Introduction}

The Mullins-Sekerka evolution is a nonlocal, third-order free boundary evolution introduced
by William Wilson Mullins and Robert Floyd Sekerka \cite{MullinsSekerka1964} as
a phenomenological model of phase separation and relaxation in a binary alloy. Mathematically it involves solving the Laplace equation ($\Delta u=0$) on either side of an interface $\Gamma$, where the value of $u$ at the interface is equal to the mean curvature of $\Gamma$, and the interface moves with the normal velocity equal to the jump of the normal derivative of the field $u$ across $\Gamma$, i.e. $V=[\nu\cdot \nabla u]$.
Within the mathematics community, there have been extensive studies, including on the derivation of the model as a sharp interface limit of a phase-field model \cite{Pego89,ABC94},  coarsening behavior (see for instance \cite{BO01, KO02, Alikakos04} and associated references), local and global well-posedness of classical solutions \cite{Chen93,CHY96,EscherSimonett98,EscherMatiocMatioc2024}, existence of weak solutions and weak-strong uniqueness \cite{LS95,R05,CL21,HenselStinson2024, FHLS24}, and long-time dynamics \cite{ChugreevaOttoWestdickenberg2019, JMPS23, OttoSchubertWestdickenberg2024, ADK25, Lukic2025}. This is by no means a comprehensive list; for a broader view we refer to the surveys \cite{NovickCohen2008,Garcke13} and the associated cited and citing references.

The Mullins-Sekerka evolution can be viewed as a gradient flow of the interfacial area, cf. e.g. \cite{Le2008} and Section \ref{S:GFS} below. This perspective is crucial in constructing weak solutions \cite{LS95,CL21}, as well as studying the convergence rate of the flow near equilibrium
\cite{ChugreevaOttoWestdickenberg2019, OttoSchubertWestdickenberg2024, Lukic2025}.   In the case of a compact limit interface and suitable initial conditions,
convergence to equilibrium is
exponential \cite{Chen93,CHY96,EscherSimonett98,Lukic2025}.
In the noncompact case, exponential convergence fails and recently optimal \emph{algebraic convergence rates} have attracted attention. We will
 summarize existing results and our contribution below after first reviewing
the convex case.

To briefly put our work with the HED method into a broader context, we remark on two higher level connections.
First, convergence rates for gradient flow solutions are often derived using the \L ojasiewicz-Simon inequality relating the energy (gap) and dissipation; see \cite{lojasiewicz1965,Simon83,CHJ09,CM15} and others. In the HED method, we admit the choice of distance as an additional degree of freedom, leading ultimately to a system of differential inequalities involving energy, dissipation, and distance, from which convergence rates follow.
Second, the leveraging of negative norms to derive convergence rates  features in
the work of Lemarie-Rieusset \cite{L} (see also the cited references), Guo and Tice \cite{GT},
Guo and Wang \cite{GW},  Sohinger and Strain \cite{SS2014}, and others.

\subsection{The HED Method and Our Convergence Result}
\label{S:TDR}

Brézis \cite[Theorem 3.7]{Brezis1973} made the following observation for convex energies (not assumed to be strictly convex).
\begin{proposition}
\label{prop:brezis}
Suppose that the energy $E$, defined on some Hilbert space with norm $|\cdot|$, is convex, twice
continuously differentiable, and bounded below with minimum energy $0$. Suppose there exists a global minimizer $u_*$ of $E$, so that $E(u_*)=0$.
Consider the gradient flow
\[
  \dot{u}=-\nabla E(u),
  \qquad
  u(0)=u_0
\]
and the related squared distance and dissipation
\[
  H\coloneq|u-u_*|^2,
  \qquad
  D\coloneq|\nabla E(u)|^2.
\]
Denote the initial values $H_0:=H(u_0),\;E_0:=E(u_0),\;D_0:=D(u_0)$.

The gradient flow satisfies
\begin{equation}
\label{Br_fundamental}
  \dot{E}=-D,
  \qquad
  \frac{1}{2}\dot{H}+E\LS 0,
  \qquad \dot{D}\LS 0,
  \qquad
  E\LS \sqrt{HD},
\end{equation}
and as a consequence of these relations, it follows that
\begin{equation}
\label{Br_output}
  H\LS H_0,
  \qquad
  E\LS\min\{E_0,\frac{H_0}{4t}\},
  \qquad
  D\LS \min\{D_0,\frac{E_0}{t},\frac{H_0}{t^2}\}.
\end{equation}
\end{proposition}

\begin{proof}
The first relation in \eqref{Br_fundamental} is classical. The second and fourth
estimates follow from convexity in the form $0=E(u_*)\GS E(u)-\langle\nabla E(u),u-u_*
\rangle$ together with
\[
  \langle\nabla E(u),u-u_* \rangle\LS |u-u_*|\, |\nabla E(u)|.
\]
The third bound in \eqref{Br_fundamental} follows from convexity in the form that
the Hessian $D^2 E(u)$ is nonegative-definite:
\[
  \dot{D} = 2 \langle \nabla E(u), D^2 E(u) \dot{u} \rangle
    = - \langle \nabla E(u), D^2 E(u) \nabla E(u) \rangle \LS 0.
\]
The facts $H\LS H_0$, $E\LS E_0$, and $D\LS D_0$ are immediate consequences of
\eqref{Br_fundamental}.

To deduce the remaining bounds in \eqref{Br_output}, one introduces the auxiliary function
\[
  F(t)\coloneq\frac{1}{2}H + t E+\frac{1}{2}t^2 D.
\]
The bounds in \eqref{Br_fundamental} deliver exactly the right information to conclude that $F$ is nonincreasing in time, from which $D\LS H_0/t^2$
follows. While one also obtains $tE\LS H_0/2$, the constant $1/4$ in the
algebraic decay of the energy leverages the interpolation inequality
$E\LS\sqrt{HD}$ in the form
\[
  F_0\GS F(t)= t E +\Big(\frac{1}{2}H + \frac{1}{2}t^2 D \Big) \GS tE +tE= 2tE.
\]
Finally \emph{because the mapping $t\mapsto E(u(t))$ is convex}, one observes
\[
  E_0\GS E+(0-t)\dot{E}=E+tD\GS tD,
\]
whence $D \LS E_0/t$.
\end{proof}
\begin{remark}
See also \cite{OV01} for the Riemannian version.
\end{remark}

Notice that $H$ is not shown to vanish, only to remain bounded in terms of its
initial value. Relaxation is reflected in the convergence of $E$ (and $D$) to
zero. This is necessary: Without strict convexity, the minimizer need not be
unique and $H$ need not converge to zero, as we illustrate in the following toy example.

\begin{example}
\label{EX:TOY}
As the simplest possible example, consider $E(x,y)=\frac{1}{2}x^2$ in $\R^2$
with the noncompact continuum of minimizers $\{(0,y)\colon y\in\R\}$. We can
pick any minimum $u_*=(0,y_*)$ and any initial data $(x_0,y_0)$. The gradient
flow solution to $E$ will relax to $(0,y_0)$ and the energy will converge to
zero, but for the squared distance, we obtain only $x(t)^2+(y(t)-y_*)^2\LS
x_0^2+(y_0-y_*)^2$. We will return to this point for the Mullins-Sekerka
evolution below.
\end{example}

One of the insights of the argument of Brézis is that convexity is
\emph{needed/used only in the form of the differential identities/estimates and
interpolation inequality} from \eqref{Br_fundamental}. In other words, without
assuming convexity of the energy, if one has \eqref{Br_fundamental}, then one
obtains~\eqref{Br_output}.

In fact, if one has even \emph{stronger} bounds on
$\dot{H}$ and $E$ (with or without convexity of the energy), then an ODE argument can be used to derive decay with improved
constants. This is what we highlight in this paper and in particular in Lemma~\ref{L:DECAY} below.
Before getting to the constants, we review the first Brézis-inspired extensions to nonconvex problems.

The so-called HED method introduced in \cite{OttoWestdickenberg2014} showed how
the insight of Brézis can be used to study \emph{nonconvex problems}
satisfying relations of the form \eqref{Br_fundamental} \emph{up to constants and error
terms}, in the following sense: Suppose that
\begin{equation}
  \dot{E}=-D,
  \qquad
  \dot{H} \lesssim \textrm{err},
  \qquad
  \dot{D}\lesssim \textrm{err},
  \qquad E\lesssim \sqrt{HD}.
\label{input3}
\end{equation}
If the error terms are small enough, then
\begin{equation}
  H\lesssim H_0,
  \qquad
  E\lesssim \frac{H_0}{t},
  \qquad
  D\lesssim \frac{H_0}{t^2}.
\label{output2}
\end{equation}
(Here and in \eqref{rate} below we use the notation $A\lesssim B$ to denote that there exists a universal constant $C\in (0,\infty)$ such that $A\LS C\,B$.) Hence the HED method can be viewed as a
quantification of nonconvexity in the sense: If a gradient flow satisfies \eqref{input3} and the error terms are small
enough, \emph{then the problem yields the same $t^{-1}$ and $t^{-2}$ decay rates
of energy and dissipation as for a convex problem}.

The HED method was first carried out in \cite{ChugreevaOttoWestdickenberg2019}
for the Mullins-Sekerka evolution in ambient space $\R^2$, obtaining
\eqref{output2} for the \emph{extrinsic} squared distance
\[
  \tilde{H}=\int_{\R^2}|\nabla \varphi|^2\,d\textbf{x},
  \quad\text{for}\quad
  -\Delta \varphi=\textbf{1}_{\Omega_+(t)}-\textbf{1}_{\{z>0\}}
\]
and under the assumption that the initial interface is given as the graph of a
function $h_0(\cdot)$ over $\R\times \{z=0\}$ with $\|\partial_xh_0\|_{L^\infty}\LS 1$,
$\tilde{H}_0$, $E_0$, and $D_0$ finite, and the dimensionless parameter $E_0^2 D_0$ small enough. Notice
that even though there is nonuniqueness as in Example~\ref{EX:TOY} (every
interface $\R\times \{z=c\}$ corresponds to an equilibrium of finite energy), a
unique equilibrium is fixed by the choice of the squared distance $\tilde{H}$:
Every equilibrium corresponding to $\R\times \{z=c\}$ for $c\neq 0$ is
infinitely far away, and the interface $\R\times \{z=0\}$ is selected as the
long-time limit by the assumption $H_0<\infty$. (The situation is different when
the interface is compact and one has to contend with a whole continuum of
minimizers; see \cite{Lukic2025}. For a related situation for the Cahn-Hilliard
equation in $n=1$ we refer to \cite{BiesenbachSchubertWestdickenberg2022}.)

More recently, \cite[Proposition~2.3]{OttoSchubertWestdickenberg2024} exchanged
the extrinsic distance from \cite{ChugreevaOttoWestdickenberg2019} for $V$, an
$L^1$-measure of distance from the interface, and studied convergence in the
full space. For $n=2$ and $n=3$ and within the graph setting, it was shown that
\begin{equation}
  E\lesssim\frac{V_0^2+E_0^4}{t}
  \quad\text{(in $n=2$)},
  \qquad
  E\lesssim \frac{V_0^2+E_0^3}{t^{4/3}}
  \quad\text{(in $n=3$)}.
\label{rate}
\end{equation}
The first estimate is the same $t^{-1}$ decay as in \eqref{output2}. We remark,
however, that while the extrinsic distance $\tilde{H}$ requires the weak
neutrality assumption
\begin{align*}
  \int_\R h_0(x)\,dx=0,
\end{align*}
(a drawback that still applies to our result) the $L^1$-method is free of this restriction.

More significantly, \cite{OttoSchubertWestdickenberg2024} captures the fact that
in $\R^3$, even initial interfaces that are not graphs
are \emph{drawn into the graph regime} (and satisfy the latter estimate
in~\eqref{rate}). This fact has its roots in an elementary argument showing that
the dimensionless (in $n=3$) parameter $ED^2$ becomes small, which is enough to
force the interface into being a graph over $\R^2\times \{z=0\}$. As explained in
\cite{OttoSchubertWestdickenberg2024}, smallness of the dimensionless (in $n=2$)
parameter $E^2D$ is \textbf{not enough in $n=2$ to imply that the interface is a
graph}. The counterexample of a flat interface separating positive and negative
phases -- with a small spherical island of negative phase embedded within the
positive phase and very far away from $\R\times\{z=0\}$ -- shows that one can have
$E^2D$ as small as one likes but still have no graph structure. The trouble in
$\R^2$ is not completely surprising: The critical Sobolev
embedding and logarithmic potential in $n=2$ pose  challenges. The question of
how and under which conditions the Mullins-Sekerka evolution in the plane
actually enters the graph setting  is  the subject of work in progress by Otto,
Salguero, Schubert, and M.G.\ Westdickenberg.

\medskip

In this paper, we have a different focus. We revisit the Mullins-Sekerka
evolution in the plane with the following objectives:
\begin{itemize}
\item Identify an intrinsic way to pose the problem directly on the interface,
the natural function spaces, and the embeddings for which one would expect to be
able to prove an existence result.
\item Using the intrinsic squared distance derived from the previous point,
identify the \textbf{sharpest constant $C_1$} for which we can show
(approximately)
\begin{align}
  E\LS C_1\, \frac{H_0}{t}.\label{sharpest}
\end{align}
\end{itemize}
In order to obtain the best possible constant in the relaxation rate, we
introduce the following lemma, which deduces the best possible relaxation result
(cf.~\eqref{E:FINAL}) as a function of $\gamma$ from variational information in the normal form~\eqref{E:ASSNS}.
We have reduced to normal form in order to present the result with only one free parameter; we explain the rescaling in Remark~\ref{rem:normal} below.

\begin{lemma}[Sharp Algebraic Decay Lemma]
\label{L:DECAY}
Consider positive functions $E$, $D$, and $H$ defined on $[0,\infty)$. Suppose
that $H$ and $E$ are absolutely continuous (hence differentiable almost
everywhere) and $D$ is continuous, and assume that the following differential
inequalities hold a.e.:
\begin{equation}
    \dot{E} \LS -D,
    \qquad
    \dot{H} + E \LS R,
    \qquad
    E \LS \sqrt{\gamma HD},
    \label{E:ASSNS}
\end{equation}
where $\gamma>0$ is a constant. Assume further that
\begin{equation}
    0\LS R \LS \frac{\sigma HD}{E},
    \qquad
    H_\mathrm{max} \coloneq \max_{t\GS 0} H(t) < \infty,
\label{E:APRIORI}
\end{equation}
with $\sigma \GS 0$ a constant. Let $\mu \coloneq 1/\gamma$ and $\nu \coloneq
\mu-1 > -1$, and define
\begin{equation}
    K_\gamma \coloneq \gamma^\frac{1}{1-\gamma}
    \quad\text{for $\gamma \neq 1$,}\qquad
    K_1 \coloneq \lim_{\gamma\to 1} K_\gamma = \frac{1}{e}.
\label{E:KGA}
\end{equation}
Notice that $K_\gamma \in (0,1)$ is monotone increasing for $\gamma>0$ and
$K_\gamma < \gamma$. Then
\begin{equation}
    E(t) \LS \frac{H_\mathrm{max}}{t} \begin{cases}
        \Big( K_\gamma \left( \frac{\gamma}{\sigma} \right)^\sigma
                \Big)^\frac{1}{1+\sigma} (1+\sigma)
            & \text{if $\sigma>0$}
        \\
        \hphantom{\Big(} K_\gamma & \text{if $\sigma=0$}
    \end{cases}
    \quad\text{for $t>0$.}
\label{E:FINAL}
\end{equation}
The constant $K_\gamma$ in \eqref{E:FINAL} cannot be improved.
\end{lemma}
We prove this lemma in Section~\ref{S:ODE} below after explaining its application to the Mullins-Sekerka flow.

\begin{remark}[Rescaling]\label{rem:normal}
Starting from the variational inequalities
\begin{equation}
  \dot{E}=-D,
  \qquad
  \frac{1}{2}\dot{H}+C^2 \, E \LS 0,
  \qquad
  \dot{D}\LS 0,
  \qquad
  E\LS \sqrt{\gamma HD},\label{better}
\end{equation}
we rescale $\tilde{H} \coloneq H/(2C^2)$ and
$\tilde{\gamma} \coloneq 2C^2\gamma$ to obtain the normal form
\[
  \dot{E} = -D,
  \qquad
  \dot{\tilde{H}}+ E \LS 0,
  \qquad
  \dot{D} \LS 0,
  \qquad
  E \LS \sqrt{\tilde{\gamma} \tilde{H} D}.
\]

Lemma~\ref{L:DECAY} with $\tilde{\gamma}$ gives
\begin{equation}
    E(t) \LS \frac{K_{\tilde{\gamma}} \tilde{H}_\mathrm{max}}{t}
    \quad\Longleftrightarrow\quad
    E(t) \LS \frac{K_{\tilde{\gamma}} H_\mathrm{max}}{2t}
    \quad\text{for $t>0$.}
\label{E:EDEC}
\end{equation}
In particular for the estimates from~\eqref{Br_fundamental} with $C=\gamma=1$,
Lemma~\ref{L:DECAY} provides the optimal constant $K_2/2 = 2^\frac{1}{1-2}/2 =
1/4$ as in~\eqref{Br_output}.
\end{remark}
\begin{example}
It is instructive to consider the simple examples of scalar gradient flows: For
$x\in\R$, $m\in\N$, and $L\in (0,\infty)$ consider the energies
\begin{align*}
E_m(x)=|x|^{m} \qquad\text{or} \qquad E^L(x)=\frac{x^4}{4}+\frac{x^2}{2L}.
\end{align*}
Let $x_0\neq 0$ denote the initial condition and $H(x)=x^2$,
$D(x(t))=-\frac{d}{dt}E(x(t))$.

Starting with $E_1$ and $x_0=1$, we read off the explicit solution $x(t) = 1-t$
for all $t \in [0,1]$, and $x(t) = 0$ afterwards. Then $H_\mathrm{max} = H_0 =
1$, and $E$, $D$, and $H$ satisfy \eqref{Br_fundamental}. The explicit energy,
distance, and dissipation of the solution are
\[
    E(x(t))= (1-t)\ONE_{[0,1]}(t),
    \quad
    H(x(t))= |1-t|^2 \ONE_{[0,1]}(t),
    \quad
    D(x(t))  = \ONE_{[0,1]}(t).
\]
Finding the optimal constant in \eqref{sharpest} then amounts to a fitting
problem: find $C_1>0$ as small as possible such that the graph of $C_1/t$
touches the graph of $t \mapsto E(x(t))$ from above in exactly one location.
This translates into the quadratic equation
\[
    \frac{C_1}{t} \overset{!}{=} E(x(t)) = 1-t
    \quad\Longleftrightarrow\quad
    t = \frac{1}{2} \pm \sqrt{\left(\frac{1}{2}\right)^2-C_1},
\]
which has a unique solution precisely for $C_1 = 1/4$.

For $E_2(x)$, the gradient flow solution with $x_0=1$ is $x(t) = e^{-2t}$ for
$t\GS 0$. The energy of the solution is $E(x(t)) = e^{-4t}$, and the touching
condition is
\[
    \frac{C_1}{t} \overset{!}{=} E(x(t)) = e^{-4t}
    \quad\Longleftrightarrow\quad
    4t e^{-4t} = 4 C_1,
\]
which has a unique solution $t = 1/4$ precisely for $C_1=1/4e$. This is
consistent with Lemma~\ref{L:DECAY}: The triple $E$, $D$, and $H$ satisfies
(with equalities)
\[
    \dot{E} \LS -D,
    \qquad
    \frac{1}{2} \dot{H} + 2E \LS 0,
    \qquad
    E^2 \LS \frac{1}{4} HD.
\]
We rescale $\tilde{H} \coloneq H/4$. Then \eqref{E:FINAL},\eqref{E:KGA} with
$\tilde{\gamma} = 1$ gives $E(x(t)) \LS 1/(4et)$.

Now consider $E_m$ for $m\gg 1$, or $E^L$ for $L\gg 1$. The energy is convex and
Pro\-position~\ref{prop:brezis} yields \eqref{Br_output}, but the actual decay
occurs with a smaller prefactor, which one would like to capture. The reason for
this faster decay is the additional weak but nonvanishing convexity, which is
encoded in the ``enhanced variational inequalities'' of the form~\eqref{better}
that the flows fulfill.

Bringing the equations into normal form and applying Lemma~\ref{L:DECAY}
and $H_0=x_0^2$ delivers the algebraic rates
\begin{equation}
    E_m\big( x(t) \big) \LS \frac{x_0^2}{m^2t}\left(\frac{2}{m}\right)^{\frac{2}{m-2}}
    \qquad\text{and}\qquad
    E^L\big( x(t) \big) \LS \frac{x_0^2}{4e\,t}.
\label{get1}
\end{equation}
For $E_m$, one can expand the explicit solution for $m$ large and verify
\begin{align*}
  E_m\big( x(t) \big) &=\left( \frac{|x_0|^{m-2}}{1+|x_0|^{m-2}m(m-2)t}\right)^{m/(m-2)}
  \approx \frac{1}{m^2 t}\\
  &\quad\text{for } \frac{1}{m^2|x_0|^{m-2}}\ll t\ll \exp(m),
\end{align*}
so that the first estimate in \eqref{get1} gives a good estimate on a
large time interval.

For the energy $E^L$, one has the classical bound $E_L(x(t))\leq E_0$. If one also observes the
energy--energy-dissipation bound $E^L\leq \frac{L}{2}\,D$, then one obtains the $L$-dependent exponential estimate
\begin{align*}
  E^L\big( x(t) \big)\leq E_0\exp(-2t/L)\approx \frac{x_0^4}{4}\exp(-2t/L)\qquad\text{for }\;x_0\gg \frac{1}{L}
\end{align*}
because of the quartic term in the energy.
In this case the
estimate on $E^L$ from~\eqref{get1} is better for the large intermediate time
interval
\begin{align*}
  \frac{1}{x_0^2}\ll t \ll L.
\end{align*}
\end{example}

We will apply Lemma~\ref{L:DECAY} to obtain our convergence result. As in~\cite{ChugreevaOttoWestdickenberg2019}, we will make heavy use of the
scale invariant quantity
\begin{equation}
\label{eq:E2D}
  \epsilon\coloneq\epsilon(t)\coloneq(E^2 D)^{\frac{1}{6}},
  \qquad
  \epsilon_0\coloneq\epsilon(0).
\end{equation}
In terms of $\epsilon$, we will derive an error term (see \eqref{eq:dtH}) of the form
\[
    R \LS
   C
        \epsilon^2 \frac{HD}{E}.
\]
In particular, the constant $\sigma \coloneq C \epsilon^2$ will be small.

\medskip

Our convergence result is as follows.

\begin{theorem}[Convergence rates]
\label{T:OPTIMAL}
There exist universal constants $\epsilon'>0$ and $C>0$ with the following
property. Consider any global solution $\{\Gamma_t\}_t$ of the Mullins-Sekerka
evolution as specified in Definition~\ref{D:SOLUTION}. Let $H(t)$,
$E(t),\,D(t),\,\text{and }\epsilon(t)$ represent the squared distance, energy, dissipation,
and scale invariant quantity as given in Section~\ref{S:GFS} and \eqref{eq:E2D},
and set $H_0\coloneq H(0)$, $E_0\coloneq E(0),$ $D_0\coloneq D(0),$ and
$\epsilon_0\coloneq \epsilon(0)$. Assume that $\Gamma_0$ is an admissible
interface as in Definition~\ref{D:AI} and that
\[
  H_0, E_0, D_0<\infty,
  \qquad
  \epsilon_0\LS \epsilon'.
\]
Then for all $t>0$ \textbf{the map $t\mapsto E(u(t))$ is convex} and the
solution satisfies
\begin{align}
  \epsilon(t)
    & \LS \epsilon_0,
\nonumber
\\
  H(t)
    &\LS H_0(1+C\epsilon_0^{2}),
\nonumber
\\
  E(t)
    & \LS \min\left\{E_0,\left(\frac{1}{4e}+C(\epsilon_0\log(\epsilon_0))^{2}\right)\frac{H_0}{t}\right\},
\label{eless}
\\
  \text{and}\quad
  D(t)
    & \LS \min\left\{ D_0, \frac{E_0}{t},
      \left(\frac{1}{e}+C(\epsilon_0\log(\epsilon_0))^{2}\right)\frac{H_0}{t^2}\right\}.
\label{dless}
\end{align}
\end{theorem}

The proof of the main theorem is given in Subsection~\ref{S:OPTIMAL} below.

\begin{remark}[Optimal constant]
In order to judge the sharpness of the results in \eqref{eless} and
\eqref{dless}, it is instructive to compare to the linearization. The
geometrically linearized Mullins-Sekerka evolution in $n=2$ consists of
\begin{alignat*}{2}
    -\Delta f & = 0 \quad & \text{in $\R\times\set{z>0}$,}
\\
    f & = \partial_{xx} h \quad & \text{on $\R\times\set{z=0}$,}
\\
    h_t & = -2f_z \quad & \text{on $\R\times\set{z=0}$,}
\\
    h(0) & = h_0 \quad & \text{on $\R\times\set{z=0}$,}
\end{alignat*}
which can be reduced using the Fourier transform to studying:
\[
  \tilde{E}\coloneq\frac{1}{2}\int_\R|k|^2 |\hat{h}|^2\,dk,
  \qquad
  \tilde{D}\coloneqq2\int_\R |k|^5 |\hat{h}|^2 \,dk,
  \qquad
  \tilde{H}\coloneq \frac{1}{2}\int_\R \frac{1}{|k|}|\hat h|^2\,dk
\]
together with the dynamic equation
\begin{equation}
  \partial_t \hat{h}=-2|k|^3\,\hat{h}.
\label{dyn}
\end{equation}

Using the solution of~\eqref{dyn}, one obtains the convergence
\begin{align*}
  \tilde{E}\LS \frac{1}{4 e}\,\frac{\tilde{H}_0}{t},
\end{align*}
which our result \eqref{eless} captures to leading order for the nonlinear problem.
The usual caveat applies,
however, that $H_0<\infty$ imposes a weak neutrality condition, and the
HED method does not capture the faster decay rate that neutrality is
expected to imply.
\end{remark}

The convergence results we  obtain in Theorem~\ref{T:OPTIMAL} raise the question
of \textbf{whether the Mullins-Sekerka energy is in fact convex}, at least close
to equilibrium. This was shown \emph{for a thin film equation}
in~\cite{Esselborn2016}. This is, however, not the case here, as the following
proposition shows. The lacking convexity is tied to the unboundedness of the
problem and insufficient decay of the solution near infinity.

\begin{proposition}
\label{prop:nonconvex}
There exists no $\epsilon>0$ such that
the
Hessian of the Mullins-Sekerka evolution is positive semidefinite on an
$\epsilon$-neighborhood of $\R\times \{z=0\}$, as measured by the squared
distance function.
\end{proposition}

\begin{notation}[Dimension, normal, curvature]
To avoid confusion with our previous work, we wish to emphasize that here we
will use $n$ to denote the ambient space dimension, whereas in
\cite{OttoSchubertWestdickenberg2024} $d$ was used to denote the dimension of
the interface, i.e.,
\[
  n\coloneq d+1.
\]
In addition, we will introduce $\nu$ as the \textbf{outward unit normal to the
positive phase}, rather than the inward normal as in
\cite{OttoSchubertWestdickenberg2024}. We adopt the sign convention that the interface has \textbf{positive curvature}
near a point if it is bent \emph{in the direction of the normal} there.
\end{notation}

\subsection{Model formulation}
Since the Mullins-Sekerka equation models the dynamics of material phases, it is
most often considered as describing the evolution of characteristic functions of
time-dependent sets. We sketch the solution concept introduced in
\cite{HenselStinson2024}. Notice that \cite{HenselStinson2024} considers the
Mullins-Sekerka equation in bounded domains (studying, in particular, the
behavior of solutions at the boundary of the domain), whereas we are interested
in the evolution on $\R^n$, which poses additional difficulties. Our
presentation is not a rigorous translation of the definition in
\cite{HenselStinson2024} but is only meant to convey the principal ideas of the
construction. A solution of the Mullins-Sekerka equation is a function $\chi \in
\BVS([0,T]\times\R^n; \{0,1\})$ satisfying
\begin{equation}
\begin{aligned}
  & \int_{\R^n} \varphi(T,x) \chi(T,x) \,dx
    - \int_{\R^n} \varphi(0,x) \chi(0,x) \,dx
\\
  & \qquad
    = \int_0^T \int_{\R^n} \partial_t\varphi(t,x) \chi(t,x) \,dx \,dt
    - \int_0^T \int_{\R^n} \nabla\varphi(t,x) \cdot \nabla u(t,x) \,dx \,dt
\end{aligned}
\label{E:TRP}
\end{equation}
for all $\varphi \in C^1_c([0,t]\times \R^n)$, for some potential function $u
\in L^2([0,T], H^1(\R^n))$. Then the two time-dependent material phases are
given for a.e.\ $t\in[0,T]$ by
\[
  \Omega_t^+ \coloneq \{x\in \R^n \colon \chi(t,x) = 1 \}
  \quad\text{and}\quad
  \Omega_t^- \coloneq \{x\in \R^n \colon \chi(t,x) = 0 \}.
\]
Equation \eqref{E:TRP} is the distributional formulation of $\partial_t\chi =
\Delta u$. Hence $u(t,\cdot)$ is harmonic outside the evolving interface
$\partial\Omega_t^+$, which moves in the normal direction with speed given as the
difference of the normal derivatives of $u(t,\cdot)$ on either side of
$\partial\Omega_t^+$. Defining the surface measure $\sigma_t \coloneq
\HAUS^{n-1} \lfloor \partial\Omega_t^+$, one requires that
\begin{equation}
  \int_{\R^n} \Big( \ONE - \nu(x)\otimes\nu(x) \Big) \colon D\zeta(x) \,d\sigma_t(x)
    = \int_{\R^n} \chi(t,x) \, \nabla\cdot\Big( w(t,x) \zeta(x) \Big) \,dx
\label{E:GIB}
\end{equation}
for all $\zeta \in C^1_c(\R^n; \R^n)$ and a.e.\ $t\in(0,T)$, with $\nu$ the
outward unit normal vector to $\Omega_t^+$, $\ONE$ the $(n\times n)$-unit
matrix, and another potential function $w \in L^2([0,T], H^1(\R^n))$. As the
left-hand side of \eqref{E:GIB} is the first variation of the area of
$\partial\Omega_t^+$ in the direction of $\zeta$ (see Section~\ref{S:MST} for
details), equation \eqref{E:GIB} expresses the fact that the mean curvature of
the interface is given as the trace of the potential $w$. The expectation is
that the two potential functions $u$ and $w$ coincide. As explained in
\cite[Section~2.3]{HenselStinson2024}, however, there is a subtle issue
connected to the question of whether the distributional trace in \eqref{E:GIB} and
the pointwise a.e.\ trace in \eqref{E:TRP} coincide. This depends on the
regularity of the interface. The two potentials $u,w$ are linked for a.e.\ $0
\LS t_1 < t_2$ via an energy dissipation inequality (EDI)
\begin{equation}
  E_{t_2} + \int_{t_1}^{t_2} \int_{\R^n}
      \Big( \HALF |\nabla u(t,x)|^2 + \HALF |\nabla w(t,x)|^2 \Big) \,dx \,dt
    \LS E_{t_1}
\label{E:EDI}
\end{equation}
for a.e.\ $t \in[0,T]$, where $E_t$ is the energy of the interface
$\partial\Omega_t^+$, defined as its area. It is well-known that an EDI such as
\eqref{E:EDI} can be used to characterize gradient flows on metric spaces; see
\cites{DeGiorgiMarinoTosques1980, AmbrosioGigliSavare2008}. The Dirichlet
energies of the potential functions $u$ and $w$ in \eqref{E:EDI} represent the
metric derivative of the solution curve and the slope of the area functional,
respectively. We refer the reader to \cite{HenselStinson2024} for details.

Instead of focusing on characteristic functions of time-dependent subsets of
$\R^n$, in this paper we wish to work intrinsically with the evolving interface
itself. It is instructive to compare the Mullins-Sekerka evolution to mean
curvature flow. Let us start with some heuristics. Suppose we have a family of
smooth hypersurfaces $\{ \Gamma_t \}_t$ in $\R^n$ without boundaries. We always
identify the interfaces with their surface measures $\sigma_t \coloneq
\HAUS^{n-1} \lfloor \Gamma_t$. If $V(t,\cdot)$ and $\nu(t,\cdot)$ denote the
normal speed and the unit normal vector along $\Gamma_t$, respectively, then one
can show that a.e.\ in time
\begin{equation}
  \frac{d}{dt} \int_{\R^n} \varphi(t,x) \,d\sigma_t(x)
    = \int_{\R^n} \Big( \partial_t\varphi
      + \BV\cdot\nabla \varphi + \DIV_\tau(\varphi \BV) \Big) \,d\sigma_t(x)
\label{E:WEIGHTED}
\end{equation}
for all $\varphi \in C^1_c([0,\infty) \times \R^n)$, for vector field $\BV =
V\nu$ and tangential divergence $\DIV_\tau$. We refer the reader to
\cite[Section~2.1]{Tonegawa2019} for further discussion of \eqref{E:WEIGHTED}.
If $\Gamma_t$ admits a generalized mean curvature $\kappa(t,\cdot) \in
L^p_\LOC(\Gamma_t, \sigma_t)$ with $p\in[1,\infty]$, then
\begin{equation}
  \frac{d}{dt} \int_{\R^n} \varphi(t,x) \,d\sigma_t(x)
    = \int_{\R^n} \Big( \partial_t\varphi
      + \BV\cdot (\nabla \varphi - \varphi h) \Big) \,d\sigma_t(x)
\label{E:CONT}
\end{equation}
with curvature vector $h = \kappa \nu$. The zeroth-order term in \eqref{E:CONT}
is needed so that the evolution of the surface measure preserves the form
$\sigma_t = \HAUS^{n-1} \lfloor \Gamma_t$.

Equation \eqref{E:CONT} is well-defined if for example
\begin{equation}
  \kappa(t,\cdot), V(t,\cdot) \in L^2_\LOC(\Gamma_t, \sigma_t)
  \quad\text{for a.e.\ $t\GS 0$.}
\label{E:BONDS}
\end{equation}
For mean curvature flow, we have $V = \kappa$. For Mullins-Sekerka, however,
curvature and normal speed are related by the \emph{Dirichlet-to-Neumann map
$N$}, so that $V = N(\kappa)$. Since $\kappa(t,\cdot)$ is the trace of the
potential $u(t,\cdot) \in \dot{H}^1(\R^n)$ on the interface $\Gamma_t$, it is
naturally contained in the homogeneous fractional Sobolev space
$\dot{H}^{1/2}(\Gamma_t,\sigma_t)$. We refer the reader to Section~\ref{S:GHA}
for definitions. By Sobolev embedding, one obtains that $\kappa \in
L^4(\Gamma_t, \sigma_t)$ for $n=3$, $\kappa \in \BMO(\Gamma_t, \sigma_t)$ for
$n=2$. Therefore the bound on the curvature in \eqref{E:BONDS} is naturally
satisfied. The bound on the normal speed, however, is more subtle. Since the
Dirichlet-to-Neumann map behaves like a differentiation, one only obtains
$V(t,\cdot) \in H^{-1/2}(\Gamma_t, \sigma_t)$, the topological dual of
$\dot{H}^{1/2}(\Gamma_t, \sigma_t)$.

In Section~\ref{S:GFS} we will argue that in the particular case of
\textbf{sufficiently flat}, one-dimensional interfaces in $\R^2$, one has
additional regularity, namely
\begin{itemize}
\item curvature $\kappa(t,\cdot) \in L^2(\Gamma_t, \sigma_t)$ (not only
$L^2_\LOC$),
\item normal speed $V(t,\cdot) \in L^2(\Gamma_t, \sigma_t)$, and
\item the Dirichlet-to-Neumann map $V = N(\kappa)$ is well-defined.
\end{itemize}
Consequently, one can consider a notion of solutions for the Mullins-Sekerka
equation that is stronger than the one in \cite{HenselStinson2024} and closer to
Brakke solutions for mean curvature flow. The higher integrability for
$\kappa(t,\cdot)$ follows from the fact that the curvature $\kappa(t,\cdot)$ is
contained in the homogeneous negative Sobolev space $\dot{H}^{-1}(\Gamma_t,
\sigma_t)$, and we can use an interpolation argument. Our claim about the
normal speed $V(t,\cdot)$ follows from the observation that, once the interface
$\Gamma_t$ is sufficiently flat, the time derivative of the dissipation becomes
negative (therefore finite) for a.e.\ $t$. The time derivative of the
dissipation formally equals the Hessian of the energy evaluated in the direction
of the gradient flow. An explicit formula for the Hessian then shows that
$V(t,\cdot) \in \dot{H}^1(\Gamma_t, \sigma_t)$, so again the result follows by
interpolation. Some of these arguments already appear in
\cite{ChugreevaOttoWestdickenberg2019}. Here we provide novel proofs that are
intrinsic to the interface. Techniques from minimal surface theory, which
already made an appearance in \cite{OttoSchubertWestdickenberg2024}, also play a
role in our approach. With higher integrability for curvature and normal speed,
the continuity equation \eqref{E:CONT} can be defined. Such a formulation is in
line with the setting of Wasserstein gradient flows.

With regards to the Dirichlet-to-Neumann map, all questions concerning existence
and uniqueness of harmonic extensions will be considered from the point of view
of \emph{single layer potentials}, which are singular integral operators acting
on function spaces defined on the interface. There now exists a rather complete
theory for such operators, spelling out precisely the necessary geometric
properties of the interface. Flatness of the interface is quantified in terms of
oscillations of the unit normal $\nu$, measured in the $\BMO$-norm. We refer the
reader to Section~\ref{S:GHA} for details.

Monotonicity of the dissipation in time is also precisely the condition enabling
us in Section~\ref{S:MSE} to prove convergence rates for solutions of
the Mullins-Sekerka equation towards the flat interface, implementing the ideas
outlined in Section~\ref{S:TDR}. Although we will mainly focus on the regime in which
the interface is already a sufficiently flat \emph{graph}, some of our
arguments go beyond the graph setting.

\begin{notation}[Constants, equivalence]
We will often use $C$ (and sometimes $C'$) to denote a universal constant whose
particular value is not of interest and may change from line to line. We will
occasionally use the notation $A \sim B$ meaning that there exist constants $C,
C' \in (0,\infty)$ such that $A\LS CB$ and $B\LS C' A$.
\end{notation}

\section{Proof of the Sharp Algebraic Decay Lemma~\ref{L:DECAY}}\label{S:ODE}
\begin{proof}
Once boundedness of $H$ has been established, the first and third inequalities
in \eqref{E:ASSNS} already imply a time decay of the energy in the form
\begin{equation}
    E(t) \LS \frac{\gamma H_\mathrm{max}}{t}
    \quad\text{for $t>0$.}
\label{E:APR2}
\end{equation}
Indeed we have the differential inequality
\[
    \frac{1}{\gamma} \LS \frac{HD}{E^2}
        \LS H_\mathrm{max} \frac{-\dot{E}}{E^2}
        = H_\mathrm{max} \frac{d}{dt}\left( \frac{1}{E} \right),
\]
which upon integration over $[\tau,t]$ gives
\[
    \frac{t-\tau}{\gamma}
        \LS H_\mathrm{max} \left( \frac{1}{E(t)} - \frac{1}{E(\tau)} \right)
        \LS \frac{H_\mathrm{max}}{E(t)}.
\]
Sending $\tau\to 0$ then yields \eqref{E:APR2}. The point of the lemma is
therefore to improve the constant in \eqref{E:APR2} from $\gamma$ to the smaller
$K_\gamma$. We proceed in two steps.

\medskip

\textbf{Step~1.} For fixed $\tau>0$ to be determined later, we define
\[
    G(t) \coloneq H(\tau) - \int_\tau^t E(s) \,ds + \int_\tau^t R(s) \,ds
    \quad\text{for $t \GS \tau$,}
\]
so that
\begin{equation}
    G(\tau) = H(\tau),
    \qquad
    \dot{G} = -E + R.\label{Gother}
\end{equation}
Using the second inequality in \eqref{E:ASSNS}, we find in addition that
\begin{equation}
    G \GS H.
\label{E:HLSG}
\end{equation}
In particular, since $H>0$ by assumption we have that $G>0$.

Let $\lambda \coloneq E/G > 0$. The last inequality in \eqref{E:ASSNS} can be
rewritten as $D \GS \mu E^2/H$. In combination with the first inequality and \eqref{E:HLSG}, we obtain the estimate
\[
    \dot{E}G \LS -DG
        \LS -\mu E^2 G/H
        \LS -\mu E^2.
\]
Combining this with \eqref{Gother} and $ER \GS 0$, it follows that
\begin{equation}
    \dot{\lambda}
        \LS \frac{-\mu E^2 - E(-E+R)}{G^2}
        \LS -\nu\lambda^2,
    \quad\text{and hence}\quad
    \frac{d}{dt}\left( \frac{1}{\lambda} \right) \GS \nu.
\label{E:OBL}
\end{equation}
Integrating the last inequality in \eqref{E:OBL} over $[s,t]$ gives
\[
    \lambda(s) \Big( 1-\nu(t-s) \lambda(t) \Big) \GS \lambda(t)
    \quad\text{for $0 \LS s \LS t$,}
\]
with the term in parentheses positive because $\lambda>0$. Hence
\begin{equation}
    \lambda(s) \GS \frac{\lambda(t)}{1 - \nu(t-s)\lambda(t)}
    \quad\text{for $0 \LS s \LS t$.}
\label{E:KLLM}
\end{equation}
If $\nu\neq 0$, then integration of \eqref{E:KLLM} in $s$ over $[\tau,t]$ gives
\begin{equation}
    -\int_\tau^t \lambda(s) \,ds
        \LS \frac{1}{\nu} \ln(1-y)
    \quad\text{with}\quad
    y \coloneq \nu (t-\tau) \lambda(t) < 1.
\label{E:INTLA}
\end{equation}
If $\nu = 0$, then $-\int_\tau^t \lambda(s) \,ds \LS -(t-\tau) \lambda(t)$ because
$\lambda$ is nonincreasing; see \eqref{E:OBL}.

Consider now $Q \coloneq GE^\sigma$, with $\sigma$ the constant introduced in
\eqref{E:APRIORI}. Then
\begin{equation}
    \frac{d}{dt}(\ln Q) = \frac{\dot{G}}{G} + \sigma \frac{\dot{E}}{E}
        \overset{\eqref{Gother}}{=} \frac{-E+R}{G} + \sigma \frac{\dot{E}}{E}
        \LS -\lambda,
\label{E:POL}
\end{equation}
where we have used \eqref{E:ASSNS}, \eqref{E:APRIORI}, and $G \GS H$ to deduce
\[
    \frac{R}{G} + \sigma \frac{\dot{E}}{E}
        \LS \frac{1}{H} \frac{\sigma HD}{E} - \frac{\sigma D}{E}
        \LS 0.
\]
Integrating \eqref{E:POL} over $[\tau,t]$, we obtain the inequality
\[
    G(t) E(t)^\sigma \overset{\eqref{Gother}}{\LS} H(\tau) E(\tau)^\sigma
        \exp\left( -\int_\tau^t \lambda(s) \,ds \right)
    \quad\text{for $t\GS \tau$.}
\]
For $\nu \neq 0$ we now use \eqref{E:INTLA} to estimate
\[
    G(t) E(t)^\sigma \LS H(\tau) E(\tau)^\sigma (1-y)^\frac{1}{\nu}.
\]
Multiplying this inequality with $(t-\tau)\lambda(t)$ and recalling $E=\lambda
G$, we derive
\begin{equation}
\begin{aligned}
    (t-\tau) E(t)^{1+\sigma}
        & \overset{\hphantom{\eqref{E:INTLA}}}{=}
            (t-\tau) \lambda(t) G(t) E(t)^\sigma
\\
        & \overset{\hphantom{\eqref{E:INTLA}}}{\LS}
            (t-\tau) \lambda(t)  H(\tau) E(\tau)^\sigma (1-y)^\frac{1}{\nu}
\\
        &\overset{\eqref{E:INTLA}}{=}
            H(\tau) E(\tau)^\sigma g_\gamma(y)
\end{aligned}
\label{E:PO1}
\end{equation}
with $g_\gamma(y) \coloneq \frac{y}{\nu} (1-y)^\frac{1}{\nu}$. Here $y$ ranges
over $[0,1)$ if $\nu>0$ and over $(-\infty,0]$ if $\nu<0$. The function
$g_\gamma$ is nonnegative and vanishes at $y=0$ and at the other end of the
admissible range. It has a maximum at $y = \nu/(\nu+1) = \nu/\mu$, with value
\[
    g_\gamma(\nu/\mu) = \frac{1}{\mu} \left( 1-\frac{\nu}{\mu} \right)^\frac{1}{\nu}
        = \mu^{-1-\frac{1}{\nu}}
        = \mu^{-\frac{\mu}{\nu}}
        = K_\gamma.
\]
For $\nu = 0$ and $y \coloneq (t-\tau)\lambda(t)$ we obtain analogously that
\begin{equation}
    (t-\tau) E(t)^{1+\sigma}
        \LS H(\tau) E(\tau)^\sigma g_1(y),
    \qquad
    g_1(y) \coloneq ye^{-y}.
\label{E:PO2}
\end{equation}
The function $g_1$ has its maximum at $y=1$ with value $1/e = K_1$.

We use~\eqref{E:APRIORI}, \eqref{E:APR2} to estimate $H(\tau)$ and $E(\tau)$
in \eqref{E:PO1}, \eqref{E:PO2}, then rearrange terms and take the
$(1+\sigma)$th root to obtain
\[
    E(t) \LS \left( H_\mathrm{max}
        \left( \frac{\gamma H_\mathrm{max}}{\tau} \right)^\sigma
        \frac{K_\gamma }{t-\tau} \right)^\frac{1}{1+\sigma}
    \quad\text{for $t\GS \tau$.}
\]
Choosing $\tau \coloneq \frac{\sigma}{1+\sigma} t$ and simplifying, we arrive at
\eqref{E:FINAL}.

\medskip

\textbf{Step~2.} To prove optimality of the constant $K_\gamma$ in
\eqref{E:FINAL}, let $\sigma=0$. We show that for any initial data $E_0, H_0 >
0$ there exist admissible, positive functions $H, E$, and $D$ with $\max_{t\GS
0} t E(t) = K_\gamma H_\mathrm{max}$. To this end, we aim to solve the system
\begin{equation}
    \dot{E}=-D,
    \qquad
    D= m \frac{E^2}{H},
    \qquad
    \dot{H}=-E,
\label{E:ORIG}
\end{equation}
where $m$ is a smooth function with $m\GS\mu$, to be determined later. Then
\eqref{E:ASSNS} will be satisfied. To solve \eqref{E:ORIG}, we observe that it
implies the equation $\dot{\lambda} = (1-m) \lambda^2$ for $\lambda \coloneq E/H$
with initial data $\lambda_0 \coloneq E_0/H_0$. This equation can be solved on
its own, after which we can rewrite \eqref{E:ORIG} as $\dot{E} = -mE/\lambda$ and
$\dot{H}=-\lambda H$. Take $m = \mu$ on $[0,t_*]$ with $t_* \coloneq H_0/E_0$. On
this interval, we find that $\lambda(t) = \lambda_0/(1+\nu\lambda_0 t)$, and all
inequalities in Step~1 hold as equalities with $G=H$. We have $t_* \lambda_0 =
1$ and
\[
    \text{$y \coloneq \nu t_* \lambda(t_*) = \nu/\mu$ for $\nu \neq 0$},
    \qquad
    \text{$y \coloneq t_* \lambda(t_*) = 1$ for $\nu=0$.}
\]
This gives \eqref{E:PO1}, \eqref{E:PO2} with equality and $\sigma,\tau=0$, thus
$t_*E(t_*) = K_\gamma H_0$. It follows from $\lambda>0$ that $D,E$ are positive,
so $E,H$ are nonincreasing and $H_\mathrm{max} = H_0$.

Now it only remains to extend the functions $H$, $E$, and $D$ for $t>t_*$ in a way
that \eqref{E:ASSNS} is satisfied. For $\nu \GS 0$ we can keep $m = \mu$
because then $\lambda \LS \lambda_0$ and
\[
    H(t) = H_0 \exp\left( -\int_0^t \lambda(s) \,ds \right) > 0,
    \qquad
    E = \lambda H >0
\]
globally. For $\nu < 0$ the choice $m = \mu$ would make $\lambda$ blow up at
time $1/(-\nu\lambda_0) > t_*$, forcing $H \longrightarrow 0$ because $E$ is
bounded. We fix a smooth, nondecreasing function $\chi$ with $\chi=0$ on
$(-\infty,0]$ and $\chi=1$ on $[1,\infty)$, and define for $0 < \delta <
\mu/(-\nu\lambda_0)$
\[
    m(t) \coloneq \mu + (1-\mu) \chi\big( (t-t_*)/\delta \big)
    \quad\text{for $t\GS 0$.}
\]
Then $m=\mu$ up to time $t_*$, $m$ increases to $1$ in $[t_*, t_*+\delta]$, and
$m=1$ afterwards. It follows that for $t\GS t_*+\delta$ we have $\dot{\lambda} = 0$.
Thus $0 < \lambda \LS \lambda_\mathrm{max} < \infty$ and
\[
    H(t) \GS H_0 \exp\left( -\lambda_\mathrm{max}t \right) > 0,
    \qquad
    E = \lambda H >0,
    \qquad
    D > 0
\]
globally. Since $m$ is smooth, the functions $E$, $D$, and $H$ are smooth as well.
\end{proof}


\section{Geometric Harmonic Analysis}
\label{S:GHA}

In this section, we collect definitions and results from geometric harmonic
analysis that will be needed, largely basing our presentation  on
\cite[Section~2]{MarinMartellMitreaMitreaMitrea2022}. The expert reader can skip to Section~\ref{S:MSE}.

\medskip

We say that a Lebesgue measurable subset $\Omega \subseteq \R^n$ has
\emph{locally finite perimeter} if the distributional derivative of its
characteristic function $\mu_\Omega \coloneq -\nabla \ONE_\Omega$ is an $\R^n$-valued
Borel measure in $\R^n$ of locally finite total variation. Polar decomposition
then gives $\mu_\Omega = -\nabla \ONE_\Omega = \nu \,|\nabla \ONE_\Omega|$,
where the total variation measure is given by
\[
  |\nabla \ONE_\Omega| = \HAUS^{n-1}\lfloor\partial_*\Omega
\]
and $\nu$ is an $\R^n$-valued function satisfying $|\nu(x)| = 1$ for
$\HAUS^{n-1}$-a.e.\ $x \in \partial_*\Omega$. Here $\partial_*\Omega$ denotes
the measure theoretic boundary of $\Omega$ (see
\cite[(2.14)]{MarinMartellMitreaMitreaMitrea2022}), as opposed to its
topological boundary $\partial\Omega$. We will refer to $\nu$ as the
\emph{outward unit normal} to the set $\Omega$. From the structure theorem for
sets of locally finite perimeter it follows that the boundary $\partial_*\Omega$
is countably rectifiable; i.e., it can be covered by countably many Lipschitz
graphs up to an $\HAUS^{n-1}$-null set. If $E, F$ are sets of locally finite
perimeter in $\R^n$, and $\nu_E$ and $\nu_F$ denote the outward unit normal
vectors to $E$ and $F$, respectively, then at $\HAUS^{n-1}$-a.e.\ point $x \in
\partial_*E \cap \partial_*F$ one has (see
\cite[Proposition~2.2]{MarinMartellMitreaMitreaMitrea2022})
\[
  \text{either}\quad
  \nu_E(x) = \nu_F(x)
  \quad\text{or}\quad
  \nu_E(x) = -\nu_F(x).
\]
For a set $\Omega \subseteq \R^n$ of locally finite perimeter, the
\emph{approximate tangent plane} exists at $\HAUS^{n-1}$-a.e.\ point $x \in
\partial_* \Omega$ and is given by the orthogonal complement $\nu(x)^\perp$.

\begin{definition}[Ahlfors Regularity]
\label{D:ARS}
A closed set $\Gamma \subseteq \R^n$ is called \emph{Ahlfors regular} if there
exists $C_A \in [1,\infty)$ such that, for all $x\in\Gamma$ and $r\in(0,2
\DIAM(\Gamma))$
\begin{equation}
  \upsilon_{n-1} r^{n-1}/C_A
    \LS \HAUS^{n-1}\big( B_r(x) \cap \Gamma \big)
    \LS C_A \upsilon_{n-1} r^{n-1},
\label{E:AHR}
\end{equation}
where $\upsilon_{n-1}$ is the volume of the unit ball in $\R^{n-1}$. We say that
$\Gamma$ is \emph{lower (upper) Ahlfors regular} when only the lower (upper)
inequality in \eqref{E:AHR} is required to hold. An open, nonempty, proper
subset $\Omega \subseteq \R^n$ is called an \emph{Ahlfors regular domain} if its
boundary $\partial\Omega$ is an Ahlfors regular set and
\begin{equation}
  \HAUS^{n-1}(\partial\Omega \setminus \partial_*\Omega) = 0.
\label{E:FULLM}
\end{equation}
\end{definition}

If $\Omega \subseteq \R^n$ is a Lebesgue measurable subset whose boundary
$\partial\Omega$ is upper Ahlfors regular, then $\Omega$ is a set of locally
finite perimeter. If the measure theoretic boundary $\partial_*\Omega$ has full
$\HAUS^{n-1}$-measure in the topological boundary $\partial\Omega$ (cf.\
\eqref{E:FULLM}), then the outward unit normal $\nu$ to $\Omega$ is well-defined
$\HAUS^{n-1}$-a.e.\ in $\partial\Omega$; see
\cite[(2.19)]{MarinMartellMitreaMitreaMitrea2022}.

If $\Omega \subseteq \R^n$ is an Ahlfors regular domain, then so is $\Omega^-
\coloneq \R^n \setminus \overline{\Omega}$, and
\[
  \partial(\Omega^-) = \partial\Omega,
  \quad
  \partial_*(\Omega^-) = \partial_*\Omega.
\]
Moreover, the outward unit normal to $\Omega^-$ is given by $-\nu$ at
$\HAUS^{n-1}$-a.e.\ point of $\partial\Omega$; see
\cite[Lemma~2.3]{MarinMartellMitreaMitreaMitrea2022}. We define the
\emph{surface measure}
\begin{equation}
  \sigma \coloneq \HAUS^{n-1} \lfloor \Gamma,
  \quad
  \Gamma \coloneq \partial\Omega.
\label{E:SURFM}
\end{equation}
We will be interested in sets with \emph{unbounded} boundary, for which
\eqref{E:AHR} holds for all $r>0$. In the following, function spaces on
$\partial\Omega$ will always be defined in terms of the surface measure
\eqref{E:SURFM}, so we will not expressly indicate the $\sigma$-dependence.

For any subset $\Omega \subseteq \R^n$ of locally finite perimeter, we define
the \emph{relative perimeter} of $\Omega$ in a Borel subset $A \subseteq \R^n$
and the \emph{perimeter} of $\Omega$, respectively, as
\[
  P(\Omega; A) \coloneq \sigma(A)
  \quad\text{and}\quad
  P(\Omega) \coloneq \sigma(\R^n);
\]
see \cite[Remark~12.2]{Maggi2012}. A $C^2$-map $f \colon
(-\epsilon,\epsilon) \times \R^n \longrightarrow \R^n$ for $\epsilon > 0$ will
be called a one-parameter family of diffeomorphisms if the map $x \mapsto f(t,x)
\eqcolon f_t(x)$ is a diffeomorphism of $\R^n$ for every $t$. We say that such a
family $\{f_t\}_t$ is a \emph{local variation in $A$}, with $A \subseteq \R^n$
open, if $f_0(x) = x$ for all $x\in \R^n$ and, for all $t \in (-\epsilon,
\epsilon)$
\[
  \text{the set $\{ x\in \R^n \colon f_t(x) \neq x \}$ is compactly contained in $A$.}
\]
For a local variation $\{f_t\}_t$ we define the vector field $\xi(x) \coloneq
\partial f_t(x)/ \partial t|_{t=0}$. Then
\[
  f_t(x) = x + t \Big( \xi(x) + R(t,x) \Big)
\]
with $R(t,x) \longrightarrow 0$ as $t\to 0$ in the sense that
\[
  \sup_{x\in A} \Big( |R(t,x)| + |D R(t,x)| \Big) \longrightarrow 0
  \quad\text{as $t \to 0$.}
\]

\begin{theorem}[First Variation of Perimeter]
\label{T:FVAR}
If $\Omega \subseteq \R^n$ is an Ahlfors regular domain, $A \subseteq \R^n$ open,
and $\{f_t\}_t$ a local variation in $A$, then (recall \eqref{E:SURFM})
\begin{equation}
  \left.\frac{d}{dt}\right|_{t=0} P\big( f_t(\Omega); A \big)
    = \int_{\R^n} \DIV_\tau \xi \,d\sigma,
\label{E:FVAR}
\end{equation}
where for $\HAUS^{n-1}$-a.e.\ point on $\Gamma$, we define the tangential divergence
as
\begin{equation}
  \DIV_\tau \zeta(x)
    = \Big( \ONE - \nu(x)\otimes\nu(x) \Big) \colon D\zeta(x).
\label{E:DIVATU}
\end{equation}
\end{theorem}

\begin{proof}
We refer the reader to \cite[\hbox{5.1.8}]{Federer1969} (see also
\cite[\hbox{6.3}]{Simon2017}).
\end{proof}

\begin{remark}
If $\Gamma \subseteq \R^n$ is a $C^2$-submanifold without boundary and the
$C^1$-vector field $\xi$ is everywhere tangent along $\Gamma$, then $\int_{\R^n} \DIV_\tau \xi
\,d\sigma = 0$ because of the divergence theorem for manifolds. On the other
hand, if $\xi = V \nu$ for a scalar field $V$, then
\[
  \DIV_\tau (V\nu) = -(V\nu)\cdot H = -V \kappa
  \quad\text{$\sigma$-a.e.,}
\]
where $H \coloneq \kappa \nu$ and $\kappa \coloneq -\DIV_\tau \nu$ is the
\emph{mean curvature}. Indeed
\[
  \DIV_\tau (V\nu)
    = \Big( \ONE - \nu(x)\otimes\nu(x) \Big) \colon
      \Big( (DV)\nu + V(D\nu) \Big)
    = (V\nu) \cdot \Big( \nu (\DIV_\tau \nu) \Big).
\]
We refer the reader to \cite[Section 2.4]{Simon2017} for further information.
\end{remark}

\begin{remark}
\label{R:C2RECT}
In this paper, we consider Ahlfors regular domains with the additional
assumption that the first variation of the perimeter is \emph{representable by
integration}, meaning the existence of a locally finite (generalized) curvature
measure $\kappa$ on the boundary $\Gamma \coloneq \partial\Omega$. Then $\Gamma$
is \emph{$C^2$-rectifiable}: There exists a countable collection $C$ of
$(n-1)$-dimensional $C^2$-submanifolds in $\R^n$ with $\HAUS^{n-1}(\Gamma
\setminus \bigcup C) = 0$; see \cite[Theorem 4.8]{Menne2013}. For the proof one
compares the behavior of $\Gamma$ near ``good'' points to the graph of harmonic
functions, where closeness is controlled in terms of the tilt excess. This
approach was pioneered by De Giorgi \cite{DeGiorgi1961} and Almgren
\cite{Almgren1968}.
\end{remark}


\subsection{Flatness from Oscillation Control}

We are interested in the regime where the Mullins-Sekerka evolution ``flattens''
an unbounded phase interface, driving it ultimately towards a straight line in
two- and a plane in three space dimensions. Studying convergence to an
\emph{unbounded} interface is fundamentally different from studying convergence
in the compact case and requires different techniques. As a quantitative measure
of flatness of the interface we use the notion of $\delta$-flat Ahlfors regular
domains introduced below; see also
\cite[Section~2.4]{MarinMartellMitreaMitreaMitrea2022}.

Let us start by recalling the John-Nirenberg space of functions of bounded mean
oscillation (BMO) defined on an Ahlfors regular set $\Gamma \subseteq \R^n$. We
define the surface measure $\sigma \coloneq \HAUS^{n-1} \lfloor \Gamma$ and the
surface balls
\[
  \Delta(x,r) \coloneq B_r(x) \cap \Gamma
  \quad\text{for $x\in \Gamma$ and $r>0$.}
\]
We will often simply write $\Delta \coloneq \Delta(x,r)$. For each $f \in
L^1_\LOC(\Gamma)$ we define
\begin{equation}
\begin{gathered}
  f_\Delta \coloneq \fint_\Delta f \,d\sigma
  \quad\text{for each surface ball}\quad
  \Delta \subseteq \Gamma,
\\
\text{and}\quad  \|f\|_{\BMO(\Gamma)} \coloneq \sup_{\Delta \subseteq \Gamma}
      \fint_\Delta |f - f_\Delta| \,d\sigma.
\end{gathered}
\label{E:BMONORM}
\end{equation}
Then $\BMO(\Gamma)$ is the space of all functions $f \in L^1_\LOC(\Gamma)$ with
$\|f\|_{\BMO(\Gamma)} < \infty$. An analoguous definition will be used for
vector-valued functions.

If $\Gamma$ is Ahlfors regular, then its surface measure is \emph{doubling},
which means that there is a constant $C \in[1,\infty)$ depending only on $n$ and
$C_A$ in \eqref{E:AHR} with
\[
  \sigma\big( B_{2r}(x) \big) \LS C \sigma\big( B_r(x) \big)
  \quad\text{for all $x \in \Gamma$ and $r>0$.}
\]
For doubling measures, the John-Nirenberg inequality \cite{JohnNirenberg1961}
implies exponential integrability of $\BMO$-functions in the following sense:
There exist $c, C \in (0,\infty)$ depending only on the doubling constant of the
measure $\sigma$ such that
\[
  \fint_\Delta \exp\left\{ \frac{c |f-f_\Delta|}{\|f\|_{\BMO(\Gamma)}} \right\} \,d\sigma \LS C
\]
for each non-constant function $f \in \BMO(\Gamma)$ and each surface ball
$\Delta \subseteq \Gamma$. One can then show that the following seminorms are
equivalent:
\begin{equation}
  \|f\|_{\BMO(\Gamma)}
  \sim
  \sup_{\Delta \subseteq \Gamma}
    \left( \fint_\Delta |f-f_\Delta|^p \,d\sigma \right)^{1/p}
  \sim
  \adjustlimits \sup_{\Delta \subseteq \Gamma} \inf_{c \in \R}
    \left( \fint_\Delta |f-c|^p \,d\sigma \right)^{1/p}
\label{E:EQQ}
\end{equation}
for all $p \in [1,\infty)$ and $f \in L^1_\LOC(\Gamma)$. The equivalence
constants depend only on the doubling constant of the surface measure $\sigma$
and the integrability exponent $p$. We refer the reader to
\cite[pp.~46--50]{MarinMartellMitreaMitreaMitrea2022} for further discussion.

\begin{definition}[$\delta$-AR Domains]
\label{D:DARD}
Consider a parameter $\delta > 0$. A nonempty, proper subset $\Omega \subseteq
\R^n$ is called a \emph{$\delta$-flat Ahlfors regular domain} (or simply a
\emph{$\delta$-AR domain}) provided $\Omega$ is an Ahlfors regular domain (see
Definition~\ref{D:ARS}) whose outward unit normal vector $\nu$ satisfies, with
surface measure \eqref{E:SURFM}, the condition
\begin{equation}
  \|\nu\|_{\BMO(\Gamma)} \LS \delta.
\label{E:BDS}
\end{equation}
\end{definition}

\begin{remark}
The estimate \eqref{E:BDS} is always true for $\delta \GS 1$. If \eqref{E:BDS}
holds with $\delta < 1$, then $\Gamma$ is unbounded; see
\cite[Lemma~2.8]{MarinMartellMitreaMitreaMitrea2022}. We remark that for any
$\delta$-AR domain $\Omega \subseteq \R^n$ with $\delta \in (0,1)$, the outward
unit normal vector $\nu$ satisfies
\begin{equation}
  \inf_{\Delta \subseteq \Gamma} \left| \fint_\Delta \nu \,d\sigma \right|
    > 1-\delta;
\label{E:INF}
\end{equation}
see \cite[(2.285)]{MarinMartellMitreaMitreaMitrea2022}. Conversely, any Ahlfors
regular domain $\Omega \subseteq \R^n$ is $\delta$-flat if
\[
\delta > \sqrt{2} \sqrt{1 - \inf_{\Delta \subseteq \Gamma}
    \left| \fint_\Delta \nu \,d\sigma \right|};
\]
see \cite[(2.286)]{MarinMartellMitreaMitreaMitrea2022}. The statement that an
Ahlfors regular domain $\Omega \subseteq \R^n$ is $\delta$-flat with $\delta \in
(0,1)$ can therefore be thought of as roughly equivalent to the statement that
the $\inf$ over averages of $\nu$ in \eqref{E:INF} is sufficiently close to $1$.
\end{remark}

\begin{example}
\label{E:LLGR}
Given $\delta > 0$, the region
\begin{equation}
  \Omega \coloneq \{ (x',z) \in \R^{n-1} \times \R \colon z > \phi(x') \},
\label{E:GRAPH}
\end{equation}
above the graph of a Lipschitz function $\phi \colon \R^{n-1} \longrightarrow
\R$ with Lipschitz constant smaller than $\delta/2$ is a $\delta$-AR domain. We
sketch the main argument and refer the reader to
\cite[Example~2.3]{MarinMartellMitreaMitreaMitrea2022} for further details.
First, we have that
\[
  \nu\big( x', \phi(x') \big)
    = \frac{\big( \nabla\phi(x'), -1 \big)}{\sqrt{1 + |\nabla\phi(x')|^2}}
  \quad\text{for $\LEB^{n-1}$-a.e.\ $x' \in \R^{n-1}$.}
\]
With $\BE_n$ the $n$th standard unit basis vector of $\R^n$, it follows that
\[
\begin{aligned}
  \nu\big( x', \phi(x') \big) + \BE_n
    & = \left( \frac{\nabla\phi(x')}{\sqrt{1+|\nabla\phi(x')|^2}},
      1 - \frac{1}{\sqrt{1+|\nabla\phi(x')|^2}} \right)
\\
    & = \left( \frac{\nabla\phi(x')}{\sqrt{1+|\nabla\phi(x')|^2}},
      \frac{|\nabla\phi(x')|^2}{\sqrt{1+|\nabla\phi(x')|^2}
        \big( 1 + \sqrt{1+|\nabla\phi(x')|^2} \big)} \right).
\end{aligned}
\]
A computation reveals that  $\LEB^{n-1}$-a.e.\ in $\R^{n-1}$, the Euclidean norm
satisfies
\[
  |\nu + \BE_n|
    = \frac{\sqrt{2} |\nabla\phi|}{(1+|\nabla\phi|^2)^{1/4}
      \big( 1 + \sqrt{1+|\nabla\phi|^2} \big)^{1/2}}
    \LS |\nabla\phi|.
\]
We can now estimate (recall \eqref{E:SURFM})
\begin{align}
  \|\nu\|_{\BMO(\Gamma)}
    & = \|\nu + \BE_n\|_{\BMO(\Gamma)}
\nonumber\\
    & \LS 2 \|\nu + \BE_n\|_{L^\infty(\Gamma)}
      \LS 2 \|\nabla\phi\|_{L^\infty(\Gamma)}
      < \delta.
\label{E:CRUDE}
\end{align}
Therefore $\Omega$ is a $\delta$-AR domain. Notice that the graph of $\phi$ may
have kinks.
\end{example}

The estimate of the $\BMO$-seminorm in terms of the $L^\infty$-norm in
\eqref{E:CRUDE} is rather crude; there are functions of bounded mean
oscillation that are unbounded. This allows us to consider domains with
slightly more singular boundaries.

\begin{example}
\label{E:STEEP}
We consider again a domain $\Omega$ defined as a supergraph as in
\eqref{E:GRAPH}, where $\phi \in L^1_\LOC(\R^{n-1})$ with $\nabla\phi \in
\BMO(\R^{n-1})$. Concretely, let $n=2$ and
\[
  \phi_\epsilon(x) \coloneq \begin{cases}
    \epsilon x \big( \log |x|-1 \big) & \text{if $x\in \R\setminus \{0\}$,}
\\
    0 & \text{if $x = 0$,} \end{cases}
\]
with $\epsilon > 0$ an arbitrary constant. The derivative $\phi'_\epsilon =
\epsilon \log |\cdot|$ is unbounded near $x = 0$. In particular, the function
$\phi_\epsilon$ is not Lipschitz continuous. On the other hand, one can show
that the supergraph of $\phi_\epsilon$ is a $\delta$-AR domain for $\epsilon$
sufficiently small. We refer the reader to
\cite[Example~2.5]{MarinMartellMitreaMitreaMitrea2022} for more details.
\end{example}

It turns out that being a $\delta$-AR domain for small $\delta > 0$ is precisely
the condition on $\Omega \subseteq \R^n$ that facilitates a rich theory of
singular integral operators on its boundary. There are two required ingredients.
The first one is uniform rectifiability, which we will discuss in
Section~\ref{S:SIO}. The second notion is Ahlfors regularity.

\begin{proposition}[Density Control]
\label{P:DENS}
Let $\Omega \subseteq \R^n$ be an Ahlfors regular domain with surface measure
\eqref{E:SURFM}. Denote by $\nu$ the outward unit normal to $\Omega$. There
exists a constant $C \in (0,\infty)$ depending only on the dimension $n$ and the
Ahlfors regularity constant $C_A$ (cf.\ Definition~\ref{D:ARS}) of $\Gamma$ with the property that
\begin{equation}
  \sup_{x \in \Gamma, r>0}
    \left| \frac{\sigma\big( \Delta(x,r) \big)}{\upsilon_{n-1}r^{n-1}} - 1 \right|
      \LS C \|\nu\|_{\BMO(\Gamma)} \Big( 1-\log\|\nu\|_{\BMO(\Gamma)} \Big),
\label{E:DEBB}
\end{equation}
with $\upsilon_{n-1}$ the volume of the unit ball in $\R^{n-1}$.
\end{proposition}

\begin{proof}
We refer the reader to \cite[Corollary~2.3]{MarinMartellMitreaMitreaMitrea2022}.
\end{proof}
Notice that no explicit curvature bound is needed in the proof (cf.\
Theorem~\ref{T:ALLARD}). The $\delta$-flatness assumption suffices to prove that
in the neighborhood of each point, the interface is contained in a Lipschitz
graph, up to a small set.

The proposition expresses the fact that, if there is a finite Ahlfors regularity constant
 of $\Gamma = \partial\Omega$  (which follows from the assumption that $\Omega$ is an Ahlfors regular
domain), then it can be controlled in terms of the flatness of $\Gamma$. On the one hand, if $\Gamma$ is a hyperplane, then the BMO-norm of the normal is zero and the quotient on the left-hand side of
\eqref{E:DEBB} is equal to $1$ for all surface balls $\Delta(x,r)$. On the other hand, if $\Gamma$ converges to a hyperplane, then $\|\nu\|_{\BMO(\Gamma)} \to
0$ and, according to \eqref{E:DEBB}, the Ahlfors constant converges to $1$.


\subsection{Minimal Surface Theory}
\label{S:MST}

The Mullins-Sekerka evolution can be interpreted as a gradient flow of the area
functional; see Section~\ref{S:GFS} for details. It is therefore not surprising
that certain notions and results originally introduced in the theory of minimal
surfaces will be relevant for our investigation.

\begin{definition}[(Oriented) Tilt Excess]
\label{D:OTE}
Let $\Omega \subset \R^n$ be an Ahlfors regular domain and $\sigma \coloneq
\HAUS^{n-1} \lfloor \partial\Omega$. The \emph{(oriented) tilt excess} of the
boundary $\Gamma \coloneq \partial\Omega$ at the point $y\in \Gamma$, for $r >
0$, and with respect to the direction $\BE \in S^{n-1}$ is given by
\begin{equation}
  \OTE_r(y,\BE) \coloneq r^{1-n} \int_{C_r(y,\BE) \cap \Gamma}
    \HALF |\nu(x)-\BE|^2 \,d\sigma(x).
\label{E:OTE}
\end{equation}
Here $\nu$ is the measure theoretic outward unit normal vector to $\Omega$ and
\[
  C_r(y,\BE) \coloneq \left\{ x\in \R^n \colon
    \text{$|(x-y)\cdot \BE| < r$ and $\big|(\ONE-\BE\otimes\BE) (x-y) \big| < r$} \right\}.
\]
Note that $\nu$ is well-defined $\sigma$-a.e.\ and $\HALF |\nu-\BE|^2 = 1 -
(\nu\cdot\BE) \in [0,2]$.
\end{definition}

The tilt excess and the $\BMO$-norm of the normal are related (recall
\eqref{E:EQQ}).

\begin{lemma}
Let $\Omega \subseteq \R^n$ be an Ahlfors regular domain with surface measure
\eqref{E:SURFM}. There exists a constant $C \in (0,\infty)$ depending only on
$n$ and the Ahlfors regularity constant $C_A$ of $\Gamma$ with the property that
for all $y \in \Gamma$ and $r>0$
\[
  \fint_{\Delta(y,r)} |\nu-\nu_{\Delta(y,r)}|^2 \,d\sigma
    \LS C E_r(y,\BE)
\]
for any $\BE \in S^{n-1}$. Moreover, provided $\nu_{\Delta(x,r)} \neq 0$, we
have that
\[
  E_{r/\sqrt{2}}\left( y, \frac{\nu_{\Delta(y,r)}}{|\nu_{\Delta(y,r)}|} \right)
    \LS C \fint_{\Delta(y,r)} |\nu-\nu_{\Delta(y,r)}|^2 \,d\sigma.
\]
\end{lemma}

\begin{proof}
We refer the reader to \cite[Lemma~3.4]{BortzEngelsteinGoeringToroZhao2022}
(see also \cite[Lemma~2.9]{MarinMartellMitreaMitreaMitrea2022}).
\end{proof}

The following result provides a different interpretation of the tilt excess.

\begin{proposition}
\label{P:DIFMAS}
Let $\Omega \subseteq \R^n$ be an Ahlfors regular domain with surface measure
\eqref{E:SURFM}. For fixed $\BE \in S^{n-1}$ we denote by $\pi_\BE \colon \R^n
\longrightarrow \R^n$ the orthogonal projection onto the hyperplane $\BE^\perp$
perpendicular to $\BE$. Then
\[
  \int_G \nu\cdot\BE \,d\sigma  = \HAUS^{n-1}\big( \pi_\BE(G) \big)
  \quad\text{for any $G \subseteq \Gamma$ Borel.}
\]
As a consequence, the tilt excess is actually a difference of measures:
\[
  \OTE_r(y,\BE) = \sigma(G) - \HAUS^{n-1}\big( \pi_\BE(G) \big)
  \quad\text{with}\quad
  G \coloneq C_r(y,\BE) \cap \Gamma.
\]
\end{proposition}

\begin{proof}
We refer the reader to \cite[Lemma~22.11]{Maggi2012}.
\end{proof}

This motivates the following definition.

\begin{definition}[Excess Energy]
With the notation of Definition~\ref{D:OTE} and with $\BE_n$ the $n$th standard
unit basis vector of $\R^n$, we define the \emph{excess energy} of $\Gamma$ as
\begin{equation}
  \OTE \coloneq \int_{\R^n} \HALF |\nu(x)+\BE_n|^2 \,d\sigma(x)
    = \int_{\R^n} \big( 1+(\nu\cdot\BE_n) \big) \,d\sigma.
\label{E:MEEQ}
\end{equation}
\end{definition}

Notice that we are interested in the situation where the boundary $\Gamma =
\partial\Omega$ is a deformation of the flat hyperplane $\R^{n-1} \times \{0\}$,
with $\Omega$ \emph{above} its boundary. In this case, the outward unit normal
vector is pointing in the direction of $-\BE_n$. The excess energy $E$ is a
measure for the extent to which $\Gamma$ deviates from said hyperplane. It
vanishes if and only if $\Gamma = \R^{n-1}\times\{0\}$; see
\cite[Proposition~22.2]{Maggi2012}.

\medskip

Next we introduce another, closely related concept of tilt excess. Notice first
that vectors $\BE \in S^{n-1}$ are in one-to-one correspondence with the
projection operators onto the hyperplanes $\BE^\perp$ orthogonal to $\BE$, given
by the $(n\times n)$-matrix $\ONE - \BE\otimes\BE$, with $\ONE$ the unit matrix.
Let $\|\cdot\|$ be the Frobenius norm of matrices.

\begin{definition}[(Non-Oriented) Tilt Excess]
In the notation of Definition~\ref{D:OTE}, we call \emph{(non-oriented) tilt
excess} of the boundary $\Gamma = \partial\Omega$ at the point $y\in \Gamma$,
for $r > 0$, and with respect to the direction $\BE \in S^{n-1}$ the quantity
\[
  \NOTE_r(y,\BE) \coloneq r^{1-n} \int_{C_r(y,\BE) \cap \Gamma}
    \HALF \|\pi_{\nu(x)}-\pi_\BE\|^2 \,d\sigma(x).
\]
\end{definition}

Note that, by definition of the Frobenius inner product of matrices, we have
\[
  \HALF \|\pi_\nu-\pi_\BE\|^2
    = \HALF \|\nu\otimes\nu - \BE\otimes\BE\|^2
    = 1 - (\nu\cdot\BE)^2.
\]
It follows that always $\NOTE_r(y,\BE) \LS 2\OTE_r(y,\BE)$. Conversely, we have
\[
  \OTE_r(y,\BE) \LS \NOTE_r(y,\BE)/\alpha
  \quad\text{provided}\quad
  1+\nu\cdot\BE \GS \alpha
\]
for $\sigma$-a.e.\ point in $C_r(y,\BE) \cap \Gamma$. Similar statements apply
to the oriented and the non-oriented excess energies, the latter of which we
define as
\begin{align}
  \NOTE \coloneq \int_\Gamma \HALF \|\pi_{\nu(x)}-\pi_{\BE_n}\|^2 \,d\sigma(x)
    = \int_\Gamma \big( 1-(\nu\cdot\BE_n)^2 \big) \,d\sigma.\label{E:EBARDEF}
\end{align}

\begin{example}
For $\Omega \coloneq \R^{n-1} \times [0,1]$, we have that $\NOTE = 0$ but $\OTE =
\infty$.
\end{example}

\begin{theorem}[Allard's Regularity Theorem]
\label{T:ALLARD}
Let $\Omega \subseteq \R^n$ be an Ahlfors regular domain, $\sigma \coloneq
\HAUS^{n-1} \lfloor \partial\Omega$, and $\Gamma \coloneq \partial\Omega$. For
$p > n-1$ there are constants $\eta, \gamma \in (0,1)$ and $C \in (0,\infty)$
depending only on $n, p$, with the following property: Suppose that, for any
$\xi \in \Gamma$, $R > 0$, and $\alpha\coloneq1-(n-1)/p$ we have
\begin{equation}
  \frac{\sigma\big( B_R(\xi) \big)}{\upsilon_{n-1} R^{n-1}} \LS 1+\eta,
  \quad
  \NOTE_R(\xi,\BE_n) \LS \eta,
  \quad
  R^\alpha \left( \int_{B_R(\xi)} |\kappa|^p \,d\sigma \right)^{1/p} \LS \eta,
\label{E:ALLCOND}
\end{equation}
with $\BE_n$ the $n$th standard basis vector in $\R^n$. There exists a
$C^{1,\alpha}$-function $h \colon \overline{B} \longrightarrow \R$, with $B
\subset \R^{n-1}$ the open ball of radius $\gamma R$ centered at
$\pi_{\BE_n}(\xi)$, such that
\begin{gather*}
  \Gamma \cap B_{\gamma R}(\xi) = \mathrm{graph}(h) \cap B_{\gamma R}(\xi)
  \quad\text{and}
\\
\begin{aligned}
  & R^{-1} \sup |h-\xi_n| + \sup |Dh|
    + R^\alpha \sup_{x,y \in B, x \neq y} \frac{|Dh(x)-Dh(y)|}{|x-y|^\alpha}
\\
  & \qquad
    \LS C \left( \NOTE_R(\xi, \BE_n)^{1/2}
      + R^\alpha \left( \int_{B_R(\xi)} |\kappa|^p \,d\sigma \right)^{1/p} \right).
\end{aligned}
\end{gather*}
\end{theorem}

\begin{proof}
We refer the reader to \cite[Theorem~23.1]{Simon1983}. As the result in
\cite{Simon1983} is stated for rectifiable varifolds, it requires an additional
condition on the multiplicity function, which is not needed here. Recall also
that $\NOTE_R(\xi,\BE) \LS 2 \OTE_R(\xi,\BE)$.
\end{proof}


\subsection{Singular Integral Operators}
\label{S:SIO}

In this section, we collect various results from geometric harmonic analysis,
starting with the notion of two-sided NTA domains; see Definition~\ref{D:NTA}
below. We give precise definitions for the sake of completeness, but will never
need to work with them directly because, if $\Omega$ is a $\delta$-AR domain for
$\delta \in (0,1)$ sufficiently small, then $\Omega$ is a two-sided NTA; see
Proposition~\ref{P:TOPPROPS}.

\begin{definition}
Fix $R \in (0,\infty]$ and $c\in (0,1)$. A nonempty, proper subset $\Omega
\subseteq \R^n$ is said to satisfy the \emph{$(R,c)$-corkscrew condition} if for
each $x \in \partial\Omega$ and $r \in (0,R)$ there exists a point $z \in
\Omega$ with the property that $B_{cr}(z) \subseteq B_r(x) \cap \Omega$.
\end{definition}

\begin{definition}\label{def:UR}
A closed set $\Gamma \subseteq \R^n$ is said to be \emph{uniformly rectifiable}
(or simply a \emph{UR} set) if $\Gamma$ is an Ahlfors regular set (see
Definition~\ref{D:ARS}) and there exist $\alpha, M \in (0,\infty)$ such that,
for each $x \in \Gamma$ and each scale $R \in (0,2 \, \DIAM(\Gamma))$, it is
possible to find a Lipschitz map $\varphi \colon B_R^{n-1} \longrightarrow \R^n$
(where $B_R^{n-1} \subseteq \R^{n-1}$ is a ball of radius $R$) with Lipschitz
constant not larger than $M$, such that
\[
  \HAUS^{n-1}\big( \Gamma \cap B_R(x) \cap \varphi(B_R^{n-1}) \big)
    \GS \alpha R^{n-1}.
\]
We refer to $\alpha, M$ as the \emph{UR constants} of $\Gamma$. An open,
nonempty, proper set $\Omega \subseteq \R^n$ is called a \emph{UR domain}
provided $\partial\Omega$ is a UR set and $\HAUS^{n-1}(\partial\Omega \setminus
\partial_*\Omega) = 0$.
\end{definition}

\begin{definition}
Fix $R \in (0,\infty]$ and $N \in \N$. An open set is said to satisfy the
\emph{$(R,N)$-Harnack chain condition} if the following statement is true: for
every choice of $\epsilon > 0$, $k \in \N$, $z \in \partial\Omega$, $x,y \in
\Omega$ with $\max\{|x-z|,|y-z|\} < R/4$ as well as
\[
  |x-y| \LS 2^k \epsilon
  \quad\text{and}\quad
  \min\big\{ \DIST(x,\partial\Omega), \DIST(y,\partial\Omega) \big\} \GS \epsilon,
\]
there exists a chain of balls $B_1, \ldots, B_K$ with $K \LS Nk$ such that
\begin{align*}
  x\in B_1,
  \quad
  y\in B_K,
  \quad
  B_i \cap B_{i+1} \neq \varnothing
  & \quad\text{for $i=1\ldots N-1$,}
\\
  \left.\begin{aligned}
    \DIST(B_i,\partial\Omega)/N & \LS \DIAM(B_i) \LS N \DIST(B_i,\partial\Omega)
\\
    \DIAM(B_i) & \GS \min\big\{ \DIST(x,B_i), \DIST(y,B_i) \big\}/N
  \end{aligned}\right\}
  & \quad\text{for $i = 1\ldots K$.}
\end{align*}
\end{definition}

\begin{definition}\label{D:NTA} Fix $R \in (0,\infty]$ and $N \in \N$. An open,
nonempty, proper subset $\Omega \subseteq \R^n$ is called an
\emph{$(R,N)$-two-sided nontangentially accessible domain} (or simply a
\emph{two-sided NTA domain} if the particular values of $R,N$ are not important)
if both $\Omega$ and $\R^n \setminus \overline{\Omega}$ satisfy both the
$(R,N^{-1})$-corkscrew condition and the $(R,N)$-Harnack chain condition. One
takes $R=\infty$ if and only if $\partial\Omega$ is unbounded.
\end{definition}

For a Lebesgue measurable set $\Omega \subseteq \R^n$ satisfying a two-sided
corkscrew condition, one has that $\partial_*\Omega = \partial\Omega$. Any open,
nonempty, proper set $\Omega \subseteq \R^n$ satisfying a two-sided corkscrew
condition and with Ahlfors regular boundary is uniformly rectifiable. Finally,
any NTA domain with unbounded boundary is pathwise connected, therefore also
connected. We refer the reader to
\cite[Section~2.1]{MarinMartellMitreaMitreaMitrea2022} for further details.

\begin{proposition}
\label{P:TOPPROPS}
For any dimension $n\GS 2$ and $C_A \in [1,\infty)$ there is a constant
$\delta_* \in (0,1)$ with the following property: Every $\delta$-AR domain
$\Omega \subseteq \R^n$ with Ahlfors regularity constant $C_A$ and $\delta <
\delta_*$ is a two-sided NTA domain, with all the relevant constants controlled
solely in terms of $C_A$ and $n$. Moreover, the sets
\begin{equation}
  \Omega^+ \coloneq \Omega,
  \quad
  \Omega^- \coloneq \R^n \setminus \overline{\Omega},
  \quad
  \Gamma \coloneq \partial\Omega^\pm
\label{E:SETS}
\end{equation}
are all unbounded and connected, with $\partial(\R^n \setminus
\overline{\Omega}) = \partial\Omega$.

For $n=2$ both $\Omega$ and $\R^2\setminus \overline{\Omega}$ are even simply
connected.
\end{proposition}

\begin{proof}
We refer the reader to \cite[Theorem~2.5]{MarinMartellMitreaMitreaMitrea2022}.
\end{proof}

The fact that a $\delta$-AR domain $\Omega \subseteq \R^n$ with $\delta \in
(0,1)$ has an unbounded boundary is a consequence of the global character of the
$\BMO$-norm \eqref{E:BMONORM}, which captures arbitrarily large scales. This is
appropriate for our application, which is concerned with convergence of
Mullins-Sekerka solutions to flat hyperplanes. In other situations, local
versions of the $\BMO$-norm can be considered; see for example
\cite{BortzEngelsteinGoeringToroZhao2022}.

\medskip

Let us next consider traces. For any open, nonempty, proper subset $\Omega
\subseteq \R^n$ and aperture parameter $\varkappa \in (0,\infty)$ we define the
nontangential approach region
\[
  G_{\varkappa}(x) \coloneq \Big\{ z\in \Omega:
    |z-x| \LS (1+\varkappa) \DIST(z,\partial\Omega) \Big\}
\]
for each $x\in \partial\Omega$. Then the \emph{nontangential maximal operator
$\NT_\varkappa$} for any Lebesgue-measurable function $u$ defined on $\Omega$ is
given as
\[
  \NT_\varkappa u(x) \coloneq \|u\|_{L^\infty(G_\varkappa(x))},
\]
with the $L^\infty$-norm taken with respect to the Lebesgue measure $\LEB^n$.
One can show that the map $\NT_\varkappa u \colon \partial\Omega \longrightarrow
[0,\infty]$ is lower-semicontinuous (see
\cite[(2.6)]{MarinMartellMitreaMitreaMitrea2022}). If $x\in\partial\Omega$ is
such that $x \in \overline{G_\varkappa(x)}$ (an accumulation point of the
nontangential approach region $G_\varkappa(x)$), then we say that the
\emph{nontangential limit} of $u$ at $x$ from within $G_\varkappa(x)$ exists with
value $a$, provided that for every $\epsilon>0$ there exists $r>0$ with
\[
  |u(y)-a| <  \epsilon
  \quad\text{for $\LEB^n$-a.e.\ $y \in B_r(x) \cap G_\varkappa(x)$.}
\]
If $\Omega$ is an Ahlfors regular domain, then $\HAUS^{n-1}$-a.e.\ $x \in
\partial\Omega$ is in fact an accumulation point of $\overline{G_\varkappa(x)}$;
see \cite[(2.44)]{MarinMartellMitreaMitreaMitrea2022}. Whenever the
nontangential limit of $u$ at $x$ exists as above, we will denote its value by
$\TR_\varkappa u(x)$. Then
\[
  |\TR_\varkappa u(x)| \LS \NT_\varkappa u(x)
  \quad\text{for such $x \in \partial\Omega$.}
\]
With $u^\pm$ functions defined in $\Omega^\pm$ (see \eqref{E:SETS}), we write
\[
  \llbracket u \rrbracket \coloneq \TR_\varkappa u^+ - \TR_\varkappa u^-
\]
where the nontangential limits exist on either side of the boundary
$\partial\Omega$.

\medskip

For an Ahlfors regular domain $\Omega \subseteq \R^n$ with surface measure
\eqref{E:SURFM}, the Lebesgue spaces $L^p(\Gamma)$ for $p \in [1,\infty]$ can be
defined as usual. Standard arguments ensure that the set $C^\infty_c(\R^n)$ of
smooth functions with compact support in $\R^n$, when restricted to $\Gamma$, is
dense in $L^p(\Gamma)$ for any $p \in [1,\infty)$.

For $j,k=1\ldots n$ we define the first-order tangential derivatives
\begin{equation}
  \partial_{\tau_{jk}} \varphi \coloneq \nu_j (\partial_k\varphi)|_{\Gamma}
    - \nu_k (\partial_j\varphi)|_{\Gamma}
\quad\text{for $\varphi \in C^\infty_c(\R^n)$.}
\label{E:TANGD}
\end{equation}
For a pair of functions $\varphi, \psi \in C^\infty(\R^n)$ we have the
integration-by-parts formula
\begin{equation}
  \int_{\R^n} (\partial_{\tau_{jk}}\varphi) \psi \,d\sigma
    = -\int_{\R^n} \varphi (\partial_{\tau_{jk}}\psi) \,d\sigma,
\label{E:WEAK}
\end{equation}
which follows from the Gauss-Green theorem by considering the divergence-free vector
field $\partial_k(\varphi\psi) \BE_j - \partial_j(\varphi\psi) \BE_k$, with
$\BE_j$ the $j$th standard basis vector in $\R^n$. Formula \eqref{E:WEAK}
motivates the following definition: For any given function $f \in
L^1_\LOC(\Gamma)$ we call $f_{jk} \in L^1_\LOC(\Gamma)$ \emph{weak tangential
derivative} of $f$ provided
\[
  \int_{\R^n} (\partial_{\tau_{jk}}\varphi) f \,d\sigma
    = -\int_{\R^n} \varphi f_{jk} \,d\sigma
  \quad\text{for all $\varphi \in C^\infty_c(\R^n)$.}
\]
One can show that in this case the function $f_{jk}$ is uniquely determined, and
so we write $\partial_{\tau_{jk}}f \coloneq f_{jk}$. We refer the reader to
\cite[Section~2.8]{MarinMartellMitreaMitreaMitrea2022} for further details.

If $\Omega$ is a two-sided NTA domain with $\Gamma = \partial\Omega$ an
unbounded Ahlfors regular set, then for any $p \in (1,\infty)$ we define the
\emph{homogeneous Sobolev space}
\[
  \dot{W}^{1,p}(\Gamma) \coloneq \Big\{ f \in L^1\Big( \Gamma,
      \textstyle\frac{\sigma}{1+|\cdot|^n} \Big) \colon
    \text{$\partial_{\tau_{jk}} f \in L^p(\Gamma)$
      for $j,k=1\ldots n$} \Big\}
\]
(writing $\dot{H}^1(\Gamma) \coloneq W^{1,2}(\Gamma)$ for $p=2$), equipped with the
seminorm
\[
  \|f\|_{\dot{W}^{1,p}(\Gamma)}
    \coloneq \sum_{j,k=1}^n \|\partial_{\tau_{jk}}f\|_{L^p(\Gamma)}.
\]
All constant functions on $\Gamma$ belong to $\dot{W}^{1,p}(\Gamma)$ and their
seminorm vanishes. It will be convenient to work with the quotient space
$\dot{W}^{1,p}(\Gamma) / \sim$ of equivalence classes $[f]$ of $f \in
\dot{W}^{1,p}(\Gamma)$ modulo constants, equipped with the seminorm
\begin{equation}
  \|[f]\|_{\dot{W}^{1,p}(\Gamma) / \sim}
    \coloneq \|f\|_{\dot{W}^{1,p}(\Gamma)}.
\label{E:SEMI}
\end{equation}
Then $\dot{W}^{1,p}(\Gamma) / \sim$ is a Banach space and \eqref{E:SEMI} a norm;
see \cite[Proposition~2.26]{MarinMartellMitreaMitreaMitrea2022}.

\begin{theorem}[Transmission Problem]
\label{T:TP}
Let $\Omega \subseteq \R^n$ be a two-sided NTA domain with unbounded Ahlfors
regular boundary $\partial\Omega$. Denote by $\nu$ the outward unit normal
vector to $\Omega$ and recall the notation \eqref{E:SURFM} and \eqref{E:SETS}.
Fix an integrability constant $p \in (1,\infty)$ and an aperture parameter
$\varkappa \in (0,\infty)$.

\begin{enumerate}[label=(\roman*)]
\item
\label{I:ONE}
The transmission problem
\begin{equation}
\begin{aligned}
  u^\pm & \in C^\infty(\Omega^\pm),
\\
  \Delta u^\pm & = 0 \quad\text{in $\Omega^\pm$,}
\\
 \NT_\varkappa (\nabla u^\pm) & \in L^p(\Gamma),
\\
  \llbracket u \rrbracket & = 0 \quad\text{in $\Gamma$,}
\\
  \llbracket \nu\cdot\nabla u \rrbracket & = f \in L^p(\Gamma)
\end{aligned}
\label{E:TRAMI}
\end{equation}
is solvable and any two solutions of \eqref{E:TRAMI} differ at most by a
constant. Here $\llbracket \nu\cdot\nabla u \rrbracket = \nu \cdot \llbracket
\nabla u \rrbracket$ (same normal on either side of $\Gamma$). Every solution of
\eqref{E:TRAMI} is given in the form $u^\pm = \SL^\pm f + c$ in $\Omega^\pm$,
with constant $c\in\R$ and single layer potential operator $\SL^\pm$ defined in
\eqref{E:SLP2} below.

The nontangential limits of $u^\pm$ and $\nu\cdot\nabla u^\pm$ exist
$\sigma$-a.e. The Neumann-to-Dirichlet map $S$, defined by $\TR_\varkappa u^\pm
\eqcolon Sf + c$ with $u^\pm$ as above, can again be expressed in terms of a
single layer potential operator; see \eqref{E:NTDM}.

There exists a constant $C$, depending only on $n, p, \varkappa$, the Ahlfors
constant $C_A$, and the UR constants of $\Gamma$ (cf. Definition~\ref{def:UR}), such that
\begin{equation}
  \|\TR_\varkappa u^\pm\|_{\dot{W}^{1,p}(\Gamma)},
  \|\TR_\varkappa (\nu\cdot \nabla u^\pm)\|_{L^p(\Gamma)}
    \LS C \|f\|_{L^p(\Gamma)}.
\label{E:QUAN}
\end{equation}
The Neumann-to-Dirichlet map $S \colon L^p(\Gamma) \longrightarrow
\dot{W}^{1,p}(\Gamma)$ is bounded.

\item
\label{I:TWO}
There exists a $\delta \in (0,1)$ that depends only $n, p,$ and the Ahlfors
regularity parameter of $\Gamma$ with the following property: If
$\|\nu\|_{\BMO(\Gamma)} < \delta$, then the Neumann-to-Dirichlet map $S$ is
surjective modulo constants, meaning that for any $g \in \dot{W}^{1,p}(\Gamma)$,
there are $f \in L^p(\Gamma), c\in \R$ with $Sf = g+c$.

Moreover, there exists a constant $C$ (again depending only on $n, p, \varkappa$, $C_A$, and the UR constants of $\Gamma$) such  that
\begin{equation}
  \|M(f)\|_{L^p(\Gamma)}
    \LS C \vartheta\Big( \|\nu\|_{\BMO(\Gamma)} \Big) \,
      \|f\|_{L^p(\Gamma)}
\label{E:SMALL}
\end{equation}
for all $f, u^\pm$ as above, with $\vartheta(s) \coloneq s \log(1/s)$ for $s\in(0,1)$
and
\begin{equation}
  M(f) \coloneq \HALF \Big( \TR_\varkappa(\nu\cdot\nabla u^+) + \TR_\varkappa(\nu\cdot\nabla u^-) \Big).
\label{E:DOUBLE}
\end{equation}
\end{enumerate}
\end{theorem}

\begin{proof}
The result is taken from \cite{MarinMartellMitreaMitreaMitrea2022}. We sketch
the proof to pinpoint where constants depend on the geometric properties of
$\Gamma = \partial\Omega$. We proceed in two steps.

\medskip

\textbf{Step~1.} Item~\ref{I:ONE} follows from
\cite[Theorem~6.15]{MarinMartellMitreaMitreaMitrea2022}. All solutions $u^\pm
\in C^\infty(\Omega^\pm)$ of the Laplace equation $\Delta u^\pm = 0$ in
$\Omega^\pm$ can be obtained as
\[
  u^\pm \coloneq \SL^\pm f + c
  \quad\text{in $\Omega^\pm$,}
\]
with $c \in \R$ and $\SL^\pm$ the (modified) single layer potential
\begin{equation}
  \SL^\pm f(x)
    \coloneq \int_{\R^n} \Big( E(x-y)-E_*(y) \Big) f(y) \,d\sigma(y)
  \quad\text{for $x\in \Omega^\pm$;}
\label{E:SLP2}
\end{equation}
see (3.27), (3.38) and
\cite[Proposition~3.5]{MarinMartellMitreaMitreaMitrea2022}. Here $E$ is the
fundamental solution of the Laplacian $\Delta$ in $\R^n$ and $E_*(y) \coloneq
E(-y) \rho(y)$, where $\rho$ is a smooth cut-off function such that $\rho(y) = 0$ if $|y|
\LS 1$ and $\rho(y) = 1$ if $|y| \GS 2$. The purpose of the non-standard
$E_*$-term in \eqref{E:SLP2} is to provide sufficient decay in $y$ for the
integral to be well-defined. Recall that $E(x) = -\frac{1}{2\pi} \log|x|$ for
$n=2$ and $E(x) = \frac{1}{4\pi} |x|^{-1}$ for $n=3$. Notice that we changed the
sign compared to \cite{MarinMartellMitreaMitreaMitrea2022} to align it with the
common convention.

Then $\TR_\varkappa u^\pm(x) = Sf(x) + c$ for $\sigma$-a.e.\ $x$, where
\begin{equation}
  Sf(x) \coloneq \int_{\R^n} \Big( E(x-y)-E_*(y) \Big) f(y) \,d\sigma(y);
\label{E:NTDM}
\end{equation}
see  \cite[(3.42), (3.47)]{MarinMartellMitreaMitreaMitrea2022}. The term $Sf(x)$
is well-defined because
\[
  L^p(\Gamma) \hookrightarrow
    L^1\left( \Gamma, \frac{\sigma}{1+|\cdot|^{n-1}} \right)
  \quad\text{continuously;}
\]
see  \cite[(2.575)]{MarinMartellMitreaMitreaMitrea2022}. It follows that indeed
$\llbracket u \rrbracket = 0$ in $\Gamma$. Moreover
\begin{equation}
  S \colon L^p(\Gamma)
    \longrightarrow \dot{W}^{1,p}(\Gamma),
\label{E:SMOOTH}
\end{equation}
with operator norm depending only on $n, p,$ and the Ahlfors regularity
parameter and UR constants of $\Gamma$. This can be seen as follows. Taking
derivatives $\partial_j(Sf)$ for $j=1\ldots k$, the problem can be reduced to
establishing the $L^p(\Gamma)$-boundedness of Calder\'on-Zygmund singular
integral operators of the form
\begin{equation}
  Tf(x) = \lim_{\epsilon\to 0} \int_{\R^n \setminus \overline{B_\epsilon(x)}}
    k(x-y) f(y) \,d\sigma(y)
\quad\text{for $x\in \Omega^\pm$,}
\label{E:CZ}
\end{equation}
with odd kernel $k \in C^\infty(\R^n \setminus \{0\})$ homogeneous of degree
$1-n$ (here $k \coloneq \partial_j E$). Uniform rectifiability, introduced by
David and Semmes \cite{DavidSemmes1993}, is precisely the necessary condition to
prove the $L^p$-boundedness of such operators. A central role is played by the
(non-centered) Hardy-Littlewood maximal function defined as
\begin{equation}
  \mathcal{M}f(x) \coloneq \sup_{x\in\Delta} \fint_\Delta |f| \,d\sigma
  \quad\text{for $x\in \Gamma$,}
\label{E:MAXFUN}
\end{equation}
where the $\sup$ is taken over all surface balls $\Delta \subseteq \Gamma$
containing $x$. Ahlfors regularity of $\Gamma$ can be used to translate
information about distances between points, derived for example from the
$(1-n)$-homogeneity of the kernel $k$, into information about the
$\sigma$-measure of surface balls of corresponding radius, thus establishing a
connection to \eqref{E:MAXFUN}. $L^p$-boundedness of the maximal function
operator can be established under rather general assumptions. It follows from
Marcinkiewicz interpolation between the $L^\infty$-case, in which the estimate
is immediate with operator norm equal to one, and weak $L^1$-continuity. The
latter is proved using a Besicovitch covering argument in $\R^n$ with constants
 depending only on $n$ (see \cite[Theorem~2.1]{CoifmanWeiss1971}). Hence
\eqref{E:SMOOTH} and the first estimate in \eqref{E:QUAN} follow with explicit
control over all constants.

There holds $\TR_\varkappa(\nu \cdot \nabla u^\pm)(x) = \left( \pm\frac{1}{2}
\mathrm{id} - K^\# \right) f(x)$  for $\sigma$-a.e.\ $x$, with
\begin{equation}
  K^\# f(x) \coloneq \lim_{\varepsilon\to 0}
    \int_{\Gamma\setminus \overline{B_\varepsilon(x)}}
      \nu(x) \cdot \nabla E(x-y) f(y) \,d\sigma(y).
\label{E:KSHARP}
\end{equation}
Again $K^\# f(x)$ is well-defined and $K^\#$ is a singular integral operator as
in \eqref{E:CZ}. We refer the reader to \cite[(3.25), (3.127), and
Proposition~3.5]{MarinMartellMitreaMitreaMitrea2022}. Using the
$L^p(\Gamma)$-boundedness of \eqref{E:KSHARP} we get the second estimate
in \eqref{E:QUAN}. Note that
\begin{align}
  \llbracket \nu\cdot\nabla u \rrbracket & = f,
\nonumber
\\
  \TR_\varkappa(\nu\cdot\nabla u^+) + \TR_\varkappa(\nu\cdot\nabla u^-)
    & = -2 K^\# f.
\label{E:SUM}
\end{align}
For later reference, let us also mention the adjoint operator
\[
  K f(x) \coloneq \lim_{\varepsilon\to 0}
    \int_{\Gamma \setminus \overline{B_\varepsilon(x)}}
      \nu(y) \cdot \nabla E(y-x) f(y) \,d\sigma(y).
\]

\medskip

\textbf{Step~2.} Item~\ref{I:TWO} follows from
\cite[Corollary~4.2]{MarinMartellMitreaMitreaMitrea2022}, which gives a refined
estimate for the operator norm of singular integral operators \eqref{E:CZ} for
kernels
\begin{equation}
  k(x-y) = (x-y)\cdot \nu(y) \, h(x-y)
  \quad\text{for $x,y \in \R^n$},
\label{E:KERN}
\end{equation}
with $h \in C^\infty(\R^n \setminus \{0\})$ even and homogeneous of degree $-n$.
There exist constants $N\in\N$ and $C$ depending only on $n, p,$ and the UR
constants of $\Gamma$ such that
\begin{equation}
  \|T\|_{L^p(\Gamma) \to L^p(\Gamma)}
    \LS C \bigg( \sum_{|\alpha|\LS N} \sup_{S^{n-1}} |\partial^\alpha h| \bigg) \,
      \vartheta\big( \|\nu\|_{\BMO(\Gamma)} \big).
\label{E:BOUND}
\end{equation}
Therefore the operator norm is small when $\Gamma$ is almost flat and vanishes
when $\Gamma$ is a hyperplane in $\R^n$. The same estimate holds with $\nu(y)$ in
\eqref{E:KERN} replaced by $\nu(x)$ (see again
\cite[Theorem~4.2]{MarinMartellMitreaMitreaMitrea2022}). Applying the estimate
\eqref{E:BOUND} with $T = K^\#$ (see \eqref{E:KSHARP}) to \eqref{E:SUM}, we get
\eqref{E:SMALL}. Surjectivity of $S$ modulo constants follows similarly. One
utilizes the fact that the operator $z\, \mathrm{id}+T$ with $|z| > 0$ is
invertible if $\|\nu\|_{\BMO(\Gamma)}$ is small enough. Indeed, because of
\eqref{E:BOUND}, the Neumann series
\[
  (z\, \mathrm{id} + T)^{-1} = z^{-1} \sum_{m=0}^\infty (-z^{-1} T)^m
\]
converges in operator norm; see
\cite[Theorems~4.8 and 4.11]{MarinMartellMitreaMitreaMitrea2022} for details.
\end{proof}


\subsection{Function Spaces}
\label{S:FS}

By Theorem~\ref{T:TP}, for given two-sided NTA domain $\Omega \subseteq \R^n$
with unbounded Ahlfors regular boundary $\Gamma \coloneq \partial \Omega$, the
transmission problem \eqref{E:TRAMI} is uniquely solvable up to constants for
any $f \in L^p(\Gamma)$. We always assume $\|\nu\|_{\BMO(\Gamma)}$ sufficiently
small such that the Neumann-to-Dirichlet map $S \colon L^p(\Gamma)
\longrightarrow \dot{W}^{1,p}(\Gamma)$ is surjective modulo constants; see
item~\ref{I:TWO} of Theorem~\ref{T:TP}.

We are interested in the case $p = 2$. First, we define
\begin{equation}
  \H_0 \coloneq \Big\{ [f] \colon f\in L^2(\Gamma) \Big\}
  \quad\text{with norm}\quad
  \NN{[f]}_0 \coloneq \|f\|_{L^2(\Gamma)}.
\label{E:H0NORM}
\end{equation}
This is well-defined. Indeed suppose there exists another $f' \in L^2(\Gamma)$
with $[f] = [f']$, i.e., there exists a constant $c\in \R$ such that $f-f' = c$
$\sigma$-a.e. Then $c$ must vanish because $f-f' \in L^2(\Gamma)$ and
$\sigma(\Gamma) = \infty$. Second, we define
\[
  \H_1 \coloneq \Big\{ [g] \colon
    \text{$g = Sf$ with $f \in L^2(\Gamma)$} \Big\}
  \quad\text{with norm}\quad
  \NN{[g]}_1 \coloneq \|f\|_{L^2(\Gamma)}.
\]
This is well-defined. Indeed suppose there exists another $g' = Sf'$ with $f'
\in L^2(\Gamma)$ and $[g] = [g']$, i.e., there exists a $c \in \R$ such that
$g-g' = S(f-f') = c$ $\sigma$-a.e. Since the transmission problem
\eqref{E:TRAMI} is uniquely solvable up to constants, it follows that $f = f'$.
Defining the inverse function $N \coloneq S^{-1}$, we can write
\begin{equation}
  \NN{[g]}_1 = \|N(g)\|_{L^2(\Gamma)}
\label{E:ALTER}
\end{equation}
for all $g$ in the range of $S$, thus for all $[g] \in \H_1$. We call $N$ the
\emph{Dirichlet-to-Neumann map}. It acts like a first order derivative,
vanishing on constants. Both $\H_0$ and $\H_1$ are Banach spaces. In fact, they
are Hilbert Spaces, as can be proved by checking the parallelogram identity for
the corresponding norms. Indeed
\begin{align*}
  \NN{[g_1]\pm[g_2]}^2_1
    & = \NN{[g_1\pm g_2]}^2_1
      = \|N(g_1\pm g_2)\|^2_{L^2(\Gamma)}
\\
    & = \NN{[g_1]}^2_1
      + \NN{[g_2]}^2_1
      \pm 2 \int_{\R^n} N(g_1) N(g_2) \,d\sigma
\end{align*}
for all $g_1, g_2$ in the range of $S$; recall \eqref{E:ALTER}. Upon addition of
the $\pm$ contributions, the mixed terms cancel. The inner product
$(\cdot,\cdot)_1$ on $\H_1$ is then uniquely defined by the polarization
identity. A similar argument works for $\H_0$. Both spaces $\H_0, \H_1$ contain
as dense subsets the equivalence classes generated by all functions in
$C^\infty_c(\R^n)$ restricted to $\Gamma$. This is immediate for $\H_0$.
For $\H_1$ we observe that if $g$ is such a function, then $g \in
\dot{H}^1(\Gamma, \sigma)$, hence there exist $f \in
L^2(\Gamma), c\in \R$ with $Sf = g+c$; see item~\ref{I:TWO} of
Theorem~\ref{T:TP}. It follows that $[g] = [Sf]$ and so $[g] \in \H_1$.

Both spaces $\H_0$ and $\H_1$ are comprised of equivalence classes modulo
constants of $\sigma$-measurable functions defined on $\Gamma$. The pair
$\bar{\H} \coloneq (\H_0, \H_1)$ is thus compatible in the sense of real
interpolation theory; see \cite{ChandlerWilseHewettMoiola2015}. Let
\[
  \Delta(\bar{\H}) \coloneq \H_0 \cap \H_1
  \quad\text{and}\quad
  \Sigma(\bar{\H}) \coloneq \H_0 + \H_1.
\]
As follows from our discussion above, the intersection $\Delta(\bar{\H})$ is
dense in both spaces $\H_0$ and $\H_1$. The inner product $(\cdot,\cdot)_1$ with
domain $\Delta(\bar{\H})$, is a closed, densely defined, nonnegative, symmetric
form on $\H_0$, with an associated unbounded, nonnegative, self-adjoint,
injective operator $T \colon \H_0 \longrightarrow \H_0$, which satisfies
\[
  (T\phi, \psi)_0 = (\phi,\psi)_1
  \quad\text{for $\phi,\psi \in \DOM(T)$.}
\]
The domain of $T$ is a dense linear subspace of the Hilbert space
$\Delta(\bar{\H})$, equipped with the usual inner product for intersections of
Hilbert spaces. Now let $R \colon \H_0 \longrightarrow \H_0$ be the unique
self-adjoint square root of $T$, which has domain $\Delta(\bar{\H})$. For $0 <
\theta < 1$ we then define $\H_\theta$ as the completion of $\Delta(\bar{\H})$ in
$\Sigma(\bar{\H})$ with respect to the norm
\[
  \NN{\phi}_\theta \coloneq \NN{R^\theta \phi}_0
  \quad\text{for $\phi \in \Delta(\bar{\H})$.}
\]

Arguing as above, using again that $\sigma(\Gamma) = \infty$, we observe that
for any $\phi \in \H_\theta$ there exists a \emph{unique} $f \in L^2(\Gamma)$
such that $R^\theta\phi = [f]$. We could therefore rephrase the spaces
$\H_\theta$ as spaces of functions, instead of equivalence classes. In a slight
abuse of notation, for any function $g$ with $\phi = [g]$ we define $R^\theta g
\coloneq f$. One can check that $R^\theta$ is linear and $R^\theta g = R^\theta
g'$ whenever $g,g'$ differ only by a constant. Hence
\begin{equation}
  \NN{[g]}_\theta = \|R^\theta g\|_{L^2(\Gamma)}
  \quad\text{for $[g] \in \H_\theta$.}
\label{E:SAME}
\end{equation}
This identity extends to $\theta = 0$ and $\theta = 1$.

Note that formally $T$ plays the role of $N^*N$. Moreover, $N$ is a symmetric
operator because for smooth functions $u, v$ with $\Delta v = 0$ in $\Omega^\pm$,
we have that
\[
\begin{aligned}
  \int_{\R^n} \nabla u \cdot \nabla v \,dx
    & = \int_{\Omega^+} \nabla\cdot (u \nabla v) \,dx
      + \int_{\Omega^-} \nabla\cdot (u \nabla v) \,dx
\\
    & = \int_\Gamma u (\nu^+\cdot\nabla v) \,d\sigma
      + \int_\Gamma u (\nu^-\cdot\nabla v) \,d\sigma
\\
    & = \int_\Gamma u \llbracket \nu\cdot\nabla v \rrbracket \,d\sigma
        = (u, Nv)_{L^2(\Gamma)},
\end{aligned}
\]
by Green's first identity. To simplify notation, we did not expressly indicate
traces or harmonic extensions. We denote by $\nu^\pm$ the outward unit normal
vectors to $\Omega^\pm$, which satify $\nu^- = -\nu^+$ a.e.\ on $\Gamma$. By
symmetry, it follows that
\[
  (u, Nv)_{L^2(\Gamma)}
    = (Nu, v)_{L^2(\Gamma)}.
\]
Setting $u = v$, we observe further that
\[
  (u, Nu)_{L^2(\Gamma)}
    = \int_{\R^n} |\nabla u|^2 \,dx \GS 0.
\]
Since the operator $R = T^{1/2}$ acts like $N$, we will write $N$ instead of
$R$.

Since $\H_s$ contains equivalence classes of functions modulo constants,
elements of its dual space $\H_s^*$ must vanish when tested against constants,
i.e., have vanishing integral. Notice that the dual spaces contain $L^2$-functions,
not equivalence classes of functions. For example, the dual space $\H_0^*$ is
the set of $g \in L^2(\Gamma)$ with $\int_\Gamma g \,d\sigma = 0$, with the dual
pairing between $[f] \in \H_0$ and $g \in \H_0^*$ given by $\int_\Gamma f g
\,d\sigma$. For $s \in [0,1]$ we define $\H_{-s}$ as the dual space $\H_s^*$,
equipped with the usual dual norm. We can think of elements of $\H_{-s}$ as
objects of the form $N^s f$ for $f \in L^2(\Gamma)$.

For $\theta \in [0,1]$ the spaces $\H_\theta$ coincide with real interpolation
spaces:
\[
  \H_\theta = (\H_0,\H_1)_{\theta,2},
\]
with equality of norms; see \cite[Theorem
3.3$'$]{ChandlerWilseHewettMoiola2022}. Hence they are \emph{geometric
inter\-polation spaces of exponent $\theta$}, which provides us with the
interpolation inequality
\begin{equation}
  \NN{\phi}_\theta \LS \NN{\phi}_0^{1-\theta} \NN{\phi}_1^\theta
  \quad\text{for all $\phi \in \H_0\cap\H_1$;}
\label{E:INTP1}
\end{equation}
see \cite[Lemma~2.1 (i)]{ChandlerWilseHewettMoiola2015}. For the dual spaces
$\H_{-\theta}$ there holds
\[
  \H_{-\theta}
    = \H_\theta^*
    = (\H_0,\H_1)_{\theta,2}^*
    = (\H_0^*, \H_1^* )_{\theta, 2}
    = (\H_{-0},\H_{-1})_{\theta,2};
\]
see \cite[Theorem~3.7 (ii)  ]{ChandlerWilseHewettMoiola2015} for a proof of the
third equality. This implies that
\[
  \NN{\phi}_{-\theta} \LS \NN{\phi}_{-0}^{1-\theta} \NN{\phi}_{-1}^\theta
  \quad\text{for all $\phi \in \H_{-0}\cap\H_{-1}$.}
\]

\begin{lemma}
\label{L:INTERP}
Let $\Omega \subseteq \R^n$ be a two-sided NTA domain with unbounded Ahlfors
regular boundary $\partial\Omega$. For any $\phi$ in the respective spaces, we
have
\begin{align}
  \NN{\phi}_\frac{1}{2}
    & \LS \NN{\phi}_{-\frac{1}{2}}^\frac{1}{3} \NN{\phi}_1^\frac{2}{3}
\label{E:INT1}
\\
  \NN{\phi}_0
    & \LS \NN{\phi}_{-\frac{1}{2}}^\frac{2}{3} \NN{\phi}_1^\frac{1}{3}
\label{E:INT2}
\\
  \NN{\phi}_{-\frac{1}{2}}
    & \LS \NN{\phi}_{-1}^\frac{2}{3} \NN{\phi}_\frac{1}{2}^\frac{1}{3}
\label{E:INT4}
\end{align}
\end{lemma}

\begin{proof}
To simplify the notation, we will not distinguish between functions and their
equivalence classes. From \eqref{E:INTP1} with $\theta = 1/2$ we obtain that
\begin{equation}
  \NN{\phi}_\frac{1}{2}^2 \LS \NN{\phi}_0 \NN{\phi}_1.
\label{E:EINS}
\end{equation}
Because of \eqref{E:H0NORM} and duality, we can write
\begin{equation}
  \NN{\phi}_0^2 = (\phi,\phi)_{L^2(\Gamma)}
    \LS \NN{\phi}_{-\frac{1}{2}} \NN{\phi}_\frac{1}{2}.
\label{E:ZWEI}
\end{equation}
Combining \eqref{E:EINS} and \eqref{E:ZWEI}, we find that
\[
  \NN{\phi}_\frac{1}{2}^2
    \LS \NN{\phi}_{-\frac{1}{2}}^\frac{1}{2} \NN{\phi}_\frac{1}{2}^\frac{1}{2} \NN{\phi}_1,
\]
from which \eqref{E:INT1} follows by rearranging terms.

The other estimates can be established analogously. We omit their proof.
\end{proof}

\begin{notation}
\label{N:FS}
Because of \eqref{E:SAME} the estimates of Lemma~\ref{L:INTERP} can also be
expressed in terms of functions. In a slight abuse of notation, we will write
(recall \eqref{E:SURFM})
\[
  \NN{g}_s \coloneq \|N^s(g)\|_{L^2(\Gamma)}
  \quad\text{for $s\in[0,1]$,}
\]
for all functions $g$ (uniquely determined modulo constants) such that $[g] \in
\H_s$. By duality, this can be extended to $s\in[-1,0]$. In this sense, we may
think of $\H_s$ as a space containing functions rather than equivalence classes.
\end{notation}


\subsection{Chord-arc Domains in the Plane}

We specialize the theory of $\delta$-Ahlfors regular domains introduced above to
the case of domains in $\R^2$, canonically identified with the complex plane
$\C$, for which there exists a particularly clean description of domain
boundaries. Our presentation is based on \cite[Sections~2.3 and
2.6]{MarinMartellMitreaMitreaMitrea2022}.

\medskip

We will call a \emph{Jordan curve (passing through infinity)} any set of the
form $\Gamma = \gamma(\R)$ where $\gamma \colon \R \longrightarrow \C$ is a
continuous, injective function with $\lim_{t \to \pm\infty} |\gamma(t)| =
\infty$. The map $\gamma$ is called a parameterization of $\Gamma$. The
complement of a Jordan curve consists precisely of two open, disjoint, and
simply connected components $\Omega_\pm$, with
\[
  \partial\Omega_+ = \partial\Omega_- = \Gamma.
\]
A \emph{compact curve} is any set of the form $\Gamma \coloneq \gamma([a,b])$
for a continuous $\gamma \colon [a,b] \longrightarrow \C$, with $[a,b]$ compact.
The length $L$ of a compact curve is defined as
\[
  L \coloneq \sup\sum_{j=1}^N |\gamma(t_j)-\gamma(t_{j-1})|,
\]
with the supremum taken over all partitions $a = t_0 < t_1 < \cdots < t_N = b$.
The length can be expressed in terms of the Hausdorff measure as $L =
\HAUS^1(\Gamma)$. We have
\begin{equation}
  |z_1-z_2| \LS \HAUS^1(\Gamma)
  \quad\text{for any compact curve $\Gamma$ with endpoints $z_1,z_2$.}
\label{E:LOWERB}
\end{equation}
We call a Jordan curve $\Gamma$ \emph{locally rectifiable} if all of its compact
sub-curves have finite length. In this case, the \emph{arc-length
parameterization} of $\Gamma$ is a map $z \colon \R \longrightarrow \C$ such
that $\Gamma = z(\R)$, $z$ is differentiable almost everywhere, and
\[
  |z'(t)| = 1
  \quad\text{for a.e.\ $t \in \R$}.
\]
There holds $|z(t_1)-z(t_2)| \LS |t_1-t_2|$ for every $t_1, t_2 \in \R$. We will
typically consider Jordan curves equipped with the corresponding arc-length
parameterization.

\begin{definition}
For given $\lambda \in [0,\infty)$, a set $\Gamma \subseteq \C$ is called a
\emph{$\lambda$-CAC} (or \emph{chord-arc curve}) if $\Gamma$ is a locally
rectifiable Jordan curve passing through infinity with
\begin{equation}
  \ell(z_1,z_2) \LS (1+\lambda) |z_1-z_2|
  \quad\text{for all $z_1,z_2 \in \Gamma$,}
\label{E:STRAIGHT}
\end{equation}
where $\ell(z_1,z_2)$ is the length of the shortest sub-curve of $\Gamma$
joining $z_1$ and $z_2$.

A nonempty, proper, open subset $\Omega \subseteq \C$ is called a
\emph{$\lambda$-CAD} (or \emph{chord-arc domain}) if its topological boundary
$\partial\Omega$ is a $\lambda$-CAC.
\end{definition}

Since \eqref{E:LOWERB} always holds, condition~\eqref{E:STRAIGHT} is a measure
of the flatness of $\Gamma$.

\medskip

As explained in Section~\ref{S:SIO}, we will be interested in $\delta$-AR
domains because we can define singular integral operators on the boundary. One
can show that in $\C$ the classes of $\delta$-AR domains and $\lambda$-CADs with
$\delta,\lambda$ small enough coincide.

\begin{proposition}
\label{P:CADBMO}
For any $M \in (0,\infty)$ there exists a $\delta_* \in (0,1)$ with the property
that, if $\delta \in (0,\delta_*)$, then any $\delta$-AR domain $\Omega
\subseteq \C$ with Ahlfors regularity constant not bigger than $M$ is a
$\lambda$-CAD with $\lambda = O(\delta \log(1/\delta))$ as $\delta \to 0+$.

Conversely, if $\Omega \subseteq \C$ is a $\lambda$-CAD with $\lambda \in
[0,\sqrt{2}-1)$, then $\Omega$ is a $\delta$-AR domain with $\delta =
O(\sqrt{\lambda})$ as $\lambda\to 0+$.
\end{proposition}

\begin{proof}
We refer the reader to \cite[Theorem~2.7]{MarinMartellMitreaMitreaMitrea2022}.
\end{proof}

Chord-arc curves have a number of useful properties.

\begin{proposition}
\label{P:PROPS}
Let $\Gamma \subseteq \C$ be a $\lambda$-CAC for some $\lambda \in [0,\infty)$, with
$z\colon \R \longrightarrow \C$ its arc-length parameterization. Then the
following statements are true.
\begin{enumerate}
\item For every Lebesgue measurable $A \subseteq \R$ and $\HAUS^1$-measurable $E
\subseteq \Gamma$ we have
\[
  \HAUS^1\big( z(A) \big) = \LEB^1(A)
  \quad\text{and}\quad
  \HAUS^1(E) = \LEB^1\big( z^{-1}(E) \big).
\]
\item\label{I:INTEG} Defining the surface measure $\sigma \coloneq \HAUS^1
\lfloor \Gamma$, for every $\sigma$-measurable $E \subseteq \Gamma$ and
nonnegative $\sigma$-measurable function $g$ defined on $\Gamma$, we have that
\[
  \int_E g(x) \,d\sigma(x) = \int_{z^{-1}(E)} g\big( z(s) \big) \,ds.
\]
Recall that $z$ is differentiable a.e.\ with $|z'(s)| = 1$ for a.e.\ $s\in\R$.
\item\label{I:OMEGA} Denote by $\Omega$ the region of $\C$ to the left of the
curve $\Gamma$ relative to its arc-length parameterization. Then $\Omega$ is a
set of locally finite perimeter and its measure theoretic outward unit normal
vector is given by
\begin{equation}
  \nu\big( z(s) \big) = -i z'(s)
  \quad\text{for a.e.\ $s\in\R$.}
\label{E:NORMZ}
\end{equation}
\item The set $\Omega$ defined in \eqref{I:OMEGA} is a simply connected,
unbounded, two-sided NTA domain with Ahlfors regular boundary $\partial\Omega =
\Gamma$.

\item We have the following estimates:
\begin{gather*}
  \frac{1}{2(1+\lambda)} \|\nu\|_{\BMO(\Gamma)}
    \LS \|z'\|_{\BMO(\R)}
    \LS \frac{\sqrt{\lambda(2+\lambda)}}{1+\lambda} < 1,
\\
  \frac{1}{2(1+\lambda)} \|z'\|_{\BMO(\R)}
    \LS \|\nu\|_{\BMO(\Gamma)}.
\end{gather*}
\end{enumerate}
\end{proposition}

\begin{proof}
We refer the reader to
\cite[Proposition~2.10]{MarinMartellMitreaMitreaMitrea2022}.
\end{proof}

\begin{remark}
\label{R:INTEG}
By standard arguments, it follows from Proposition~\ref{P:PROPS}~\eqref{I:INTEG}
that
\begin{equation}
  \|g\|_{L^p(\Gamma)} = \|g\circ z\|_{L^p(\R)}
  \quad\text{for all $1\LS p\LS \infty$,}
\label{E:UNO}
\end{equation}
for any $g \in L^p(\Gamma)$. Analogous statements are true for
$L^p_\LOC(\Gamma)$ and $\BMO(\Gamma)$. Writing $z'(s) = z_1'(s)+iz_2(s)$ for
a.e.\ $s\in\R$ we have
\[
  \nu\big(z(s)\big)
    = \nu_1\big(z(s)\big) + i \nu_2\big(z(s)\big)
    = z_2'(s) - i z_1'(s);
\]
see Proposition~\ref{P:PROPS}~\eqref{I:OMEGA}. It follows that \eqref{E:TANGD}
translates into the identity
\[
  \partial_{\tau_{12}} \varphi \big( z(s) \big)
    = z_1'(s) \, \partial_1 \varphi \big( z(s) \big)
      + z_2'(s) \, \partial_2 \varphi \big( z(s) \big)
    = \frac{d}{ds} \varphi\big(z(s)).
\]
Recall that $\partial_{\tau_{12}} \varphi = -\partial_{\tau_{21}} \varphi$ and
$\partial_{\tau_{kk}} \varphi = 0$ for $k=1,2$. Hence, by density,
\begin{equation}
  \|g\|_{\dot{W}^{1,p}(\Gamma)} = \| (g\circ z)'\|_{L^p(\R)}
  \quad\text{for $g \in \dot{W}^{1,p}(\Gamma)$,}
\label{E:DOS}
\end{equation}
for all $1\LS p<\infty$. An analogous statement is true for $\BMO$-derivatives.
For later use, let us record the following observation: If $g \in L^2(\Gamma) \cap
\dot{H}^1(\Gamma)$, then
\begin{equation}
  \|g\|_{L^\infty(\Gamma)}^2
    \LS 2 \|g\|_{L^2(\Gamma)} \|g\|_{\dot{H}^1(\Gamma)} .
\label{E:INT5}
\end{equation}
It is sufficient to prove \eqref{E:INT5} for smooth and compactly supported $g$.
We have
\begin{align*}
  g\big( z(s) \big)^2
    & = \int_{-\infty}^s \frac{d}{dr} \Big( g\big( z(r) \big)^2 \Big) \,dr
      = 2 \int_{-\infty}^s g\big( z(r) \big) \, (g\circ z)'(r) \,dr
\\
    & \vphantom{\int}
      \LS 2 \|g\circ z\|_{L^2(\R)} \|(g\circ z)'\|_{L^2(\R)}
  \quad\text{for all $s\in \R$,}
\end{align*}
because of the Cauchy-Schwarz inequality. Then we use
\eqref{E:UNO},\eqref{E:DOS}.
\end{remark}

For the Mullins-Sekerka evolution we will consider chord-arc domains in $\C$
with sufficiently smooth boundaries. Notice that the presence of a cusp prevents
a curve from being a chord-arc curve; see the argument following
\cite[Definition~2.14]{MarinMartellMitreaMitreaMitrea2022}. The interfaces will
be assumed to have curvature in $L^2(\Gamma)$, hence be $C^2$-rectifiable; see
the discussion at the beginning of Section~\ref{S:GHA}. We will assume that the
parameterization of the chord-arc curve $\Gamma \subseteq \C$ is given as
\begin{equation}
  z(s) = z_0 + \int_0^s e^{ib(t)} \,dt
  \quad\text{for all $s\in\R$,}
\label{E:AKR}
\end{equation}
for a suitable function $b \colon \R \longrightarrow \R$ and $z_0 \in \C$.

\begin{example}
\label{EX:GRAPH}
If $b \in L^\infty(\R)$ with $L^\infty$-norm strictly less than $\pi/2$, then
\eqref{E:AKR} is the arc-length parameterization of a Lipschitz graph in the
plane (which is a chord-arc curve). We refer the reader to
\cite[Proposition~2.12]{MarinMartellMitreaMitreaMitrea2022} for a proof.
\end{example}

More interesting is the case $b \in \BMO(\R)$.

\begin{proposition}
Let $b \in \BMO(\R)$ with norm strictly less than $1$ and
\[
  \lambda \coloneq \frac{\|b\|_{\BMO(\R)}}{1-\|b\|_{\BMO(\R)}} \in [0,\infty).
\]
Define the function $z \colon \R \longrightarrow \C$ by \eqref{E:AKR} and
consider its image $\Gamma \coloneq z(\R)$. Then  $\Gamma$ is a $\lambda$-CAC
containing $z_0$, \eqref{E:AKR} is its arc-length parameterization, and
\[
  \|z'\|_{\BMO(\R)} \LS 2 \|b\|_{\BMO(\R)}.
\]
\end{proposition}

\begin{proof}
We refer the reader to
\cite[Proposition~2.14]{MarinMartellMitreaMitreaMitrea2022}.
\end{proof}

Combining this result with Proposition~\ref{P:PROPS} we observe that if $\Omega$
is the region of $\C$ to the left of $\Gamma$ relative to its arc-length
parameterization, then $\Omega$ is a set of locally finite perimeter and a
simply connected, unbounded, two-sided NTA domain with Ahlfors regular boundary
$\partial\Omega = \Gamma$, for which $\|\nu\|_{\BMO(\Gamma)} \LS 4\lambda$.
Choosing $\|b\|_{\BMO(\R)}$ sufficiently small, we can ensure that $\Omega$ is a
$\delta$-AR domain for suitable $\delta \in (0,1)$ small. The geometric
properties of the interface $\Gamma$ (the Ahlfors regularity constant, for
instance) remain bounded as $\|b\|_{\BMO(\R)} \to 0+$.

\begin{example}
\label{E:SPIRAL}
\begin{figure}
\includegraphics[width=0.4\textwidth]{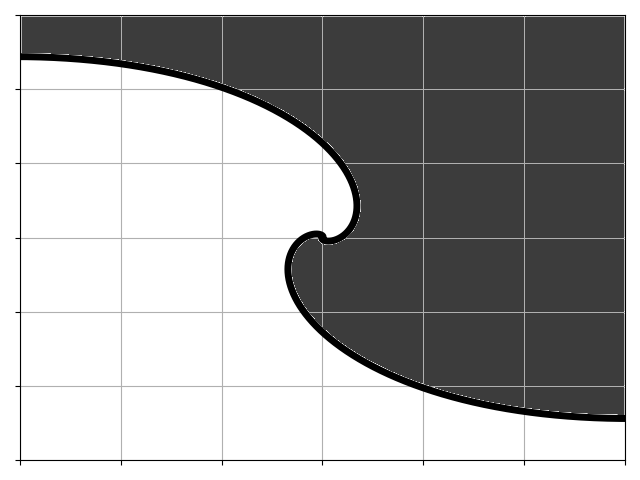}
\includegraphics[width=0.4\textwidth]{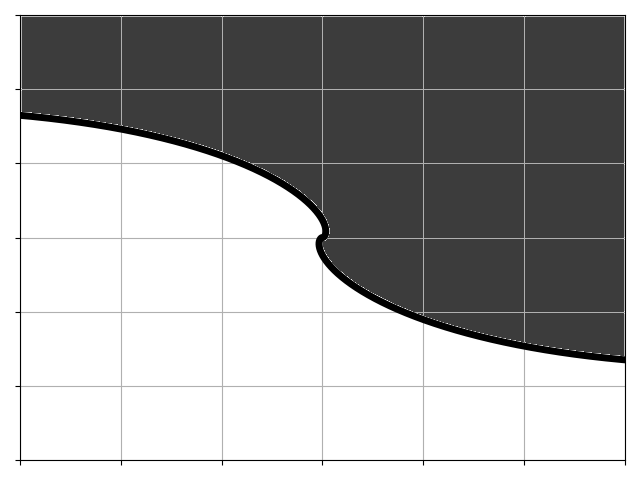}
\caption{Spiral point (left: $\epsilon=0.8$, right: $\epsilon=0.4$). Note that
this is only a very incomplete depiction of the situation. In reality, the
interface spirals \emph{infinitely many times} around the origin. The limited
resolution cannot make this effect visible.}
\label{F:SPIRAL}
\end{figure}
Consider in \eqref{E:AKR} the function $b \colon \R \longrightarrow \R$ defined
by
\begin{equation}
  b(s) \coloneq \epsilon \log|s|
  \quad\text{for $s\in \R\setminus\{0\}$,}
  \quad\text{with $0 < \epsilon < \|\log |\cdot| \|_{\BMO(\R)}^{-1}$.}
\label{E:BS}
\end{equation}
It is known that $\|\log |\cdot| \|_{\BMO(\R)} \LS 3 \log(3/2)$ (see
\cite[p.~55]{Grafakos2004}), so that
\[
  0 < \epsilon < \big( 3\log(3/2) \big)^{-1} \approx 0.8221
\]
is sufficient. In this case, the function $b$ does have small $\BMO(\R)$-norm,
but it is nonetheless surjective onto $\R$. Therefore the outward unit normal
vector to the region $\Omega$ to the left of the interface $\Gamma \coloneq
z(\R)$, given by \eqref{E:NORMZ}, rotates infinitely many times along
$\Gamma$, thus making $z_0$ a spiral point; see Figure~\ref{F:SPIRAL}. Still,
the curve has locally finite length so that the arc-length parameterization is
well-defined. By choosing $\epsilon$ small enough, we can arrange $\Omega$ to be
a $\delta$-AR domain for $\delta \in (0,1)$ arbitrarily small. It is also
possible to construct $b$ such that the resulting curve contains infinitely many
isolated spiral points. We refer the reader to
\cite[Examples~2.7--2.9]{MarinMartellMitreaMitreaMitrea2022} for details. We
conclude that the boundaries of $\delta$-AR domains can be quite complicated. In
particular, this class goes beyond (Lipschitz) graphs as discussed in
Examples~\ref{E:LLGR} and \ref{E:STEEP} above.
\end{example}

\begin{remark}
As mentioned above, we are interested in interfaces $\Gamma$ with curvature in
$L^2(\Gamma)$, which implies that $\Gamma$ is $C^2$-rectifiable. If $z \colon \R
\longrightarrow \C$ given by \eqref{E:AKR} is the arc-length parameterization of
such a curve, then the curvature
\begin{equation}
  \kappa\big(z(s)\big)
    = z''(s) \cdot \nu\big(z(s)\big)
    = -b'(s)
\label{E:CURVI}
\end{equation}
for all $s \in\R$, for which the right-hand side of \eqref{E:CURVI} is defined
(which is a.e.).
\end{remark}


\section{Mullins-Sekerka Evolution}
\label{S:MSE}

In this section, we will derive optimal rates of convergence of unbounded
solutions of the Mullins-Sekerka equation in two dimensions towards the
flat line.


\subsection{Gradient Flow Structure}
\label{S:GFS}

Our solution concept for the Mullins-Sekerka equation uses an intrinsic
description that works directly with the evolving interface, instead of the
material phases. We refer the reader to Le \cite{Le2008} for a similar
treatment, conditional on strong regularity requirements for the interface. In
this paper we make less restrictive assumptions, which appear to be more
natural. We also mention the work by Escher-Matioc-Matioc
\cite{EscherMatiocMatioc2024}, which also studies the Mullins-Sekerka equation
using the method of potentials. The authors consider the interface specified as
the graph of a height function $h \in H^r(\R)$ with $r > 3/2$, which is the
subcritical case. This translates roughly into the normal $\nu \in
\dot{H}^{r-1}(\R)$, which is slightly stronger than our assumption of $\nu$
having regularity $1/2$, which implies that $\nu \in \BMO(\Gamma)$; cf.\ the
proof of Lemma~\ref{L:CONTROL}. Since our focus here is on establishing optimal
convergence rates of Mullins-Sekerka solutions towards the equilibrium, we will
not, in this paper, concern ourself with the question of existence of solutions
to the Mullins-Sekerka equation. This will have to be settled in future work.

We state the following definitions for dimension $n=2$, but they generalize naturally to higher dimensions.

\begin{definition}[Admissible Interface]
\label{D:AI}
A closed set $\Gamma \subseteq \R^2$ is called an \textbf{admissible interface}
if $\Gamma$ is the boundary of a $\delta$-AR domain $\Omega$ (see
Definition~\ref{D:DARD}) with
\[
  \|\nu\|_{\BMO(\Gamma)} \LS \delta < \delta_*
\]
and $\nu$ the outer unit normal to $\Omega$. Here $\delta_* \in (0,1)$ is chosen
sufficiently small to ensure the conclusions of Proposition~\ref{P:TOPPROPS} and
item~\ref{I:TWO} of Theorem~\ref{T:TP}. We further assume that $\Gamma$ admits a
\emph{generalized curvature} $\kappa \in L^2_\LOC(\Gamma)$ in the sense that
\begin{equation}
  \int_{\R^2} \DIV_\tau \zeta \,d\sigma
    = -\int_{\R^2} (\zeta\cdot\nu) \kappa \,d\sigma
  \quad\text{for all $\zeta \in C^1_c(\R^2; \R^2)$},
\label{E:DIVE}
\end{equation}
where $\sigma := \HAUS^1 \lfloor \Gamma$ and $\DIV_\tau$ is the tangential
divergence defined in \eqref{E:DIVATU}.
\end{definition}

We remind the reader that because of Proposition~\ref{P:TOPPROPS}, the interface
$\Gamma$ is unbounded and connected, and the sets $\Omega$ and $\R^2\setminus
\overline{\Omega}$ are unbounded and (simply) connected. The $\BMO$-norm of
$\nu$ controls the Ahlfors constants of $\Gamma$ as estimated in
Proposition~\ref{P:DENS}. The existence of a generalized curvature implies that
the interface $\Gamma$ is $C^2$-rectifiable; see Remark~\ref{R:C2RECT}. Because
of item~\ref{I:TWO} of Theorem~\ref{T:TP} the single layer potential $S \colon
L^2(\Gamma) \longrightarrow \dot{H}^1(\Gamma)$ is surjective up to constants,
with $\dot{H}^1(\Gamma)$ a Banach space when taken as a space of equivalence
classes; see Section~\ref{S:FS}. By injectivity and the open mapping theorem,
the map $S$ is a homeomorphism for the relevant norms.

We adopt throughout the notation \eqref{E:SURFM} and \eqref{E:SETS}.

\begin{definition}[Tangent Space]
\label{D:TS}
Suppose a family $\{\Gamma_t\}_{t\in[0,\infty)}$ satisfies:
\begin{enumerate}[label=(\roman*)]
\item the set $\Gamma_t$ is an admissible interface for a.e.\ $t$;
\item for any $K\subseteq \R^2$ compact and $T<\infty$ there holds
$\sup_{t\in[0,T]} \sigma_t(K) < \infty$;
\item there is a generalized curvature (see Section~\ref{S:FS} for notation)
\begin{equation}
  \kappa \in L^\infty_\LOC\big( [0,\infty), \H_\frac{1}{2} \big)
    \cap L^\infty\big( [0,\infty), \H_{-1} \big).
\label{E:KAPPE}
\end{equation}
\end{enumerate}

Then the \emph{normal speed} corresponding to $\{\Gamma_t\}_{t\in[0,\infty)}$ is
the scalar field
\begin{equation}
  W \in L^\infty_\LOC\big( [0,\infty), \H_{-\frac{1}{2}} \big)
    \cap L^2_\LOC\big( [0,\infty), \H_1 \big)
\label{E:NORME}
\end{equation}
satisfying the inhomogeneous continuity equation
\begin{equation}
\begin{gathered}
  \partial_t \sigma_t
    + \nabla \cdot (\sigma_t \BW_t)
    + g_t \sigma_t = 0
  \quad\text{in duality with $C^1_c([0,\infty) \times \R^n)$,}
\\
  \text{where $\BW_t \coloneq W_t\nu_t$ and $g_t \coloneq \kappa_t W_t$.}
\end{gathered}
\label{E:SIMPLE}
\end{equation}
We will use the notation $\kappa_t \coloneq \kappa(t,\cdot)$ etc.\ for all $t$.
\end{definition}

More explicitly, requirement \eqref{E:SIMPLE} means that for all $\varphi \in
C^1_c([0,\infty) \times \R^n)$
\begin{equation}
  -\int_{\R^n} \varphi(0,\cdot) \,d\sigma_0
    = \int_0^\infty \int_{\R^n} \Big( \partial_t \varphi
      + \BW_t\cdot\nabla \varphi - g_t \varphi \Big) \,d\sigma_t \,dt.
\label{E:AUSG}
\end{equation}
By multiplying $\varphi$ by time-dependent, smooth cut-off functions $\eta_k$
that converge to the characteristic function $\ONE_{[t_1,t_2]}$, then passing to
the limit $k\to\infty$, we have
\begin{equation}
  \left. \int_{\R^2} \varphi(t,\cdot) \,d\sigma_t \right|^{t_2}_{t_1}
    = \int_{t_1}^{t_2} \int_{\R^2} \Big( \partial_t\varphi
      + \BW_t \cdot \nabla \varphi - g_t\varphi \Big) \,d\sigma_t \,dt
\label{E:CONTI}
\end{equation}
for a.e.\ $0\LS t_1 < t_2$. By dividing \eqref{E:CONTI} by $t_2-t_1$ and
letting $t_2 \to t_1$, we obtain a differential expression for the continuity
equation for a.e.\ time $t_1$.

We emphasize that we study the evolution of the interface $\Gamma_t$ by lifting
it to its surface measure $\sigma_t$, which then satisfies \eqref{E:AUSG}. The
integrals in \eqref{E:AUSG} are well-defined because both curvature $\kappa_t$
and normal speed $W_t$ are in $L^2(\Gamma_t)$, as follows from \eqref{E:KAPPE}
and \eqref{E:NORME} by interpolation. Notice that $W_t$ has vanishing mean
because of the first inclusion of \eqref{E:NORME}. It is associated to Dirichlet
data $\phi_t$ satisfying
\[
  \phi_t = SW_t
  \quad\Longleftrightarrow\quad
  W_t = N(\phi_t)
\]
(strictly speaking, $\phi_t$ is an equivalence). Here $S$ denotes the single
layer potential and $N$ is the Dirichlet-to-Neumann map on the interface
$\Gamma_t$; see Section~\ref{S:SIO}. Since the $\dot{H}^1$-norm and
$\NN{\cdot}_1$ are equivalent (see remarks after Definition~\ref{D:AI}), it
follows that $\phi_t \in H^1(\Gamma_t)$ for a.e.\ time $t$. While containment
in $\dot{H}^1(\Gamma_t)$ is natural in view of Theorem~\ref{T:TP}, the fact that
$\phi_t \in L^2(\Gamma_t)$  is a
nontrivial piece of information. Indeed the Poincaré inequality would only give
$L^2_\LOC(\Gamma_t)$ because the interface $\Gamma_t$ has infinite measure.

The map $t \mapsto \sigma_t$ is continuous with respect to the vague topology of
measures. Indeed, because of assumption (ii) in Definition~\ref{D:TS}, it suffices
to test against $C^1_c$-functions, which are dense in $C_c$.

To motivate \eqref{E:KAPPE} we remind the reader that in the Mullins-Sekerka
evolution, interfaces $\Gamma$ move in normal direction with speed $V$
determined by the jump in the normal derivatives of the harmonic extension of
the curvature $\kappa$ to either side of $\Gamma$. The natural space for this
harmonic extension is the homogeneous Sobolev space $\dot{H}^1(\R^2 \setminus
\Gamma)$, which admits traces on $\Gamma$ in the space
$\dot{H}^\frac{1}{2}(\Gamma)$, by classical results in function space theory. It
is therefore natural to consider interfaces with curvature $\kappa \in
\H_\frac{1}{2}$. Notice that here and in the following, we do not expressly
indicate the dependence of function spaces on the interface when it is
clear from the context. Classically, the curvature is the tangential divergence
of the normal field, hence a derivative, which suggests that it is contained in a
negative Sobolev space. We make this precise in Proposition~\ref{P:CURVATURE},
which shows that indeed $\kappa \in \H_{-1}$.

\medskip

We pointed out in the introduction that the Mullins-Sekerka model can formally
be interpreted as a gradient flow of the interfacial area.  Note, however, that
for our application the interfaces $\Gamma$ are unbounded so their surface
measures are infinite. For our analysis we therefore work with the
\textbf{excess energy}
\[
  E(\Gamma) \coloneq \int_{\R^2} \HALF |\nu+\BE_2|^2 \,d\sigma
    = \int_{\R^2} (1+\nu\cdot\BE_2) \,d\sigma,
\]
to which we will also refer as \emph{energy} from now on. By
Proposition~\ref{P:DIFMAS}, it measures the difference between the surface of
the evolving interface (which is infinite) and the surface of the equilibrium
hyperplane (also infinite). We will work with initial interfaces with \emph{finite
excess energy}, and the excess energy will then remain finite for all future times under the flow. Moreover, we
will show in Proposition~\ref{P:CURVATURE} that the excess energy $E(\Gamma)$
controls $\NN{\kappa}_{-1}$.

If the normal speed $V$ is contained in the dual space of $\H_\frac{1}{2}$, then
the pairing of $\kappa$ and $V$, which happens to be precisely the first
variation of the perimeter of $\Omega$ (see Theorem~\ref{T:FVAR} and
\eqref{E:DIVE} with $\zeta = V\nu$), is well-defined. For the Mullins-Sekerka
equation, normal speed $V$ and curvature $\kappa$ are related through the
\emph{Dirichlet-to-Neumann map}, that is $V = N(\kappa)$, which acts like a
first order differential operator. This motivates the first half of
\eqref{E:NORME}. We define the \textbf{dissipation} $D(\Gamma)$ by
\begin{equation}
  D(\Gamma)^\frac{1}{2}
    \coloneq \NN{\kappa}_\frac{1}{2}
    = \NN{N(\kappa)}_{-\frac{1}{2}}.
\label{E:DISSDEF}
\end{equation}
The second half of \eqref{E:NORME} is motivated by Proposition~\ref{P:DECAY},
which shows that the time derivative of the dissipation provides higher
regularity for $V$. It may appear unnecessarily restrictive to assume
\eqref{E:KAPPE}, \eqref{E:NORME} instead of
\[
  \kappa \in L^2_\LOC\big( [0,\infty), \H_\frac{1}{2} \big)
  \quad\text{and}\quad
  V \in L^2_\LOC\big( [0,\infty), \H_{-\frac{1}{2}} \big),
\]
which would be sufficient for the energy-dissipation inequality \eqref{E:ENDI}, but
we are after stronger information: We will use \eqref{E:KAPPE} and \eqref{E:NORME} to identify a regime in which the dissipation \eqref{E:DISSDEF} is
finite and nonincreasing.

The Mullins-Sekerka evolution drives the interface towards a straight line,
which is the equilibrium configuration (the energy minimizer). To measure the
progress in this relaxation, we introduce the \textbf{squared distance}
$H(\Gamma)$ defined by
\begin{equation}
  H(\Gamma) \coloneq \NN{\nu\cdot\pi^\perp_{\BE_2}}_{-\frac{1}{2}}^2,
\label{E:DISTA}
\end{equation}
where $\pi^\perp_{\BE_2}(x) \coloneq (0,x_2)^\mathrm{T}$ for $x\in\R^2$. The
particular form of the squared distance is motivated by the HED-inequality in
Lemma~\ref{L:HED}. It is an intrinsic version of the
$\dot{H}^{-1}(\R^2)$-distance used in
\cite{OttoSchubertWestdickenberg2024} to measure the area between the
Mullins-Sekerka interface and the flat line. If the interface $\Gamma$ is the
graph of a height function $h \colon \R \longrightarrow \R$ and almost flat,
then the squared distance is close to the $\dot{H}^{-\frac{1}{2}}(\R)$-norm of
$h$. Indeed if the normal $\nu$ is basically vertical, then
$\nu\cdot\pi^\perp_{\BE_2} \approx h$ and integration over $\R$ is close to
integration over $\Gamma$. Notice that by our particular choice of distance,
we have prescribed that the energy minimizer must be the line $\{x \in \R^2 \colon
x_2 = 0\}$. Any other \emph{parallel} line interface would result in zero excess
energy but infinite squared distance. (Any \emph{non-parallel} line would have
infinite excess energy.)

\begin{notation}
We will simply use $H$, $E$, and $D$ when $\Gamma$ is clear from the context.
\end{notation}

\begin{definition}[Mullins-Sekerka Solution]
\label{D:SOLUTION}
A family $\{\Gamma_t\}_{t\in[0,\infty)}$ as in Definition~\ref{D:TS} is said to
be a solution to the Mullins-Sekerka equation if
\begin{enumerate}[label=(\roman*)]
\item $\Gamma_0$ has finite excess energy, dissipation, and squared distance;
\item the curvatures $\kappa_t$ and normal speeds $V_t$ satisfy
\begin{equation}
  V_t = N(\kappa_t)
  \quad\text{for a.e.\ $t$,}
\label{E:VNK}
\end{equation}
with $N$ the Dirichlet-to-Neumann map on the interface $\Gamma_t$;
\item for a.e.\ $0 \LS t_1 < t_2$ we have the energy-dissipation inequality
\begin{equation}
  E(\Gamma_{t_2}) + \int_{t_1}^{t_2}
      \Big( \HALF \|\kappa_t\|^2_{\H_\frac{1}{2}}
        + \HALF \|V_t\|^2_{\H_{-\frac{1}{2}}} \Big) \,dt
    \LS E(\Gamma_{t_1}).
\label{E:ENDI}
\end{equation}
\item the maps $t \mapsto D(\Gamma_t)$ and $t \mapsto H(\Gamma_t)$ are
absolutely continuous.
\end{enumerate}
\end{definition}

One can show that the energy-dissipation inequality \eqref{E:ENDI} holds with
equality if the interfaces are sufficiently smooth. Condition (iv) of
Definition~\ref{D:SOLUTION} will enable us to conclude in
Proposition~\ref{P:DECAY} that, if the interfaces are sufficiently flat as
measured by smallness of the dimensionless parameter \eqref{eq:E2D}, the
dissipation is nonincreasing in time, which in turn provides the
$\H_1$-regularity of the normal speed $V_t$ for a.e.\ $t$. Dividing
\eqref{E:ENDI} by $t_2-t_1$, then letting $t_2 \to t_1$, we obtain the
inequality
\[
  \frac{d}{dt} E(\Gamma_t) + D(\Gamma_t) \LS 0
  \quad\text{for a.e.\ $t$,}
\]
where we have used \eqref{E:VNK} and \eqref{E:DISSDEF}.

\begin{remark}
The curvature of the spiraling interface defined by \eqref{E:AKR},\eqref{E:BS} is not
integrable; it behaves like $1/|s|$ with respect to the arc-length
parameterization. Moreover, its excess energy is infinite: Since $b(s)$ grows
unboundedly not only for $|s|\to 0$ but also for $|s|\to\infty$, the interface
spirals around the origin infinitely many times with ever increasing distance.
Every cylinder in the $\BE_2$-direction therefore sees infinitely many pieces of the
interface. This example is perfectly fine from the point of view of a
$\delta$-AR domain with small parameter $\delta$, but it is a degenerate case
for the Mullins-Sekerka evolution. We expect the interface to ultimately
converge to a line, but possibly without ever entering the graph regime.
Figure~\ref{F:SPIRAL} suggests that the interface becomes flatter near the
origin with less pronounced spiraling as $\epsilon$ decreases, while the spiral
arms corresponding to large $|s|$ are pushed away towards infinity.
\end{remark}


\subsection{Static/algebraic estimates}

In this section, we specialize to the case of space dimension $n=2$ and prove
various estimates that will be needed later.

\begin{definition}[Admissible Curve]
A closed set $\Gamma \subseteq \C$ is called an \textbf{admissible curve} if
$\Gamma$ is an admissible interface in the sense of Definition~\ref{D:AI} and
given by the arclength parameterization \eqref{E:AKR} for some $b \colon \R
\longrightarrow \R$. Since $\tau\circ z = e^{ib}$ we refer to $b$ as the
\emph{tangent angle} function. We abbreviate $\|b\|_\infty \coloneq
\|b\|_{L^\infty(\R)}$.
\end{definition}

Our first proposition controls various norms of $\kappa$ in terms of $E$ and
$D$.

\begin{proposition}[Curvature Bounds]
\label{P:CURVATURE}
There is a universal constant $C'\in(0,\infty)$ with the following property.
Suppose that $\Gamma \subseteq \C$ is an admissible curve, for which
$\|b\|_\infty < 2\pi$. Defining
\begin{equation}
    C= C'\cdot C(b)^{1/2},
    \qquad
    C(b):=\frac{\|b\|_\infty^2}{1-\cos(\|b\|_\infty)},
\label{Cb}
\end{equation}
we have that
\begin{equation}
  \NN{\kappa}_{-1} \LS CE^\frac{1}{2},
  \quad
  \NN{\kappa}_{-\frac{1}{2}} \LS C (E^2 D)^\frac{1}{6},
  \quad
  \NN{\kappa}_0 \LS C(E D^2)^\frac{1}{6},
  \quad
  \NN{\kappa}_\frac{1}{2} = D^\frac{1}{2}.
\label{E:GEOM}
\end{equation}
\end{proposition}

\begin{remark}
Note that the constant $C(b)$ is increasing in $\|b\|_\infty$, $C(b) = 2+O\big(
\|b\|_\infty^2 \big)$ as $\|b\|_\infty \downarrow 0$, and if $b$ is bounded
uniformly away from $2\pi$, then the estimates \eqref{E:GEOM} hold with a
universal constant.
\end{remark}

We wish to emphasize that this result holds for certain curves that are
\emph{not graphs}. (The graph case corresponds to $\|b\|_\infty \LS \pi/2$; see
Example~\ref{EX:GRAPH}.) The proof requires two results of independent interest.
The first lemma shows that the Dirichlet-to-Neumann map $N$ indeed acts like a
first order derivative.

\begin{lemma}
\label{L:EQUI}
There exists a universal constant $C
\in (0,\infty)$ with the following property.
For any admissible curve $\Gamma \subseteq \C$ and $f\in \dot{H}^{1}(\Gamma)$, there holds
\begin{equation}
  \left| \|\partial_{\tau_{12}} f\|_{L^2(\Gamma)}
      - \HALF \|N(f)\|_{L^2(\Gamma)} \right|
    \LS C \vartheta\Big( \|\nu\|_{\BMO(\Gamma)} \Big)^{1/2}
      \|N(f)\|_{L^2(\Gamma)},
\label{E:EQNORM}
\end{equation}
where $\vartheta(s)
\coloneq s \log(1/s)$.
\end{lemma}

\begin{proof}
By density, it suffices to consider the case of compactly supported $N(f)$.
By definition of $N$, we have, with $C$ some constant, the identity
\[
  f(x) = \int_\Gamma E(x-y) \big( Nf(y) \big) \,d\sigma(y) + C
  \quad\text{for $\sigma$-a.e.\ $x$,}
\]
with fundamental solution $E(x) = - (2\pi)^{-1} \log|x|$ for all $x \neq 0$; see
Section~\ref{S:SIO}. We define functions $g \coloneq N(f)\circ z$ and $u
\coloneq f\circ z$. Using Remark~\ref{R:INTEG}, we find that
\[
  \partial_{\tau_{12}} f\big( z(t) \big)
    = \partial_t u(t)
    = \frac{1}{2\pi} \int_\R \frac{z'(s)}{z(t)-z(s)} g(s) \,ds
    \eqcolon (C_\R g)(t).
\]
We compare this expression to the Hilbert transform
\[
  (\HIL g)(t) \coloneq \lim_{\epsilon\to 0+} \frac{1}{\pi}
    \int_{|t-s|>\epsilon} \frac{g(s)}{t-s} \,ds
  \quad\text{for $t\in\R$,}
\]
which is well-defined for all $g \in L^2(\R)$. We have the operator bounds
\begin{align}
  \|C_\R - \HALF\HIL\|_{L^2(\Gamma) \longrightarrow L^2(\Gamma)}
    & \LS C \vartheta\Big( \|\nu\|_{\BMO(\Gamma)} \Big)^{1/2},
\label{E:MMMMM}
\\
  \|\HIL\|_{L^2(\Gamma) \longrightarrow L^2(\Gamma)}
    & = \HALF
\nonumber
\end{align}
for a universal constant $C \in (0,\infty)$. For \eqref{E:MMMMM} we refer the
reader to \cite[pp.~247--249 (see in particular
(4.36))]{MarinMartellMitreaMitreaMitrea2022} and Proposition~\ref{P:CADBMO}.
Then \eqref{E:EQNORM} follows from the reverse triangle inequality and another
application of Remark~\ref{R:INTEG}.
\end{proof}

The second lemma provides $L^2(\R)$-control of the tangent angle $b$.

\begin{lemma}
\label{L:ANGLE}
Let $\Gamma \subseteq \C$ be an admissible curve. Suppose that $\|b\|_\infty <
2\pi$. Then
\[
  \|b\|_{L^2(\R)}^2 \LS C(b) E\qquad\text{for}
  \qquad \text{$C(b)$ from \eqref{Cb}}
\]
where $E = E(\Gamma)$ is the excess energy \eqref{E:MEEQ} of $\Gamma$.
\end{lemma}

\begin{proof}
The proof follows from the estimate (recall that $\nu\circ z = -ie^{ib}$)
\[
  \alpha \|b\|_{L^2(\R)}^2
    \LS \int_\R \Big( 1-\cos\big(b(s)\big) \Big) \,ds
    = \int_\Gamma \big( 1+(\nu\cdot \BE_2)\big) \,d\sigma
    = E,
\]
where $\alpha \coloneq \inf_{|\beta| \LS \|b\|_\infty} \big( 1-\cos(\beta)
\big)/\beta^2 = C( b)^{-1} > 0$.
\end{proof}

\begin{proof}[Proof of Proposition~\ref{P:CURVATURE}]
The equality in~\eqref{E:GEOM} is by definition, so it suffices to show the
remaining bounds. Using \eqref{E:UNO}, \eqref{E:DOS}, and \eqref{E:CURVI}, we
can write
\begin{align*}
  \int_\Gamma \kappa g \,d\sigma
    = -\int_\R b'(s) g\big( z(s) \big) \,ds
    & = \int_\R b(s) (g\circ z)'(s) \,ds
\\
    & \LS \|b\|_{L^2(\R)} \|(g\circ z)'\|_{L^2(\R)}
    = \|b\|_{L^2(\R)} \|g\|_{\dot{H}^1(\Gamma)}
\end{align*}
for all functions $g$ obtained by restricting $C^\infty_c(\R^2)$ to $\Gamma$.
Because of \eqref{E:EQNORM}, there exists a constant $C' =
1/2+C\vartheta(\|\nu\|_{\BMO(\Gamma)})^{1/2}$, such that
$\|g\|_{\dot{H}^1(\Gamma)} \LS C' \NN{g}_1$ for all $g$ as above. Since $\Gamma$
is an admissible interface, we know that $\|\nu\|_{\BMO(\Gamma)}$ is bounded by
the universal constant $\delta_*$. By density and duality, we deduce
\[
  \NN{\kappa}_{-1} \LS  C' \|b\|_{L^2(\R)}.
\]
Then we apply Lemma~\ref{L:ANGLE} to obtain the first inequality in
\eqref{E:GEOM} with $C=C' C(b)^{1/2}$. From this we derive the second inequality
in \eqref{E:GEOM}, using \eqref{E:INT4} and \eqref{E:DISSDEF}. For the third, we
use another duality inequality and \eqref{E:DISSDEF}.
\end{proof}

We next turn to two static results that we will later use to show that in the
regime of admissible interfaces and for $\epsilon$ is small enough, the
interface remains admissible for future times. We begin with
Lemma~\ref{L:CONTROL}, which controls the $\BMO$-norm of $\nu$ in terms of
$\epsilon$.

\begin{lemma}
\label{L:CONTROL}
There exists a universal constant $C'\in(0,\infty)$ with the following property.
Let $\Gamma \subseteq \C$ be an admissible curve with $\|b\|_\infty < 2\pi$.
Then
\begin{equation}
  \|\nu\|_{\BMO(\Gamma)}
    \LS C'C(b)^{1/2} \epsilon
  \qquad\text{and}\qquad
  \|b\|_\infty \LS C'C(b)^{1/2} \epsilon,
\label{E:REU}
\end{equation}
for $C(b)$ from  \eqref{Cb}.
\end{lemma}

\begin{proof}
Using Remark~\ref{R:INTEG} and Sobolev embedding, we obtain the inequality
\begin{equation}
  \|\nu\|_{\BMO(\Gamma)}
    = \|\nu\circ z\|_{\BMO(\R)}
    \LS C \|\nu\circ z\|_{\dot{H}^\frac{1}{2}(\R)},
\label{E:SHSE}
\end{equation}
with $C \in (0,\infty)$ some universal constant. For the reader's convenience, we provide
a proof of \eqref{E:SHSE} as the first result in Appendix~\ref{A:PROOF}. Then we apply the chain
rule for fractional derivatives (see
\cite[Proposition~3.1]{ChristWeinstein1991}) to $\nu\circ z = -i e^{ib}$ and use
\eqref{E:CURVI} to estimate
\[
  \|\nu\circ z\|_{\dot{H}^\frac{1}{2}(\R)}
    \LS C \|b\|_{\dot{H}^\frac{1}{2}(\R)}
    = C \|b'\|_{\dot{H}^{-\frac{1}{2}}(\R)}
    = C \|\kappa\circ z\|_{\dot{H}^{-\frac{1}{2}}(\R)}.
\]
The first equality can be checked using the Fourier transform. Finally, we have
\begin{equation}
  \|\kappa\circ z\|_{\dot{H}^{-\frac{1}{2}}(\R)}
    \sim \NN{\kappa}_{-\frac{1}{2}},
\label{E:KAPCE}
\end{equation}
as follows from \eqref{E:UNO}, \eqref{E:DOS}, interpolation, and duality. If
we now assume that $\|b\|_\infty < 2\pi$, then we can utilize
Proposition~\ref{P:CURVATURE} to estimate \eqref{E:KAPCE} by $C' C(b)^{1/2} \epsilon$.

For the second inequality in \eqref{E:REU}, we combine \eqref{E:INT5} with the
estimates of $b, b' \in L^2(\R)$ from Lemma~\ref{L:ANGLE},
\eqref{E:CURVI}, and \eqref{E:GEOM}.
\end{proof}
\begin{remark}
  Notice from the proof above that, independent of any angle condition, $\|\nu\|_{\BMO(\Gamma)}$ is \emph{always}
controlled by $\NN{\kappa}_{-\frac{1}{2}}$. In this sense, the two quantities
play a similar role in measuring flatness. Recall that the $\BMO$-norm of the
normal controls the Ahlfors constant of the interface via
\eqref{E:DEBB}.

The lemma above says that if $\|b\|_\infty$ is
small (less than $2\pi$), then there is a quantitative bound (less than
$C\epsilon$) that makes it even smaller. Since the $\BMO$-norm can be bounded by
twice the $L^\infty$-norm, the first inequality in \eqref{E:REU} can also be
derived from the second one.
\end{remark}
We now show that for $\epsilon$ small enough, the interface is a graph.

\begin{lemma}[Graph Structure]
There exists a constant $\eta_* \in (0,1)$ with the following property. 
If $\Gamma \subseteq \C$ is an admissible curve with $\epsilon\, C(b)^{1/2} < \eta_*$ (recall~\eqref{Cb}), then $\Gamma$ is a graph. Concretely, there exists
a function $h\in \mathrm{Lip}(\R)\cap C^{1,\frac{1}{2}}_{loc}(\R)$ with
\[
    \Gamma = \big\{ \big( x,h(x) \big) \colon x\in \R \big\}.
\]
Moreover, there exists a universal constant $C'\in (0,\infty)$ such that
\[
    \|h_{x}\|_{L^\infty(\R)}
        \LS C'C(b)^{1/2}\epsilon \qquad\text{for $C(b)$ from  \eqref{Cb}}.
\]
\end{lemma}

\begin{proof}
In order to prove that $\Gamma$ has graph structure, we simply check the
conditions in Allard's Regularity Theorem~\ref{T:ALLARD}. Because of the density
control \eqref{E:DEBB}, the first condition is satisfied for any $R>0$ and all
locations $\xi \in \Gamma$ provided $\|\nu\|_{\BMO(\Gamma)}$ is sufficiently
small, which holds if $\epsilon$ is small; see Lemma~\ref{L:CONTROL}.

For the second term in \eqref{E:ALLCOND} we use that
\[
  \NOTE_R(\xi,-\BE_2) \LS 2 \OTE_R(\xi,-\BE_2) \LS R^{-1} E
  \quad\text{for all $R>0$;}
\]
recall \eqref{E:OTE} and \eqref{E:MEEQ}. Notice here that Allard's result
requires the upward normals to $\Gamma$ to be close to $\BE_2$, as expressed by
smallness of the tilt excess. Since in our case, we consider outward unit normal
vectors pointing \emph{downward}, we must consider the tilt excess for the
direction $-\BE_2$, consistent with our definition of the excess energy; see
\eqref{E:MEEQ}. We satisfy the second condition in \eqref{E:ALLCOND} if $R^{-1}
E \LS \eta$.

Finally, estimating the third term in \eqref{E:ALLCOND} with $p=2$ via
\begin{align*}
  R^{1/2} \left( \int_{B_R(\xi)} |\kappa|^2 \,d\sigma \right)^{1/2}
    \LS C' C(\beta)^{1/2} R^{1/2} (E D^2 )^{1/6}
\end{align*}
(see Proposition~\ref{P:CURVATURE}), we satisfy the third condition if $C'C(\beta)^{1/2}
R^{1/2} (E D^2 )^{1/6}  \LS  \eta$.

Making an ansatz $R\coloneq (E/D)^{1/3}$, we see that all three conditions in
\eqref{E:ALLCOND} are satisfied provided that $\epsilon$ is sufficiently small.
By Allard's result, the interface is a graph in the cylinder $B \times \R$,
where $B \subset \R$ is a ball of radius $\gamma R$ with center $\xi_1 \in \R$.
Since $\xi \in \Gamma$ was arbitrary and the estimates above are independent of
$\xi$, we can cover $\R$ by balls and so conclude that $\Gamma$ is a graph
globally.
\end{proof}

\begin{remark}
The fact that an admissible curve $\Gamma$ must be a graph for $\epsilon$ small
also follows from Lemma~\ref{L:CONTROL}. All we need is $C'C(b)^{1/2}\epsilon <
\pi/2$; recall Example~\ref{EX:GRAPH}.
\end{remark}

\begin{remark}
It was shown in a construction in \cite{OttoSchubertWestdickenberg2024} that by
adding the boundary of a very small ball $B$ sufficiently far away from
$\Gamma$, the parameter $\epsilon$ corresponding to the combined interface
$\Gamma \cup \partial B$ can still be arbitrarily small, without the combined
interface being a graph. This is not the situation we are considering here
because smallness of $\|\nu\|_{\BMO(\Gamma)}$ (see Definition~\ref{D:AI}) forces
the interface to be connected.
\end{remark}

We conclude the section with a crucial interpolation inequality  that we will need
in Section~\ref{S:OPTIMAL} for establishing convergence rates for solutions of
the Mullins-Sekerka equation. It bounds the energy $E$  by a suitable
combination of the squared distance $H$ to the equilibrium interface and the
dissipation $D$.

\begin{lemma}[HED-inequality]
\label{L:HED}
Let $\Gamma$ be an admissible curve with $E<\infty$. Then
\begin{equation}
  \bar E \LS \sqrt{HD}.
\label{EQ:HED}
\end{equation}
Here $\bar E$ is the non-oriented excess energy defined in \eqref{E:EBARDEF}.

If additionally $\|b\|_\infty < \pi$, then
\begin{equation}
  E \LS \frac{1}{1+\cos(\|b\|_\infty)} \sqrt{HD}.
\label{E:HEDB}
\end{equation}
If $\Gamma$ is a graph (so that $\|b\|_\infty \LS \pi/2$), then \eqref{E:HEDB}
holds with prefactor $1$.

For $\epsilon$ sufficiently small (recall \eqref{E:REU}), we have
\begin{align}\label{eq:refine_HED}
E\LS \left(\frac{1}{2}+C\epsilon^2\right)\sqrt{HD}
\end{align}
for some universal constant $C$ inependent of $\epsilon$.
\end{lemma}

\begin{proof}
We begin with \eqref{EQ:HED}. There is nothing to prove if $\sqrt{HD}$ is
infinite, so let us assume both $H$ and $D$ are finite. Consider a cut-off $\eta
\in C^\infty_c(\R^2)$ such that
\[
  \eta(\R^2) \subseteq [0,1],
  \quad
  \text{$\eta(x) = 1$ for $|x|\LS 1$,}
  \quad
  \text{$\eta(x) = 0$ for $|x|\GS 2$.}
\]
We define $\eta_R(x) \coloneq \eta(x/R)$ for all $x \in \R^2$ and $R > 0$, then
use the $\R^2$-valued test function $\zeta \coloneq \pi^\perp_{\BE_2} \eta_R$
with $\pi^\perp_{\BE_2}(x) \coloneq (0, x_2)^\mathrm{T}$ in the divergence formula
\eqref{E:DIVE}.

Writing $\partial_k$ for the partial derivative with respect to the $k$th
component, we have
\begin{align*}
  & \DIV_\tau \zeta(x)
    = \Big( \ONE-\nu\otimes\nu \Big) \colon D\zeta(x)
\\
  & \qquad
    = -\Big( \nu_1(x) \partial_1\big( \eta_R(x) \big) \Big) \big( x_2\nu_2(x) \big)
      + \big( 1-\nu_2(x)^2 \big) \Big( x_2 \partial_2 \big( \eta_R(x) \big) + \eta_R(x) \Big).
\end{align*}
Using this in \eqref{E:DIVE} and rearranging terms, we obtain
\begin{equation}
\begin{aligned}
  & \int_{\R^n} \bigg\{ -\kappa(x)\eta_R(x)
      + \nu_1(x) \partial_1 \big( \eta_R(x) \big) \bigg\}
    \Big( x_2\nu_2(x) \Big) \,d\sigma(x)
\\
  & \qquad
    = \int_{\R^n} \Big( 1-\nu_2(x)^2 \Big)
      \bigg\{  \frac{x_2}{R} \partial_2 \eta\Big(\frac{x}{R}\Big)
        + \eta\Big(\frac{x}{R}\Big) \bigg\} \,d\sigma(x).
\end{aligned}
\label{E:DIVID}
\end{equation}
The term in curly brackets on the right-hand side of \eqref{E:DIVID} is bounded
uniformly in both $x$ and $R$, and converges to $1$ pointwise as $R \to \infty$.
Recall that
\[
  \SPT\partial_2 \eta \subset \overline{B_2(0)} \setminus B_1(0).
\]
Since $0 \LS 1-\nu_2^2 = 1-(\nu\cdot \BE_2)^2 \LS 2 (1+\nu\cdot\BE_2)$, the
$L^1(\Gamma)$-norm of $1-\nu_2^2$ is bounded by $2E$. By
dominated convergence, we find that
\[
   \int_{\R^n} \Big( 1-\nu_2(x)^2 \Big)
      \bigg\{  \frac{x_2}{R} \partial_2 \eta\Big(\frac{x}{R}\Big)
        + \eta\Big(\frac{x}{R}\Big) \bigg\} \,d\sigma(x)
      \longrightarrow \NOTE
  \quad\text{as $R \to \infty$.}
\]
Using duality, we can estimate the left-hand side of \eqref{E:DIVID} in absolute
value by
\begin{equation}
  \sqrt{HD} + \sqrt{H} \Big(
      \NN{ \kappa(1-\eta_R) }_\frac{1}{2}
        + \NN{ \nu_1 \partial_1\eta_R }_\frac{1}{2} \Big);
\label{E:ALMOST}
\end{equation}
see \eqref{E:DISTA}. It remains to estimate the two error terms in
\eqref{E:ALMOST}. We have
\begin{equation}
  \NN{ \kappa(1-\eta_R) }_\frac{1}{2}
    \LS \NN{ \kappa(1-\eta_R) }_0^\frac{1}{2} \NN{ \kappa(1-\eta_R) }_1^\frac{1}{2},
\label{E:ERR1}
\end{equation}
by interpolation. Since $\kappa \in L^2(\Gamma)$ and $1-\eta_R \in [0,1]$
everywhere, while converging to zero pointwise as $R\to \infty$, the first
factor on the right-hand side of \eqref{E:ERR1} vanishes in the limit. For the
second term, we recall that because of Lemma~\ref{L:EQUI} the
$\dot{H}^1(\Gamma)$-norm and $\NN{\cdot}_1$ are equivalent. Since for solutions
of the Mullins-Sekerka equation $N(\kappa) = V$ with $V \in L^2(\Gamma)$, the
second factor on the right-hand side of \eqref{E:ERR1} remains bounded. Hence
\eqref{E:ERR1} converges to zero as $R\to \infty$. For the second error term in
\eqref{E:ALMOST} we can argue similarly, using the fact that $\nu_1 \in
L^2(\Gamma)$ because $\nu_1^2 = 1-(\nu\cdot\BE_2)^2$, and thus
$\|\nu_1\|_{L^2(\Gamma)}^2 \LS 2E$. The $\dot{H}^1(\Gamma)$-norm of $\nu_1$ is
controlled by the $L^2(\Gamma)$-norm of $\kappa$; see \eqref{E:CURVI}. Since
$H$ is assumed finite, it follows that
\begin{align}
  \NOTE \LS \sqrt{HD}.
\label{E:EBAR}
\end{align}

Now we would like to improve this to an estimate of $E$. We argue similarly as in the
proof of Lemma~\ref{L:ANGLE}. For $\|b\|_\infty < \pi$ we have that
\[
  \inf_{|\beta| \LS \|b\|_\infty} \Big( 1+\cos(\beta) \Big)
    = 1+\cos(\|b\|_\infty) > 0.
\]
Recalling the expressions for $E$ and $\bar E$ from \eqref{E:MEEQ},
\eqref{E:EBARDEF}, we obtain
\begin{equation}
  \NOTE = \int_{\R} \Big( 1-\cos^2(b(s)\big) \Big) \, ds
    \GS \Big( 1+\cos(\|b\|_\infty) \Big) E.
\label{E:ETOE}
\end{equation}
This together with \eqref{E:EBAR} gives the HED-inequality \eqref{E:HEDB}. The
more refined estimate \eqref{eq:refine_HED} follows by Taylor expansion of
$\cos$ and \eqref{E:REU}.
\end{proof}

\begin{remark}
We wish to emphasize that Proposition~\ref{P:CURVATURE} and Lemma~\ref{L:HED}
remain true beyond the graph regime; certain overhanging interfaces are still
admissible. The proofs break down, however, if along $\Gamma$ the tangent
makes a full $2\pi$-turn as it did in Example~\ref{E:SPIRAL}.
\end{remark}


\subsection{Differential estimates}

In this section, we will prove that the dissipation $D$ in the Mullins-Sekerka
evolution is decreasing in time if the solution is sufficiently close to the
equilibrium, and hence that the map $t\mapsto E(t)$ is convex. To estimate the
time derivative of $D$, we will make use of the observation in
\cite{OttoWestdickenberg2005}, that convexity of the Hessian of a functional can
be probed by studying the evolution of first order perturbations along the
corresponding gradient flow. Indeed suppose that $t \mapsto \gamma_t$ is a
solution of the gradient flow equation $\frac{d\gamma_t}{dt} = -\nabla
E(\gamma_t)$ on a Riemannian manifold $\MANI$, for some energy $E$. We denote by
$(\cdot,\cdot)_\gamma$ the Riemannian metric at a point $\gamma \in \MANI$, with
$\|\cdot\|_\gamma$ the induced norm. If $W_t = \frac{d\gamma_t}{ds}\Big|_{s=0} $ is a
perturbation of $\gamma_t$, then
\begin{equation}
\begin{aligned}
  -\frac{d}{dt} \HALF\|W_t\|_{\gamma_t}^2
    & = -\left\langle W_t, \frac{\nabla}{dt} \frac{d\gamma_t}{ds} \right\rangle_{\gamma_t}
\\
    & = -\left\langle W_t, \frac{\nabla}{ds} \frac{d\gamma_t}{dt} \right\rangle_{\gamma_t}
      = \left\langle W_t, D^2 E(\gamma_t) W_t \right\rangle_{\gamma_t}.
\end{aligned}
\label{E:SCHEME}
\end{equation}
Here $\frac{\nabla}{dt}$ denotes the \emph{covariant derivative}. It is needed
because $W_t$ is a tangent vector. The second equality in \eqref{E:SCHEME}
follows from the fact that second mixed derivatives of $\gamma_t$  commute.
The third equality follows by inserting the gradient flow equation and using the
chain rule. Here $D^2 E$ represents the Hessian of $E$.

Naturally if the energy is convex, the Hessian $D^2E(\gamma_t)$ is positive
semidefinite and $\frac{d}{dt} D(\gamma_t) \LS 0$. Without a convex energy, one can
still consider perturbations in the direction $V_t = \frac{d\gamma_t}{dt}$ of
the gradient flow, so that
\[
  -\frac{d}{dt} E(\gamma_t) = D(\gamma_t) = \|V_t\|_{\gamma_t}^2,
\]
and probe convexity along the solution. Our goal is to carry this out.

We like to think of solutions to the Mullins-Sekerka equation as curves taking
values in the manifold $\MANI$ of admissible interfaces $\Gamma \subseteq \R^2$.
Since the deformation of curves only happens in normal directions, tangent
vectors to $\MANI$ at some point $\Gamma \in \MANI$ can therefore be identified
with normal speeds $V$ that are obtained via the Dirichlet-to-Neumann map $N$
from suitable Dirichlet data $\phi$. Both $V$ and $\phi$ are scalar functions
defined on $\Gamma$. Since the natural space for harmonic extensions of $\phi$
is the homogeneous Sobolev space $\dot{H}^1(\R^2 \setminus \Gamma)$, we expect
\[
  \phi \in \H_\frac{1}{2}
  \quad\text{and}\quad
  V \in \H_{-\frac{1}{2}}.
\]
A curve $t \mapsto \Gamma_t \in \MANI$ is a \textbf{gradient flow} of the
(modified) area functional provided the normal speed satisfies $V_t =
N(\kappa_t)$, with $\kappa_t$ the curvature of the interface $\Gamma_t$. We say
that $\kappa_t$ is indeed the ``gradient'' of the area functional because of the
first variation of perimeter formula \eqref{E:FVAR}. The Riemannian metric on
$T_\Gamma \MANI$ is obtained as the inner product inherited from the Hilbert
space $\H_{-\frac{1}{2}}$, i.e., we define
\[
  (V,W)_\Gamma \coloneq (V,W)_{\H_{-\frac{1}{2}}}
  \quad\text{for $V,W \in T_\Gamma \MANI$,}
\]
where $\|\cdot\|_\Gamma$ is the induced norm; see Definition~\ref{D:TS} for
more.

In this paper, we make the \emph{stronger assumption} of normal speeds $W \in
L^2(\Gamma)$.

The Neumann-to-Dirichlet map maps normal speeds $W$ into Dirichlet data $\phi$
such that $W = N(\phi)$. Since $\phi$ is only defined up to constants, the
Neumann-to-Dirichlet map maps $W$ into equivalence classes of functions modulo
constants. One possible representative in the image of $W$ under the
Neumann-to-Dirichlet map is obtained by the single layer potential operator $S$
discussed in Section~\ref{S:SIO}. We have
\begin{equation}
  \|W\|_\Gamma^2 = \int_\Gamma W(SW) \,d\sigma
  \quad\text{for $W \in T_\Gamma \MANI \cap L^2(\Gamma)$.}
\label{E:METR}
\end{equation}
Indeed \eqref{E:METR} would not change if $SW$ were replaced by other Dirichlet
data $\phi$ with $W = N(\phi)$ because $W \in \H_{-\frac{1}{2}}$ implies that
$V$ has vanishing integral.

We compute the $t$-derivative of \eqref{E:METR}. To simplify notation, we will
occasionally not mark the time dependence of objects.

\begin{lemma}[Hessian]
\label{L:HESSIAN}
Let $(s,t) \mapsto \Gamma_{s,t}$ be a two-parameter family of admissible
interfaces (cf.\ Definition~\ref{D:AI}) that satisfy the conditions in Definition~\ref{D:TS}. We assume that the surface measures
$\sigma_{s,t} \coloneq \HAUS^{1} \lfloor \Gamma_{s,t}$ depend continuously on
$(s,t)$ in the vague topology (testing against $C_c$-functions). We assume that
the sets $\Omega_{s,t}$ evolve in the normal direction with normal speed
$V_{s,t} \in L^2(\Gamma_{s,t})$ with respect to $t$, and $W_{s,t} \in
L^2(\Gamma_{s,t})$ with respect to $s$. Moreover, $t\mapsto \Gamma_{0,t}$ is a Mullins-Sekerka solution in the sense of Definition \ref{D:SOLUTION}.
Assume that $t\mapsto \|W_{0,t}\|_{\Gamma}^2$ is differentiable at $t=0$. Then
\begin{equation}
 - \left.\frac{d}{dt}\right|_{s,t=0} \HALF \|W\|^2_\Gamma
    = \int_\Gamma \bigg(
        \Big( |\nabla_\tau W|^2 - \kappa^2 W^2 \Big)
          - \Big( M(V) W^2 \Big) \bigg) \,d\sigma,
\label{E:HES}
\end{equation}
where $\kappa$ is the curvature of the interface and the operator $M$ is defined
in \eqref{E:DOUBLE}.
\end{lemma}

We refer the reader to Appendix~\ref{A:PROOF} for the proof of
Lemma~\ref{L:HESSIAN}.

We are now ready for the main result of this section.

\begin{proposition}[Dissipation of Mullins-Sekerka]
\label{P:DECAY}
Let $\{\Gamma_t\}_{t}$ be a Mullins-Sekerka solution in the sense of Definition
\ref{D:SOLUTION}. Given $\alpha \in (0,1)$ there exists $\epsilon_* \in (0,1)$
such that for all times at which $\epsilon < \epsilon_*$ and
$\|b\|_{L^\infty(\R)}\LS \pi$, $D$ satisfies
\begin{equation}
  -\frac{d}{dt} D
    \GS (1-\alpha) \int_\Gamma |\nabla_\tau V|^2 \,d\sigma.
\label{E:RESULT1}
\end{equation}
\end{proposition}

\begin{corollary}[Bounds]
\label{cor:epsilon}
Suppose that $\{\Gamma_t\}_{t\in [0,\infty)}$ is a Mullins-Sekerka solution in
the sense of Definition \ref{D:SOLUTION}. There exists a universal constant
$\epsilon_*\in (0,1)$ with the following property: If $\epsilon_0 \LS
\epsilon_*$ and $\|b_0\|_\infty \LS \pi$, then
\begin{itemize}
\item [(i)] $\epsilon(t)\LS \epsilon_0$ for all $t\GS 0$, and
\item [(ii)] for almost every $t\GS 0$, $-\frac{d}{dt}D$ is finite and $V_t\in \H_1$.
\end{itemize}
\end{corollary}

\begin{proof}
Absolute continuity of $\epsilon$ follows from that of $E$ and $D$. The
monotonicity of $\epsilon$ follows from $\dot{\epsilon}\LS 0$ a.e., which
follows from $\dot{E}\LS 0$, $\dot{D}\LS 0$ a.e. (from
Proposition~\ref{P:DECAY}) and nonnegativity of $E$, $D$. We maintain control of
the tangent angle in time because of \eqref{E:REU}. From $D\GS 0$
and~\eqref{E:RESULT1} it then follows that $-\frac{d}{dt} D$ is finite and  $V_t \in
\H_1$ a.e.
\end{proof}

\begin{proof}[Proof of Proposition~\ref{P:DECAY}]
For a solution $t \mapsto \Gamma_t$ of the Mullins-Sekerka equation in the sense of
Definition \ref{D:SOLUTION}, we consider variations in the direction of the
gradient flow, i.e., with $W_t = V_t$. According to our solution concept (cf.
Definition~\ref{D:SOLUTION}),  the dissipation along the gradient flow satisfies
\[
  -\frac{d}{dt} E(\Gamma_t) \GS D(\Gamma_t) = \NN{V_t}_{-\frac{1}{2}}^2.
\]
We split the proof in two steps.

\medskip

\textbf{Step~1.} Let us first collect some estimates. By assumption, we have $V
\in \H_{-\frac{1}{2}} \cap \H_1$ and $\kappa \in \H_{-1} \cap \H_\frac{1}{2}$.
Moreover, since $V = N(\kappa)$ there holds
\begin{equation}
  \NN{V}_{-\frac{1}{2}} = \NN{\kappa}_\frac{1}{2} 
  \quad\text{and}\quad
  \NN{V}_0 = \NN{\kappa}_1
\label{E:SAME2}
\end{equation}
(recall Notation~\ref{N:FS}). We  now use \eqref{E:INT2} to deduce
\begin{equation}
  \NN{\kappa}_0
    \LS \NN{\kappa}_{-\frac{1}{2}}^\frac{2}{3} \NN{\kappa}_1^\frac{1}{3}
    = \NN{\kappa}_{-\frac{1}{2}}^\frac{2}{3} \NN{V}_0^\frac{1}{3}.
\label{E:OP1}
\end{equation}
Combining \eqref{E:INT2} with \eqref{E:SAME2}, \eqref{E:INTP1}, and
\eqref{E:OP1}, we get
\[
  \NN{V}_0
    \LS \NN{V}_{-\frac{1}{2}}^\frac{2}{3} \NN{V}_1^\frac{1}{3}
    \LS \left( \NN{\kappa}_0^\frac{1}{2} \NN{V}_0^\frac{1}{2} \right)^\frac{2}{3}
      \NN{V}_1^\frac{1}{3}
    \LS \left( \NN{\kappa}_{-\frac{1}{2}}^\frac{1}{3} \NN{V}_0^\frac{2}{3} \right)^\frac{2}{3}
      \NN{V}_1^\frac{1}{3},
\]
which after simplification results in
\begin{equation}
  \NN{V}_0
    \LS \NN{\kappa}_{-\frac{1}{2}}^\frac{2}{5} \NN{V}_1^\frac{3}{5}.
\label{E:OP2}
\end{equation}

\textbf{Step~2.}
We use \eqref{E:INT5} with \eqref{E:EQNORM}, then \eqref{E:SAME2},
\eqref{E:OP1}, and \eqref{E:OP2} to obtain
\begin{align}
  \int_\Gamma \kappa^2 V^2 \,d\sigma
    & \LS \|\kappa\|_{L^\infty(\Gamma)}^2 \NN{V}_0^2
      \LS C_*\Big( \NN{\kappa}_0 \NN{\kappa}_1 \Big) \NN{V}_0^2
\nonumber
\\
    & \LS C_*\left( \NN{\kappa}_{-\frac{1}{2}}^\frac{2}{3} \NN{V}_0^{\frac{4}{3}} \right) \NN{V}_0^2
      \LS C_*\NN{\kappa}_{-\frac{1}{2}}^2 \NN{V}_1^2.
\label{E:PART1}
\end{align}
Here $C_*=1/2+C\vartheta(\|\nu\|_{\BMO(\Gamma)})^{1/2}$ is the constant from
\eqref{E:EQNORM}.

Similarly, we use \eqref{E:SMALL} and \eqref{E:INT5}, then \eqref{E:OP1} and
\eqref{E:OP2}, to get
\begin{align}
  \int_\Gamma |M(V)| V^2 \,d\sigma
    & \LS \NN{M(V)}_0 \Big( \NN{V}_0 \|V\|_{L^\infty(\Gamma)} \Big)
\nonumber
\\
    & \LS C_\nu \NN{V}_0 \Big( \NN{V}_0^\frac{3}{2} \NN{V}_1^\frac{1}{2} \Big)
      \LS C_\nu \NN{\kappa}_{-\frac{1}{2}} \NN{V}_1^2,
\label{E:PART2}
\end{align}
with $C_\nu = C \vartheta( \|\nu\|_{\BMO(\Gamma)} )$; see \eqref{E:SMALL}.

Recall that $\NN{V}_1$ and $\|V\|_{\dot{H}^1(\Gamma)}$ are equivalent (cf.\
Lemma~\ref{L:EQUI}). Both $\NN{\kappa}_{-\frac{1}{2}}$ and
$\|\nu\|_{\BMO(\Gamma)}$ can be controlled by $\epsilon$ when
$\|b\|_{L^\infty(\R)}\LS \pi$; recall \eqref{E:GEOM} and \eqref{E:REU}. Then
\eqref{E:HES}, \eqref{E:PART1}, and \eqref{E:PART2} imply the lower bound
\eqref{E:RESULT1}.
\end{proof}


\subsection{Optimal Convergence Rates}
\label{S:OPTIMAL}

In this section, we will prove our main result, Theorem~\ref{T:OPTIMAL}. We start with a lemma that
connects the change of the squared distance $H$ (recall \eqref{E:DISTA}) along
the gradient flow to the energy $E$. This is a relaxation of \eqref{input3}.
Constants are universal and, in particular, independent of $\epsilon$, $E$, $D$,
and $H$. The subscript $0$ marks the value of a quantity at
initial time $t=0$.

\begin{lemma}
There exist $\epsilon'>0$  with the following property. For any Mullins-Sekerka solution $\{\Gamma_t\}_{t\in [0,\infty)}$ in the sense of
Definition \ref{D:SOLUTION}, if
$\epsilon_0\LS \epsilon'$, then the following holds true:
\begin{alignat}{2}
  \frac{1}{2} \frac{d}{dt} H(t) + \Big( 2-C \epsilon_0^2 |\log(\epsilon_0)|^2 \Big) E
    & \LS C \epsilon_0^2 \frac{HD}{E}
    && \quad\text{for a.e.\ $t\GS 0$,}
\label{eq:dtH}
\\
  H(t)
    & \LS \Big( 1+C\epsilon_0^2 \Big) H_0
    && \quad\text{for all $t\GS 0$.}
\label{eq:Hbd}
\end{alignat}
\end{lemma}

\begin{proof}
We choose $\epsilon'>0$ sufficiently small such that, according to
Corollary~\ref{cor:epsilon}, the map $t\mapsto \epsilon(t)$ is decreasing for
$\epsilon_0\LS\epsilon'$. We denote $W_t\coloneq \nu_t \cdot \pi^\perp_{\BE_2}$.
Then
\begin{equation}
  H(t) = \NN{W_t}_{-\frac{1}{2}}^2.
\label{E:HEQW}
\end{equation}
For simplicity of notation we will write $\Gamma\coloneq\Gamma_t$,
$\epsilon\coloneq\epsilon(t)$ and $W\coloneq W_t$.

\medskip

\textbf{Step~1.} We turn first to \eqref{eq:dtH}. We will show the stronger estimate
\begin{align}
    \frac{1}{2} \frac{d}{dt} H(t) + \Big( 2-C \epsilon_0^2 |\log(\epsilon_0)|^2 \Big) E
    & \LS C H^\frac{2}{5} E^\frac{13}{15} D^\frac{11}{15},\label{thirteen}
\end{align}
from which~\eqref{eq:dtH} follows via
\begin{align}
E\left(H^\frac{2}{5} E^\frac{13}{15} D^\frac{11}{15}\right)= (E^2D)^{\frac{1}{3}} (H^2 E^6 D^2)^{\frac{1}{5}} \stackrel{\eqref{eq:E2D},\eqref{eq:refine_HED}}{\LS} \epsilon^2 HD\LS \epsilon_0^2 HD.\notag
\end{align}

For the moment we assume that $W_t\in L^\infty_{\LOC}([0,\infty), \H_{1})$ so that Lemma \ref{L:HESSIAN} is applicable. Later in (ii) below we will show that this condition is in fact satisfied.  It follows from \eqref{E:HES}
that
\begin{align*}
  \frac{1}{2} \frac{d}{dt} H(t)
    = -\int_\Gamma \bigg( \Big( |\nabla_\tau W|^2 - \kappa^2 W^2 \Big)
      - M(V) W^2 \bigg) \,d\sigma.
\end{align*}
We will estimate the three terms in the integral on the right-hand side
separately. Using the arclength parameterization \eqref{E:AKR} and
\eqref{E:NORMZ}, we can write
\begin{gather}
  W \circ z
    = (\nu \cdot \pi^\perp_{\BE_2}) \circ z
    = -\cos(b) z_2,
\label{woz}
\\
  (\nabla_\tau W) \circ z
    = (W\circ z)'
    = -\cos(b) \sin(b) + \sin(b) b' z_2,
\label{E:WCZ2}
\end{gather}
where we have used the fact that
\begin{align}
  z_2' = \sin(b).\label{sinz}
\end{align}

\begin{enumerate}[label=(\roman*)]
\item \emph{A lower bound of  $\displaystyle\int_{\Gamma}|\nabla_\tau W|^2\,d\sigma$.}
\end{enumerate}

Because of \eqref{E:DOS} we can write
\begin{equation}
  \int_{\Gamma} |\nabla_\tau W|^2 \,d\sigma = \int_\R |(W\circ z)'(s)|^2 \,ds.
\label{E:INSH}
\end{equation}
Expanding the square of \eqref{E:WCZ2}, we obtain
\[
  |(W\circ z)'|^2
    = \Big( \cos(b) \sin(b) \Big)^2
      -2 \cos(b)  \sin^2(b)  b' z_2
      + \Big( \sin(b)b' z_2 \Big)^2.
\]
For the mixed term, we integrate by parts to find that
\[
  -2 \int_\R \cos(b) \sin^2(b)  b' z_2 \,ds
    = -\frac{2}{3} \int_\R \Big( \sin^3(b) \Big)' z_2 \,ds
    = \frac{2}{3} \int_\R \sin^4(b)  \,ds.
\]
The integral \eqref{E:INSH} therefore consists of three nonnegative terms.
Notice, however, that terms proportional to $\sin(b)$ are higher order because
the tangent angle converges to zero for $t\to\infty$; see \eqref{E:REU}. It is
sufficient to only keep the leading order term to obtain (for $\epsilon$ small
enough) the estimate
\[
  \int_{\Gamma}|\nabla_\tau W|^2\,d\sigma
    \GS \int_\R \Big( \cos(b) \Big)^2 \Big( 1-\cos^2(b) \Big) \,ds
    \GS \Big( 2-C\epsilon \Big) E,
\]
where we have used the Taylor expansion of cosine, \eqref{E:ETOE}, and
\eqref{E:REU}.

\begin{enumerate}[label=(\roman*), resume]
\item \emph{Estimate of $\displaystyle\int_{\Gamma} \kappa^2 W^2 \,d\sigma$ and an upper bound of $\displaystyle\int_{\Gamma}|\nabla_\tau W|^2 \, d\sigma$.}
\end{enumerate}

We first observe that
\begin{align*}
  \|z_2\|_{L^2(\R)}
    &\LS \|\cos(b)z_2\|_{L^2(\R)} + \|(1-\cos(b)) z_2\|_{L^2(\R)}
\\
    & \overset{\eqref{woz}}\LS \|W\circ z\|_{L^2(\R)} + \|z_2\|_{L^\infty(\R)} (2E)^\frac{1}{2}.
\end{align*}
With this and \eqref{E:INT5}, we then estimate
\begin{align}
  \|z_2\|^2_{L^\infty(\R)}
    & \overset{\eqref{sinz}}\LS 2 \|z_2\|_{L^2(\R)} \|\sin(b)\|_{L^2(\R)}
\nonumber
\\
    & \LS 2 \Big( \|W\circ z\|_{L^2(\R)} + \|z_2\|_{L^\infty(\R)} (2E)^\frac{1}{2} \Big)
      (2E)^\frac{1}{2},
\label{E:Z2PREL}
\end{align}
where for the last inequality we have used the fact that
\[
  \int_\R \sin(b)^2 \,ds
    = \int_\R \Big( 1-\cos(b)^2 \Big) \,ds
    \LS 2 \int_\R \Big( 1-\cos(b) \Big) \,ds
    = 2 E.
\]
With \eqref{E:UNO} and \eqref{E:INT2}, we can now write (recall \eqref{E:HEQW})
\[
  \|(W\circ z)\|_{L^2(\R)}
    = \NN{W}_0
    \LS \NN{W}_{-\frac{1}{2}}^\frac{2}{3} \NN{W}_1^\frac{1}{3}
    = H^\frac{1}{3} \NN{W}_1^\frac{1}{3},
\]
which we insert into \eqref{E:Z2PREL}. By Young's inequality (with
$(p,q)=(2,2)$), there holds
\[
  4 \|z_2\|_{L^\infty(\R)} E
    \LS \HALF \|z_2\|^2_{L^\infty(\R)} + 8 E^2,
\]
Absorbing $\HALF \|z_2\|^2_{L^\infty(\R)}$ in \eqref{E:Z2PREL} into the
left-hand side, we obtain
\begin{equation}
  \ESUP_\Gamma |W|^2
    \LS \|z_2\|^2_{L^\infty(\R)}
    \LS 16 \Big( H^\frac{1}{3} E^\frac{1}{2} \NN{W}_1^\frac{1}{3} + E^2 \Big).
\label{E:LINF}
\end{equation}
Since $\|\nu\|_{\BMO(\Gamma)}$ is assumed sufficiently small, which by
Lemma~\ref{L:CONTROL} can be achieved by choosing $\epsilon$ small, the
$\dot{H}^1(\Gamma)$-norm and $\|\cdot\|_1$ are equivalent; see
Lemma~\ref{L:EQUI}. With \eqref{E:GEOM}, \eqref{eq:E2D}, and Young's inequality
(with $(p,q) = (\frac{6}{5},6)$), we obtain
\begin{align*}
  \int_{\Gamma} \kappa^2 W^2 \,d\sigma\LS  \NN{\kappa}_{0}^2 \|z_2\|_{L^\infty(\R)}^2
    & \LS C (ED^2)^\frac{1}{3}
      \Big( H^\frac{1}{3} E^\frac{1}{2} \NN{W}_1^\frac{1}{3} + E^2 \Big)
\\
    & = C \Big( H^\frac{1}{3} E^\frac{13}{18} D^\frac{11}{18}
      \big( \epsilon \NN{W}_1 \big)^\frac{1}{3} + \epsilon^4 E \Big)
\\
    & \LS \epsilon^2 \int_\Gamma |\nabla_\tau W|^2 \,d\sigma
      + C \Big( H^\frac{2}{5} E^\frac{13}{15} D^\frac{11}{15} + \epsilon^4 E \Big).
\end{align*}

The above inequality yields an upper bound for $\int_{\Gamma}|\nabla_\tau W|^2 \, d\sigma$. Indeed, it follows from (i) that
\begin{align*}
\int_{\Gamma}|\nabla_\tau W|^2 \, d\sigma &\LS 4	E + \int_{\Gamma}\left(\sin(b)b'z_2\right)^2 \, ds\LS 4E + \NN{\kappa}_0^2 \|z_2\|_{L^\infty(\R)}^2\\
&\LS \epsilon^2 \int_\Gamma |\nabla_\tau W|^2 \,d\sigma
      + C \Big( H^\frac{2}{5} E^\frac{13}{15} D^\frac{11}{15} + E \Big).
\end{align*}
Thus, absorbing the $\epsilon^2$ term in the left-hand side and by a standard approximation argument, we have
\begin{align*}
\int_{\Gamma}|\nabla_\tau W|^2 \, d\sigma &\LS
     C \Big( H^\frac{2}{5} E^\frac{13}{15} D^\frac{11}{15} + E \Big).
\end{align*}
Since by assumption $t\mapsto H(t)$ and $t\mapsto D(t)$ are absolutely continuous and $t\mapsto E(t)$ is monotonically decreasing, we thus obtain $W\in L^\infty([0,\infty); \H_{1})$.

\begin{enumerate}[label=(\roman*), resume]
\item \emph{Estimate of $\displaystyle\int_{\Gamma} M(V) W^2 \,d\sigma$.}
\end{enumerate}

By an approximation argument and Green's second identity, we have
\[
  \int_\Gamma M(V) W^2 \,d\sigma
    = \int_\Gamma M\big( N(\kappa) \big) W^2 \,d\sigma
    = \int_\Gamma \kappa M\big( N(W^2) \big) \,d\sigma,
\]
which we can estimate using the Cauchy-Schwarz inequality. The $L^2(\Gamma)$-norm of
$\kappa$ can be controlled by \eqref{E:GEOM}, and the operator $M$ is
$L^2(\Gamma)$-continuous with norm bounded by $C
\vartheta(\|\nu\|_{\BMO(\Gamma)})$; see \eqref{E:SMALL}. It remains to estimate
the $L^2(\Gamma)$-norm of $N(W^2)$. By Lemma~\ref{L:EQUI}, for
$\|\nu\|_{\BMO(\Gamma)}$ sufficiently small, which can be controlled in terms of
$\epsilon$ using \eqref{E:REU}, the $\dot{H}^1(\Gamma)$-norm and
$\|N(\cdot)\|_{L^2(\Gamma)}$ are equivalent. We estimate
\begin{align*}
  \|N(W^2)\|_{L^2(\Gamma)}
    & \LS C \|W\|_{L^\infty(\Gamma)}
      \left( \int_\Gamma |\nabla_\tau W|^2 \,d\sigma \right)^\frac{1}{2}
\\
    & \LS C \left( H^\frac{1}{6} E^\frac{1}{4}
        \left( \int_\Gamma |\nabla_\tau W|^2 \,d\sigma \right)^\frac{1}{12} + E \right)
          \left( \int_\Gamma |\nabla_\tau W|^2 \,d\sigma \right)^\frac{1}{2},
\end{align*}
where we have used \eqref{E:LINF} and the fact that $\sqrt{a+b} \LS \sqrt{a} +
\sqrt{b}$ for any $a,b \GS 0$. Combining all estimates, we arrive at
\begin{align*}
  & \int_\Gamma \big|\kappa M\big( N(W^2) \big)\big| \,d\sigma
\\
  & \qquad
    \LS C \vartheta(\epsilon) (ED^2)^\frac{1}{6}
      \left( H^\frac{1}{6} E^\frac{1}{4}
        \left( \int_\Gamma |\nabla_\tau W|^2 \,d\sigma \right)^\frac{1}{12} + E \right)
        \left( \int_\Gamma |\nabla_\tau W|^2 \,d\sigma \right)^\frac{1}{2}
\\
  & \qquad
    \LS C \vartheta(\epsilon) \left( H^\frac{1}{6} E^\frac{5}{12} D^\frac{1}{3}
        \left( \int_\Gamma |\nabla_\tau W|^2 \,d\sigma \right)^\frac{7}{12}
      + E^\frac{7}{6} D^\frac{1}{3}
        \left( \int_\Gamma |\nabla_\tau W|^2 \,d\sigma \right)^\frac{1}{2} \right).
\end{align*}
We now use Young's inequality with $(p,q) = (\frac{12}{7},
\frac{12}{5})$ to obtain
\begin{align*}
  & C \vartheta(\epsilon) H^\frac{1}{6} E^\frac{5}{12} D^\frac{1}{3}
      \left( \int_\Gamma |\nabla_\tau W|^2 \,d\sigma \right)^\frac{7}{12}
\\
  & \qquad
    = C H^\frac{1}{6} E^\frac{13}{36} D^\frac{11}{36}
      \left( \epsilon^2 |\log(\epsilon)|^\frac{12}{7}
        \int_\Gamma |\nabla_\tau W|^2 \,d\sigma \right)^\frac{7}{12}
\\
  & \qquad
    \LS \frac{1}{2} \epsilon^2 |\log(\epsilon)|^\frac{12}{7}
        \int_\Gamma |\nabla_\tau W|^2 \,d\sigma
      + C H^\frac{2}{5} E^\frac{13}{15} D^\frac{11}{15}.
\end{align*}
A second application of Young's inequality with $(p,q) = (2,2)$ yields
\begin{align*}
  & C \vartheta(\epsilon) E^\frac{7}{6} D^\frac{1}{3}
      \left( \int_\Gamma |\nabla_\tau W|^2 \,d\sigma \right)^\frac{1}{2}
\\
  & \qquad
    = C \epsilon^2 E^\frac{1}{2}
      \left( \epsilon^2 |\log(\epsilon)|^2
        \int_\Gamma |\nabla_\tau W|^2 \,d\sigma \right)^\frac{1}{2}
\\
  & \qquad
    \LS \frac{1}{2} \epsilon^2 |\log(\epsilon)|^2
        \int_\Gamma |\nabla_\tau W|^2 \,d\sigma
      + C \epsilon^4 E.
\end{align*}

Collecting all terms, we arrive at
\[
  \frac{1}{2} \frac{d}{dt} H(t)
    \LS -\bigg( \Big( 1-2 \epsilon^2 |\log(\epsilon)|^2 \Big)
        (2-C\epsilon) - C' \epsilon^4 \bigg) E
      + C'' H^\frac{2}{5} E^\frac{13}{15} D^\frac{11}{15},
\]
for suitable constants $C$, $C'$, and $C''$. We have assumed without loss of
generality that $\epsilon < 1$.
Then \eqref{thirteen}
follows by simplifying.

\medskip

\textbf{Step~2.}
We now turn to estimate \eqref{eq:Hbd}. We further restrict $\epsilon'>0$ if
necessary so that $2-C(\epsilon')^2|\log(\epsilon')|^2 \GS 0$. In view of
\eqref{thirteen}, we have
\begin{equation}
  \frac{d}{dt}H(t)
    \LS CH^\frac{2}{5} E^\frac{13}{15} D^\frac{11}{15}
  \quad\text{for all $t\GS 0$.}
\label{E:DODDY}
\end{equation}
Without loss of generality we may assume that $t\mapsto E(t)$ is strictly
decreasing (since otherwise $-\frac{d}{dt}E(t)=D(t)=0$ and the evolution becomes
stagnant). This allows us to express $H$ as a function of $E$, turning
\eqref{E:DODDY} into
\begin{equation}
  -\frac{dH}{dE} \LS CH^\frac{2}{5} E^\frac{13}{15} D^{-\frac{4}{15}}.
\label{E:NOBBY}
\end{equation}
By smallness of $\epsilon$ and \eqref{E:REU}, the constant in the HED inequality
in Lemma~\ref{L:HED} can be considered universal, hence $D\GS C E^2H^{-1}$.
Using this in \eqref{E:NOBBY} gives
\[
  -\frac{d(H^\frac{1}{3})}{dE}\LS C\frac{d(E^\frac{4}{3})}{dE}.
\]
Integrating from $E=E(t)$ for $t>0$ to $E_0$ (note that $0 \LS E(t) \LS E_0$),
we get
\begin{equation}
\label{eq:ini_H2}
  H(t)^\frac{1}{3} \LS H_0^\frac{1}{3} + CE_0^\frac{4}{3}.
\end{equation}
A second application of the HED inequality implies
\begin{equation}
  E_0 \LS C D_0^\frac{1}{2} H_0^\frac{1}{2}
    = C \big( \epsilon_0^6 E_0^{-2} \big)^\frac{1}{2} H_0^\frac{1}{2}
  \quad\Longleftrightarrow\quad
  E_0 \LS C \epsilon_0^\frac{3}{2} H_0^\frac{1}{4},
\label{E:LOBBY}
\end{equation}
where we have used \eqref{eq:E2D}. Substituting \eqref{E:LOBBY} into
\eqref{eq:ini_H2}, we obtain \eqref{eq:Hbd}.
\end{proof}

\begin{proof}[Proof of Theorem~\ref{T:OPTIMAL}]
Notice first that because of item~(ii) of Definition~\ref{D:TS}, the interfaces
$\Gamma_t$ of our Mullins-Sekerka solution remain locally finite, hence
rectifiable. To prove that they remain $\delta$-AR domains, we need to control
the $\BMO$-norm of their normals. This can be accomplished using
Lemma~\ref{L:CONTROL} and the fact that $t \mapsto \epsilon(t)$ is nonincreasing
because of Corollary~\ref{cor:epsilon}. The gradient flow structure immediately
gives $E(t)\LS E_0$ and Proposition \ref{P:DECAY} implies that $D(t)\LS D_0$ and
that the map $t\mapsto E(t)$ is convex, whence (as in the proof of
Proposition~\ref{prop:brezis}) it follows that
\[
  D(t) \LS \frac{E_0}{t}.
\]

In view of \eqref{eq:refine_HED}, \eqref{eq:dtH}, the inequalities in
\eqref{E:ASSNS} are satisfied for $\tilde{H}$, $E, D$ with
\begin{align*}
\tilde{H}:=H/\alpha,\quad \alpha= 4-C\epsilon_0^2|\log\epsilon_0|^2, \quad \text{and}\quad \gamma=1-C\epsilon_0^2|\log\epsilon_0|^2.
\end{align*}
By \eqref{eq:Hbd} and \eqref{eq:dtH}, we have
$$\tilde{H}_{\text{max}}\leq (1+C\epsilon_0^2)H_0/\alpha\qquad\text{and}\qquad R:=\sigma \frac{HD}{E} \;\;\text{with}\; \sigma:=C\epsilon_0^2.$$
An application of Lemma \ref{L:DECAY} then gives the decay rate for $E(t)$ from \eqref{eless}.

To obtain the second decay rate for $D(t)$ from~\eqref{dless}, we use $\dot{E}\leq -D$, $\dot{D}\leq 0$, and the decay estimate of $E$ to get
\begin{align*}
(t-s)D(t)\LS \int_s^t D(\tau)\, d\tau\LS E(s)\LS \frac{H_0}{4e s}(1+C\epsilon_0^2|\log \epsilon_0|^2),\quad \text{for all}\; 0<s<t.
\end{align*}
Consequently,
\begin{align*}
D(t)\leq \frac{H_0}{4e s(t-s)}(1+C\epsilon_0^2|\log \epsilon_0|^2),\quad \text{for all}\; 0<s<t.
\end{align*}
Minimizing the right-hand side (with $s=t/2$) gives
\begin{align*}
D(t)\leq \frac{H_0}{e t}(1+C\epsilon_0^2|\log \epsilon_0|^2).
\end{align*}
This proves the result.
\end{proof}


\subsection{Hessian Not Positive Semidefinite}

Here we prove nonconvexity of $E$.

\begin{proof}[Proof of Proposition~\ref{prop:nonconvex}]
We will construct a family of admissible interfaces $\epsilon \mapsto
\Gamma_\epsilon$ and smooth and compactly supported $W_\epsilon$ such that the
right-hand side of \eqref{E:HES} is \emph{strictly negative} while
\begin{equation}
  E(\Gamma_\epsilon)^2 D(\Gamma_\epsilon) \LS C\epsilon^{1/2}
  \quad\text{and}\quad
  H(\Gamma_\epsilon)\LS C \epsilon^\frac{4}{3} |\log \epsilon|
\label{eq:smallness}
\end{equation}
for some universal constant $C$. In the first step, we introduce the general idea of the construction, depending on a parameter $R$, showing that we can obtain a negative Hessian. In the second step, we introduce an additional parameter $r$ and choose $r$ and $R$ to depend on $\epsilon$ so that our construction  satisfies all three properties listed above.

\medskip

\textbf{Step 1:} Let $\Gamma \coloneq \{(x,h(x)): x\in \R\}$ be a smooth curve
such that
\begin{gather*}
  h(x) = \begin{cases}
      \sin(x)
        & \text{for $x \in (-R,R)$,}
\\
      0 & \text{for $x \not\in (-R-1,R+1)$,}
    \end{cases}
\\
  \text{$h$ odd},
  \quad
  \|h'\|_{L^\infty(\R)} \LS C,
\end{gather*}
so that the curvature $\kappa$ of $\Gamma$ satisfies
\begin{equation}
  \|\kappa\|_{L^\infty(\Gamma)} + \|\nabla_\tau \kappa\|_{L^\infty(\Gamma)} \LS C,
  \quad
  \|\kappa\|_{L^2(\Gamma)} \GS CR,
\label{eq:curvature}
\end{equation}
where $C>0$ is a universal constant and $R>1$ is large. We choose a smooth speed $W
\colon \Gamma\longrightarrow \R$ with vanishing integral, such that the function
$U$, defined by
\[
  U(x) \coloneq W\big( x,h(x) \big)
  \quad\text{for all $x\in \R$,}
\]
satisfies $\|U\|_{C^1(\R)}\LS 1$ and
\begin{gather*}
  U(x) = \begin{cases}
       1 & \text{for $x \in (-4R, -1)$,}
\\
      -1 & \text{for $x \in (1,4R)$,}
    \end{cases}
\\
  \supp(U) \subset (-4R-2, 4R+2).
\end{gather*}
Hence $U$ is a truncated and regularized step function, with absolute value
equal to $1$ where $\Gamma$ oscillates, but fairly insensitive to the curvature
of $\Gamma$. We remark that for compactly supported $W$ with $\int_{\Gamma}W\,d\sigma=0$, one has $W\in \H_{-\frac{1}{2}}$. To see this, it suffices to show  $W=N(\phi)$ with $\phi\in L^2(\Gamma)$ (then $\int_{\Gamma} WSW \, d\sigma= \int_{\Gamma}W\phi\, d\sigma\LS \|W\|_{L^2(\Gamma)}\|\phi\|_{L^2(\Gamma)}<\infty$). Indeed, let $\phi(x):=\int_{\Gamma} E(x-y)W(y)\, d\sigma(y)$ be the single layer potential. Then
\begin{align*}
\phi(x)=\int_{\Gamma} (E(x-y)-E(x))W(y)\,d\sigma(y)\LS C\|W\|_{L^\infty(\Gamma)}(1+|x|)^{-1}.
\end{align*}
Here the first equality is by the mean-zero property of $W$, and the inequality is due to  $|E(x-y)-E(x)|\leq C/(1+|x|)$,
where $C$ depends on the support of $W$. The decay estimate in particular gives that $\phi\in L^2(\Gamma)$.

We will now argue that the right-hand side of \eqref{E:HES} is negative if $R$ is
sufficiently large.
First, it follows immediately from the construction that
\[
  \|\nabla_\tau W\|_{L^2(\Gamma)}^2 \LS C
  \quad\text{and}\quad
  \|\kappa W\|_{L^2(\Gamma)}^2 \GS \|\kappa\|_{L^2(\Gamma)}^2 \GS CR,
\]
using \eqref{eq:curvature} and that $W(x)$ is nonconstant only for $|x|$ near $1$ and $4R$.

It therefore suffices to show that
\begin{equation}
  \left| \int_\Gamma M(V) W^2 \,d\sigma \right| \LS C.
\label{eq:tildeV2}
\end{equation}
To see this, we first observe that, if $u^\pm$ denotes the harmonic extension of
$\kappa$ to $\Omega^\pm$, then the normal derivatives have vanishing integral
along the interface:
\[
  \int_\Gamma \tau_\varkappa\Big( \nu\cdot \nabla u^\pm \Big) \,d\sigma = 0.
\]
This follows from an approximation argument and the divergence theorem, because
$\Delta u^\pm = 0$ in $\Omega^\pm$; see Proposition~1.7.8 in
\cite{MitreaMitreaMitrea2023}. Recalling the definition \eqref{E:DOUBLE} of the
operator $M$ with $V = N(\kappa)$, we conclude that $M(V)$ has vanishing
integral along $\Gamma$. This enables us to rewrite \eqref{eq:tildeV2} and estimate:
\[
  \left|\int_\Gamma M(V) W^2 \,d\sigma\right|
    = \left|\int_\Gamma M(V) (W^2-1) \, d\sigma\right|
    \LS \int_{\Gamma\setminus \{ W^2=1\}} |M(V)| \, d\sigma,
\]
where we have used $|W(x,h(x))|\leq 1$. By construction, $W^2 \equiv 1$ in  $(B_{4R}(0)
\setminus B_1(0) )\times \R$. To establish \eqref{eq:tildeV2}, it is
therefore sufficient to estimate the $L^1$ norm of $M(V)$ on the complement of this
set. To this end, we will use the double layer potential operator
\[
  u^\pm(x)
    = \int_{\R^2} K(x,y) \kappa(y) \,d\sigma(y)
  \quad\text{for $x \in \Omega^\pm$,}
\]
with $E(x) =
-\frac{1}{2\pi} \log|x|$, $E(x,y) \coloneq E(x-y)-E_*(y)$ the fundamental solution from
\eqref{E:SLP2}, and
\[
  K(x,y) \coloneq \nu_t(y) \cdot \nabla_y E(x,y);
\]
see \cite[(3.28)]{MarinMartellMitreaMitreaMitrea2022}. We remark that
\begin{equation}
  |K(x,y)| \LS \frac{C\text{dist}(x, \Gamma)}{|x-y|^2}
  \quad\text{and}\quad
  |\nabla_x K(x,y)| \LS \frac{C}{|x-y|^2}
\label{eq:kernel}
\end{equation}
for $x \in \Omega^\pm$ and $y \in \Gamma$, for some $C$. Since $\kappa$ is
supported in $B_{2R}(0) \times \R$, we have
\begin{align}
  & \int_{\Gamma \cap (B_{4R}(0)^c \times \R)} |M(V)| \,d\sigma
\label{E:ADEEN}
\\
  & \qquad
    \LS C \|\kappa\|_{L^\infty(\Gamma)}
      \int_{\Gamma \cap (B_{4R}(0)^c \times \R)}
      \int_{\Gamma \cap (B_{2R}(0) \times \R)}
        |\nabla_x K(x,y)|\, d\sigma(y) d\sigma(x)
    \LS C.
\nonumber
\end{align}
Here we have used the fact that, for large $|x|$, the interface coincides with the axis
$\R\times \{0\}$. To compute the $L^1$ norm of $M(V)$ over $(B_1(0)\times \R) \cap
\Gamma$, we write
\begin{align}
  u^\pm(x)
    & = \int_\Gamma K(x,y) \,\big( \kappa\eta \big)(y) \,d\sigma(y)
      + \int_\Gamma K(x,y) \,\big( \kappa(1-\eta) \big)(y) \, d\sigma(y)
\nonumber
\\
    & \eqcolon u^\pm_1(x)+u^\pm_2(x)
\label{E:SPLITT}
\end{align}
and $V_i \coloneq \llbracket \nu\cdot\nabla u_i \rrbracket$ for $i=1,2$, with
$\eta:\R\to [0,1]$ a smooth cut-off function such that
\begin{gather*}
  \eta(x) = \begin{cases}
       1 & \text{for $x \in B_2(0)\times\R$,}
\\
       0 & \text{for $x \not\in B_3(0)\times\R$}.
    \end{cases}
\end{gather*}
Using Hölder's inequality and \eqref{E:SMALL}, we obtain
\begin{align}
  \int_{\Gamma \cap (B_1(0)\times\R)} |M(V_1)| \,d\sigma
    & \LS C \left( \int_{\Gamma \cap (B_1(0)\times\R)} |M(V_1)|^2
        \,d\sigma \right)^\frac{1}{2}
\label{E:DVAH}
\\
    & \vphantom{\int}
      \LS C \|M(V_1)\|_{L^2(\Gamma)}
      \LS C \|\kappa\eta\|_{L^2(\Gamma)}
      \LS C,
\nonumber
\end{align}
with a constant $C$ that depends on the surface measure of $\Gamma \cap
(B_1(0)\times\R)$. For $V_2$ we use the estimate in \eqref{eq:kernel} for
$|\nabla_x K(x,y)|$ and the fact that $u(1-\eta)$ is supported in $B_2(0)^c
\times \R$, to conclude  for all $x\in \Gamma\cap (B_1(0)\times\R)$ that
\[
  |M(V_2)(x)|
    \LS C\int_{\Gamma\cap (B_2(0))^c\times\R} \frac{|\kappa(y)|}{|x-y|^2}
        \,d\sigma(y)
    \LS C \|\kappa\|_{L^\infty(\Gamma)}
    \overset{\eqref{eq:curvature}} \LS C.
\]
This implies that
\begin{equation}
  \int_{\Gamma\cap (B_1(0)\times\R)} |M(V_2)| \,d\sigma \LS C.
\label{E:TREE}
\end{equation}
Combining \eqref{E:ADEEN}, \eqref{E:DVAH}, and \eqref{E:TREE} completes the proof of
\eqref{eq:tildeV2}.

\medskip

\textbf{Step~2.} For any $\epsilon > 0$ sufficiently small, we now introduce an additional parameter $r$
and fix
\begin{equation}
  r \coloneq \epsilon^\frac{1}{3}
  \quad\text{and}\quad
  R \coloneq L\epsilon^{-\frac{1}{3}}
\label{E:SCLL}
\end{equation}
for a large but universal constant $L$ to be specified below. Using these parameters, we will construct interfaces $\Gamma_\epsilon$ and speeds $W_\epsilon$ such that \eqref{E:HES} is
negative and \eqref{eq:smallness} is satisfied.

To this end, let $\Gamma_\epsilon \coloneq \{(x,h_\epsilon(x)): x\in
\R\}$ be a smooth curve such that
\begin{gather}
  h_\epsilon(x) = \begin{cases}
      \epsilon\sin\bigg( \DST\frac{x}{\sqrt{\epsilon}} \bigg)
        & \text{for $x \in (-R,R)$,}
\\
      0 & \text{for $x \not\in (-R-1,R+1)$,}
    \end{cases}
\label{E:PROPHE}
\\
  \text{$h_\epsilon$ odd},
  \quad
  \|h_\epsilon'\|_{L^\infty(\R)} \LS C \epsilon^\frac{1}{2},
  \quad
  \|h_\epsilon''\|_{L^\infty(\R)} \LS C,
  \quad
  \int_\R h_\epsilon(x) \,dx = 0.
\nonumber
\end{gather}
Notice that the curvature $\kappa_\epsilon$ of $\Gamma_\epsilon$ satisfies
\begin{equation}
  \|\kappa_\epsilon\|_{L^\infty(\Gamma_\epsilon)} \LS C,
  \quad
  \|\kappa_\epsilon\|_{L^2(\Gamma_\epsilon)} \GS C,
  \quad
  \|\nabla_\tau\kappa_\epsilon\|_{L^\infty(\Gamma_\epsilon)} \LS C\epsilon^{-\frac{1}{2}}.
\label{eq:curvature2}
\end{equation}
We choose a smooth speed
$W_\epsilon \colon \Gamma\longrightarrow \R$ with vanishing integral, such that
the function $U_\epsilon$, defined by
\[
  U_\epsilon(x) \coloneq W_\epsilon\big( x,h_\epsilon(x) \big)
  \quad\text{for all $x\in \R$,}
\]
satisfies $\|U_\epsilon\|_{C^1(\R)}\LS 1$ and
\begin{gather*}
  U_\epsilon(x) = \begin{cases}
       1 & \text{for $x \in (-4R, -r)$,}
\\
      -1 & \text{for $x \in (r,4R)$,}
    \end{cases}
\\
  \supp(U_\epsilon) \subset (-4R-2, 4R+2).
\end{gather*}

By our choice of $W_\epsilon$ and \eqref{eq:curvature2}, we have
\begin{equation}
 \|\nabla_\tau W_\epsilon\|_{L^2(\Gamma_\epsilon)}^2 \LS Cr^{-1}
  \quad\text{and}\quad
  \|\kappa_\epsilon W_\epsilon\|_{L^2(\Gamma_\epsilon)}^2 \GS CR.
\label{eq:local}
\end{equation}
To estimate the nonlocal term in \eqref{E:HES}, we argue as in Step~1 to obtain
\[
  \left| \int_{\Gamma_\epsilon} M(V_\epsilon)W_\epsilon^2 \,d\sigma_\epsilon \right|
    \LS \int_{\Gamma_\epsilon \setminus \{W_\epsilon^2 = 1\}} |M(V_\epsilon)| \,d\sigma_\epsilon.
\]
Here $V_\epsilon \coloneq \llbracket \nu\cdot \nabla u_\epsilon\rrbracket$ with
harmonic extension
\[
  u^\pm_\epsilon = \int_{\R^2} K(x,y) \kappa_\epsilon(y) \,d\sigma_\epsilon(y)
  \quad\text{for $y \in \Omega^\pm$.}
\]
Hence, recalling \eqref{eq:local} and \eqref{E:SCLL}, negativity of  \eqref{E:HES} will have been established if we can show
\begin{equation}
  \int_{\Gamma_\epsilon \setminus \{W_\epsilon^2 = 1\}} |M(V_\epsilon)| \,d\sigma_\epsilon
    \LS C \bigg( \Big( \frac{r}{\epsilon} \Big)^\frac{1}{2} + 1 \bigg)
\label{eq:MV}
\end{equation}
for a universal constant $C>0$. The constants in \eqref{eq:local} and \eqref{eq:MV} determine the size of $L$, which thereafter is fixed (and hence estimated as $C$).

To prove \eqref{eq:MV}, we argue as in Step~1 for
\begin{align*}
  \int_{\Gamma_\epsilon \cap (B_{4R}(0)^c \times \R)} |M(V_\epsilon)| \,d\sigma_\epsilon
    \LS C.
\end{align*}
Turning to the small ball,
 let $V_{\epsilon,i} \coloneq \llbracket \nu\cdot \nabla
u_{\epsilon,i} \rrbracket$ for $i=1,2$, where $u_{\epsilon,i}$ is defined as in
\eqref{E:SPLITT} with $\kappa_\epsilon$ in place of $\kappa$. Arguing
 as in Step~1 for $V_{\epsilon,2}$, we find
\begin{align*}
  \int_{\Gamma_\epsilon\cap (B_r(0)\times\R)} |M(V_{\epsilon,2})|^2 \,d\sigma_\epsilon
    \LS C.
\end{align*}
For $V_{\epsilon,1}$ we use Hölder's inequality and \eqref{E:SMALL} to deduce
\begin{align*}
  \int_{\Gamma_\epsilon \cap (B_r(0)\times\R)} |M(V_{\epsilon,1})| \,d\sigma_\epsilon
    & \LS C r^\frac{1}{2} \left( \int_{\Gamma_\epsilon \cap (B_r(0)\times\R)}
        |M(V_{\epsilon,1})|^2 \,d\sigma_\epsilon \right)^\frac{1}{2}
\\
    & \vphantom{\int}
      \LS C r^\frac{1}{2} \|M(V_{\epsilon,1})\|_{L^2(\Gamma_\epsilon)}
      \LS C r^\frac{1}{2} \|\kappa_\epsilon\eta\|_{L^2(\Gamma_\epsilon)}
      \LS C \Big( \frac{r}{\epsilon} \Big)^\frac{1}{2},
\end{align*}
where the last inequality uses \eqref{eq:curvature2}. This concludes the proof of \eqref{eq:MV}.

It only remains to establish \eqref{eq:smallness}. First we infer from the
definition of the excess energy $E(\Gamma_\epsilon)$ and the dissipation
$D(\Gamma_\epsilon)$ in Section~\ref{S:GFS} that
\[
  E(\Gamma_\epsilon)
    \LS C\|h_\epsilon'\|_{L^2(\R)}^2
    \LS C\epsilon R,
  \qquad
  D(\Gamma_\epsilon)
    \LS C\|h_\epsilon''\|^2_{\dot{H}^{1/2}(\R)}
    \LS CR\epsilon^{-\frac{1}{2}};
\]
see \cite{ChugreevaOttoWestdickenberg2019} for similar computations. The scaling
\eqref{E:SCLL} then implies the bound on $E^2D$ from \eqref{eq:smallness}.

It remains only to verify the second bound in \eqref{eq:smallness}, that is, to show that $\Gamma_\epsilon$ is in an $(\epsilon^\frac{4}{3} |\log \epsilon|)^\frac{1}{2}$-neighborhood of the flat interface $\R\times\{0\}$. To estimate $H$, we define
 $\pazocal{W}_\epsilon \coloneq \nu_\epsilon \cdot
\pi^\perp_{\BE_2}$, so that
\[
  H(\Gamma_\epsilon)
    = \int_{\Gamma_\epsilon} \pazocal{W}_\epsilon (S\pazocal{W}_\epsilon) \,d\sigma_\epsilon,
\]
where $S$ the single layer potential of \eqref{E:SLP2}. Notice that $\pazocal{W}_\epsilon$
is compactly supported and has vanishing integral because $h_\epsilon$ is an odd
function. Therefore the $E_*$-term in $\pazocal{W}_\epsilon$ can be ignored; it
contributes only a constant that disappears upon integration against
$\pazocal{W}_\epsilon$.
In the
graph coordinates $\Gamma_\epsilon=\{(x,h_\epsilon(x)): x\in \R\}$, the squared distance can be expressed as
\[
  H(\Gamma_\epsilon)
    = -\frac{1}{2\pi} \iint_{\R^2}
      \log \sqrt{|x-y|^2 + |h_\epsilon(x)-h_\epsilon(y)|^2}
        \, h_\epsilon(x) h_\epsilon(y) \,dx \,dy.
\]
Using \eqref{E:PROPHE}, we conclude that
\[
  H(\Gamma_\epsilon) \LS C\epsilon^2 R^2 \log R,
\]
which with \eqref{E:SCLL} yields the bound on $H$ from \eqref{eq:smallness}.

\end{proof}


\appendix

\section{Additional Proofs}
\label{A:PROOF}

Let us first prove the Sobolev embedding used in Lemma~\ref{L:CONTROL}.

\begin{proof}[Proof of \eqref{E:SHSE}]
The statement is equivalent to
\begin{equation}
  \|I_{1/2} h\|_{\BMO(\R)} \LS C \|h\|_{L^2(\R)}
  \quad\text{for all $h\in L^2(\R)$,}
\label{E:TODO}
\end{equation}
where $I_{1/2}$ is the Riesz potential defined for all $x\in\R$ as
\[
  (I_{1/2} h)(x) \coloneq \frac{1}{\sqrt{2\pi}} \int_\R \frac{h(y)}{|x-y|^{1/2}} \,dy.
\]
Indeed one can check that the Fourier transform of $|x|^{-1/2}$ is
$|\xi|^{-1/2}$ up to constants, so that the Riesz transform can be understood as
the inverse of a fractional derivative of order $1/2$. We include a short proof
of \eqref{E:TODO} for the sake of completeness. Let $Q$ be some interval
centered at the origin. We decompose $h = h_1 + h_2$ with $h_1 \coloneq h
\ONE_{2Q}$. For $p \in (1,2)$ and $q \coloneq 2p/(2-p)$ we can then estimate
\begin{align*}
  & \fint_Q |I_{1/2} h_1| \,dx
    \LS \left( \fint_Q |I_{1/2} h_1|^q \,dx \right)^{1/q}
    \LS |Q|^{-\frac{1}{q}} \left( \int_\R |I_{1/2} h_1|^q \,dx \right)^{1/q}
\\
  & \qquad\quad
    \LS C |Q|^{-\frac{1}{q}} \left( \int_\R |h_1|^p \,dx \right)^{1/p}
    \LS C 2^{\frac{1}{q}} |2Q|^{\frac{1}{p}-\frac{1}{q}}
      \left( \fint_{2Q} |h|^p \,dx \right)^{1/p}
\\
  & \qquad\quad
    \LS C 2^{\frac{1}{q}} |2Q|^{\frac{1}{p}-\frac{1}{q}}
      \left( \fint_{2Q} |h|^2 \,dx \right)^{1/2}
    \LS C 2^{\frac{1}{q}} |2Q|^{\frac{1}{p}-\frac{1}{q}-\frac{1}{2}}
      \|h\|_{L^2(\R)},
\end{align*}
using the Jensen and Hardy-Littlewood-Sobolev inequalities. We now define
\[
  c_Q \coloneq \frac{1}{\sqrt{2\pi}} \int_{\R\setminus 2Q} \frac{h(y)}{|y|^{1/2}} \,dy.
\]
For $x \in Q$ and $y \in \R\setminus 2Q$, there holds $|x-y| \GS |Q|/2$. It
follows that
\begin{align*}
  & |(I_{1/2} h_2)(x)-c_Q|
    \LS \frac{1}{\sqrt{2\pi}} \int_{\R\setminus 2Q} |h(y)|
      \left| \frac{1}{|x-y|^{1/2}}-\frac{1}{|y|^{1/2}} \right| \,dy
\\
  & \qquad\quad
    \LS \frac{C}{\sqrt{2\pi}} \int_{\R\setminus 2Q} |h(y)|
      \frac{|x|}{|y|^{3/2}} \,dy
    \LS \frac{C|Q|}{2\sqrt{2\pi}} \|h\|_{L^2(\R)}<
      \left( \int_{\R\setminus 2Q} |y|^{-3} \,dy \right)^{1/2}
\\
  & \qquad\quad
    \LS \frac{C}{2\sqrt{2\pi}} \|h\|_{L^2(\R)}
      \left( \int_{|z|\GS 1} |z|^{-3} \,dz \right)^{1/2},
\end{align*}
where in the last inquality we have changed variables $y \coloneq |Q| z$. We combine these estimates with the triangle inequality to deduce
\[
  \fint_Q |I_{1/2} h - c_Q| \,dx
    \LS \fint_Q |I_{1/2} h_1| \,dx + \fint_Q |I_{1/2} h_2 - c_Q| \,dx
    \LS C \|h\|_{L^2(\R)}
\]
for some constant $C$. The same estimate
holds for any non-centered interval $Q$. Taking the supremum over all intervals
$Q$ and using the fact that we have an analogous equivalence as in
\eqref{E:EQQ}, we obtain \eqref{E:TODO}.
\end{proof}


Now we prove formula \eqref{E:HES} representing the Hessian of the excess
energy.

\begin{proof}[Proof of Lemma~\ref{L:HESSIAN}]
We proceed in six steps. To simplify the notation, we will occasionally not mark
the time dependence of functions.

\medskip

\textbf{Step~1.} Let us start by deriving some inequalities. By
Definition~\ref{D:AI} of admissible interfaces $\Gamma$, the single layer
potential $S \colon L^2(\Gamma) \longrightarrow \dot{H}^1(\Gamma)$ must be
surjective up to constants, with $\dot{H}^1(\Gamma)$ a Banach space when
understood as a space of equivalence classes; see item~\ref{I:TWO} of
Theorem~\ref{T:TP}. By injectivity and the open mapping theorem, the map $S$ is
a homeomorphism. For $n=2$ Lemma~\ref{L:EQUI} controls the operator norms in
terms of the $\BMO(\Gamma)$-norm of the normal, which in turn is bounded by
$\epsilon$ because of Lemma~\ref{L:CONTROL}. In the following, we will therefore
use repeatedly and without further notice that the $\dot{H}^1$-norm and
$\NN{\cdot}_1$ are equivalent, with constants that are universal, meaning
dependent only on $\epsilon$. Recall also that $\|\cdot\|_{L^2(\Gamma)} =
\NN{\cdot}_0$.

For any $T>0$ we have
\begin{align}
  \int_0^T \|W_t\|_{L^2(\Gamma_t)}^6 \,dt
    & \LS \int_0^T \Big( \NN{W_t}_{-\frac{1}{2}}^\frac{2}{3}
      \NN{W_t}_1^\frac{1}{3} \Big)^6 \,dt
\nonumber
\\
    & \LS \sup_{t\in[0,T]} \NN{W_t}_{-\frac{1}{2}}^4
      \left( \int_0^T \NN{W_t}_1^2 \,dt \right) < \infty,
\label{E:VTSQ}
\end{align}
using \eqref{E:INT2} and \eqref{E:NORME}. Thus $W \in L^6_\LOC([0,\infty),
L^2(\Gamma))$.

Using interpolation and Sobolev embedding (recall that $n=2$), there holds
\begin{align*}
  \|W_t\|_{L^4(\Gamma_t)}^2
    & \LS C \|W_t\|_{L^2(\Gamma_t)} \|W_t\|_{\BMO(\Gamma_t)}
\\
    & \LS C \|W_t\|_{L^2(\Gamma_t)} \|W_t\|_{\dot{H}^\frac{1}{2}(\Gamma_t)}
      \LS C \NN{W_t}_0^\frac{3}{2} \NN{W_t}_1^\frac{1}{2},
\end{align*}
with $C$ depending on the Ahlfors constant of the interface, which is universal
because of Proposition~\ref{P:DENS} and Lemma~\ref{L:CONTROL}. It follows that
\begin{align}
  \int_0^T \|W_t\|_{L^4(\Gamma_t)}^4 \,dt
    & \LS C \int_0^T \Big( \NN{W_t}_{-\frac{1}{2}}^\frac{2}{3}
        \NN{W_t}_1^\frac{1}{3} \Big)^3 \NN{W_t} \,dt
\nonumber
\\
    & \LS C \sup_{t\in[0,T]} \NN{W_t}_{-\frac{1}{2}}^2
      \left( \int_0^T \NN{W_t}_1^2 \,dt \right) < \infty;
\label{E:SUSQ}
\end{align}
see again \eqref{E:NORME}. Thus $W \in L^4_\LOC([0,\infty), L^4(\Gamma))$.

By interpolation and Sobolev embedding (with $n=2$), we have
\begin{align}
  \|\kappa_t\|_{L^4(\Gamma_t)}^2
    & \LS C \|\kappa_t\|_{L^2(\Gamma_t)} \|\kappa_t\|_{\BMO(\Gamma_t)}
\nonumber
\\
    & \LS C \|\kappa_t\|_{L^2(\Gamma_t)} \|\kappa_t\|_{\dot{H}^\frac{1}{2}(\Gamma_t)}
      \LS C \NN{\kappa_t}_{-\frac{1}{2}}^\frac{1}{2} \NN{\kappa_t}_\frac{1}{2}^\frac{3}{2},
\label{E:KAPL}
\end{align}
with a universal constant $C$. Using Proposition~\ref{P:CURVATURE} and
\eqref{E:KAPPE}, the right-hand side of \eqref{E:KAPL} can be bounded uniformly
in $t$. Thus $\kappa \in L^\infty_\LOC([0,\infty), L^4(\Gamma))$.

\medskip

\textbf{Step~2.} Suppose now that we can prove \eqref{E:HES} for $W \in
C^1_c([0,\infty) \times \R^2)$ restricted to the interface $\Gamma_t$, which we
rewrite in weak form as
\begin{align}
  & \int_I \eta'(t) \int_{\R^2} \Big( W (SW) \Big)(t,x) \,d\sigma_t(x) \,dt
\label{E:WEK}
\\
  & \qquad
    = 2 \int_I \eta(t) \int_{\R^2} \bigg( \Big( |\nabla_\tau W|^2 - \kappa^2 W^2 \Big)
      - \Big( M(V) W^2 \Big) \bigg)(t,x) \,d\sigma_t(x) \,dt
\nonumber
\end{align}
for all $\eta \in C^1_c((0,\infty))$ and $I \coloneq [0,\infty)$. The curve of
surface measures $t \mapsto \sigma_t$ is fixed by the Mullins-Sekerka evolution,
so that \eqref{E:KAPPE} is satisfied. Since $C^1_c([0,\infty) \times \R^n)$ is
dense in $L^2_\LOC([0,\infty), H^1(\Gamma))$, for any $W$ as in the statement of
the Lemma, we can find a sequence of functions $W_k \in C^1_c([0,\infty) \times
\R^n)$ such that
\[
  W_k \longrightarrow W
  \quad\text{strongly in $L^2_\LOC\big([0,\infty), H^1(\Gamma)\big)$,}
\]
while at the same time
\begin{align}
  \ESUP_{t\in[0,T]}\ \NN{W_{k,t}}_{-\frac{1}{2}}
  \quad\text{remains uniformly bounded in $k$,}\label{E:NEG}
\end{align}
for every $T>0$. Utilizing the bounds derived in Step~1, we can pass to the
limit on either side of \eqref{E:WEK}. Indeed \eqref{E:VTSQ}, \eqref{E:SUSQ},
and \eqref{E:KAPL} together imply that
\[
  \kappa^2 W^2, M(V) W^2 \in L^2_\LOC([0,\infty), L^2(\Gamma)).
\]
For the latter we made use of the fact that $M \colon L^2(\Gamma_t)
\longrightarrow L^2(\Gamma_t)$ is continuous with operator controlled by the
$\BMO(\Gamma_t)$-norm of the normal, thus universally because of
Lemma~\ref{L:CONTROL}; recall \eqref{E:SMALL}. To control the left-hand side of
\eqref{E:WEK}, we make use of the fact that by Definition~\ref{D:TS}, admissible
tangent vectors are normal speeds obtained as $W_t = N(\phi_t)$ for suitable
Dirichlet data $\phi_t \in L^2(\Gamma_t)$ for a.e.\ time $t$. This implies in
particular that $SW_t = \phi_t$ must have sufficiently fast decay at infinity. Since $SW_t$ and $SW_{k,t}$ are uniformly bounded in $L^2(\Gamma_t)$ for a.e. $t$ (the latter is a consequence of \eqref{E:NEG}, uniform boundedness of $W_{k,t}\in L^2_\LOC([0,\infty), \H_1)$ and the interpolation inequality \eqref{E:INT2}), the convergence of the left hand side of \eqref{E:WEK} then follows from \cite[(3.46)]{MarinMartellMitreaMitreaMitrea2022} and a standard approximation argument.
For the following, it therefore suffices to consider $W \in C^1_c([0,\infty)
\times \R^2)$.

\medskip

\textbf{Step~3.} Consider now the vaguely continuous map $(s,t) \mapsto
\sigma_{s,t}$ introduced above. Recall that vague convergence implies Kuratowski
convergence of the supports of $\sigma_{s,t}$, hence of $\Gamma_{s,t}$; see the
proof of \cite[Proposition~5.1.8]{AmbrosioGigliSavare2008}. Suppose there exists
a family of points $(s,t) \mapsto x_{s,t} \in \Gamma_{s,t}$, with $x_{s,t}$
contained in some bounded set for all $(s,t)$ in the support of $\eta$. We claim
that for any $\alpha>0$ there exists a constant $C$ depending only on $\alpha$,
the family $x_{s,t}$, and the \emph{Ahlfors constant} of $\Gamma_{s,t}$ with
the following property: Writing $\Gamma \coloneq \Gamma_{s,t}$ and $\sigma
\coloneq \HAUS^{1} \lfloor \Gamma$, we have that
\begin{equation}
  \|f\|_{L^1(\Gamma, w\sigma)}
    \LS C \|f\|_{L^2(\Gamma)}
  \quad\text{for all $f \in L^2(\Gamma)$,}
\label{E:EMBED}
\end{equation}
where $w \coloneq (1+|\cdot|)^{-(1+\alpha)/2}$. Indeed by the Cauchy-Schwarz
inequality we obtain
\[
  \int_\Gamma \frac{|f(x)|}{(1+|x|)^{(1+\alpha)/2}} \,d\sigma(x)
    \LS \|f\|_{L^2(\Gamma)}
      \Bigg( \int_\Gamma \frac{d\sigma(x)}{(1+|x|)^{1+\alpha}} \Bigg)^\frac{1}{2},
\]
so we only need to bound the integral on the right-hand side. We estimate
\begin{align*}
  \int_\Gamma \frac{d\sigma(x)}{(1+|x|)^{1+\alpha}}
    & \LS C(x_*) \Bigg( \int_{\Gamma \cap B_1(x_*)}
        \frac{d\sigma(x)}{(1+|x-x_*|)^{1+\alpha}}
\\
    & \qquad + \sum_{j=0}^\infty
      \int_{\Gamma \cap \big( B_{2^{j+1}}(x_*) \setminus B_{2^j}(x_*) \big)}
        \frac{d\sigma(x)}{(1+|x-x_*|)^{1+\alpha}} \Bigg)
\\
    & \LS C(x_*) \Bigg( \sigma\big( B_1(x_*) \big)
      + \sum_{j=0}^\infty 2^{-j(1+\alpha)} \sigma\big( B_{2^{j+1}}(x_*) \big) \Bigg),
\end{align*}
where $x_* \in \Gamma$ is some point and
\begin{equation}
  C(x_*)
    := \sup_{x\in\R^2} \left( \frac{1+|x-x_*|}{1+|x|} \right)^{1+\alpha}
    = (1+|x_*|)^{1+\alpha}
    < \infty.
\label{E:CX0}
\end{equation}
Since $\Gamma$ is Ahlfors regular (hence the surface measure $\sigma$ is
doubling) there holds
\[
  \sigma\big( B_{\lambda r}(x) \big) \LS C \lambda \sigma\big( B_r(x) \big)
    \quad\text{for all $x\in\Gamma$ and $\lambda\GS 1$,}
\]
for some constant $C$ that depends on $n$ and the Ahlfors constant of $\Gamma$.
Then
\[
  \int_\Gamma \frac{d\sigma(x)}{(1+|x|)^{1+\alpha}}
    \LS C(x_*) \sigma\big( B_1(x_*) \big)
      \Bigg( 1 + C \sum_{j=0}^\infty 2^{-j\alpha} \Bigg)
    < \infty.
\]
This proves \eqref{E:EMBED}. Recall that the Ahlfors constant of $\Gamma$ can be
controlled by the $\BMO$-norm of its normal; see Proposition~\ref{P:DENS}. Since
\eqref{E:CX0} depends continuously on $x_*$, it remains bounded if we take the
$\sup$ over a compact subset of $\R^2$.

\medskip

\textbf{Step~4.} By Definition~\ref{D:AI} of admissible interfaces, for every
$(s,t)$ there exists an unbounded and (simply) connected subset $\Omega_{s,t}
\subseteq \R^2$  such that $\partial\Omega_{s,t} = \Gamma_{s,t}$. We denote by
$\chi_{s,t}$ the characteristic function of the set $\Omega_{s,t}$. Then it
follows that distributionally with respect to the variables $(s,t)$ and $x$
there holds
\begin{equation}
  \partial_t \chi_{s,t} = V_{s,t} \sigma_{s,t}
  \quad\text{and}\quad
  \partial_s \chi_{s,t} = w_{s,t} \sigma_{s,t}.
\label{E:DISTRI}
\end{equation}
Now pick test functions $\eta \in C^1_c(I^2)$ and $\xi \in C^1_c(\R^2)$. Then
\begin{align}
  & \iint_{I^2} \partial_t \Big( \partial_s \eta(s,t) \Big)
      \int_{\R^2} \xi(y) \, \chi_{s,t}(y) \,dy \,ds \,dt
\label{E:MIXED}
\\
  & \qquad
    = \iint_{I^2} \partial_s \Big( \partial_t \eta(s,t) \Big)
      \int_{\R^2} \xi(y) \, \chi_{s,t}(y) \,dy \,ds \,dt
\nonumber
\end{align}
because second mixed derivatives commute for smooth functions. We used $I
\coloneq (0,\infty)$ for notational convenience. Applying \eqref{E:DISTRI}, we
obtain
\begin{align}
  & \iint_{I^2} \partial_t \Big( \partial_s \eta(s,t) \Big)
      \int_{\R^2} \xi(y) \, \chi_{s,t}(y) \,dy \,ds \,dt
\nonumber
\\
  & \qquad
    = -\iint_{I^2} \partial_s \eta(s,t)
      \Bigg( \int_{\R^2} \xi(y) V_{s,t}(y) \,d\sigma_{s,t}(y) \Bigg) \,ds \,dt
\label{E:LIMIT}
\\
  & \qquad
    = \iint_{I^2} \eta(s,t) \, \frac{\partial}{\partial s}
      \Bigg( \int_{\R^2} \xi(y) V_{s,t}(y) \,d\sigma_{s,t}(y) \Bigg) \,ds \,dt.
\nonumber
\end{align}
The analogous identity holds for the right-hand side in \eqref{E:MIXED}.

We insert in \eqref{E:LIMIT} $\xi(y) \coloneq E(x,y) \rho_R(y)$, where $E(x,y) =
E(x-y)-E_*(y)$ is the modified fundamental solution of \eqref{E:SLP2}, $x \in
\R^2$ a fixed point, and $\rho$ a compactly supported cut-off function. That is,
we require that $\rho_R(x) \coloneq \rho(x/R)$ for all $x\in \R^n$ and $R>0$,
where $\rho \in C^1_c(\R^2)$ satisfies $\rho(\R^2) \subseteq [0,1]$ and
\[
  \rho(x) = 1
  \quad\text{for $|x| \LS 1$,}
  \qquad
  \rho(x) = 0
  \quad\text{for $|x| \GS 2$.}
\]
Notice that indeed $\xi \in C^1_c(\R^2)$. We wish to pass to the limit $R \to
\infty$ in \eqref{E:LIMIT} and the corresponding expression involving $W_{s,t}$.
To prove that this is possible we first observe that $y \mapsto E(x,y)$ decays
like $|y|^{-1}$ for $|y|\to \infty$, by construction. Then we apply
\eqref{E:EMBED} to $V_{s,t} \in L^2(\Gamma_{s,t})$ with base point $x_* \coloneq
x_{s,t}$, which is contained in a bounded subset of $\R^2$ for all relevant
$(s,t)$, by assumption. Since the $\BMO$-norms of the interfaces $\Gamma_{s,t}$
are bounded uniformly (recall Lemma~\ref{L:CONTROL}), we can choose in
\eqref{E:EMBED} a constant $C$ that works for all relevant $(s,t)$. We also
assume that $x \not\in \Gamma_{s,t}$ for all such $(s,t)$, so that $\xi$ is
indeed smooth. Then
\begin{equation}
  \Big\{ d\mu_{s,t}(y) \coloneq E(x,y) V_{s,t}(y) \,d\sigma_{s,t}(y) \Big\}_{s,t}
\label{E:FAMILY}
\end{equation}
is a \emph{tight} family of measures. With $\varkappa := (1-\alpha)/2 > 0$ we
can estimate
\begin{align}
  \int_{\R^2} (1+|y|)^\varkappa \,d|\mu_{s,t}|(y)
    & \LS C \int_{\R^2} (1+|y|)^\varkappa \frac{|V_{s,t}(y)|}{(1+|y|)}
      \,d\sigma_{s,t}(y)
\label{E:MOMENT}
\\
    & \LS C \int_{\R^2} \frac{|V_{s,t}(y)|}{(1+|y|)^{(1+\alpha)/2}}
        \,d\sigma_{s,t}(y)
      \LS C' \|V_{s,t}\|_{L^2(\Gamma_{s,t})}
\nonumber
\end{align}
for all $t \GS 0$, because of \eqref{E:EMBED}. Here we have use that
\[
  |E(x,y)| \LS C (1+|y|)^{-1}
  \quad\text{for all $y\in \R^n$,}
\]
for some $C$ that depends only on the fixed point $x \in \R^n$. The constant
$C'$ in \eqref{E:MOMENT} depends additionally on $\alpha$ and the Ahlfors
constant of $\Gamma_t$, and is therefore universal; see the discussion above.
The measures in \eqref{E:FAMILY} are hence finite and have bounded
$\varkappa$-moments, all controlled by the $L^2(\Gamma_{s,t})$-norms of the
normal speeds $V_{s,t}$. Then we can pass to the limit $R\to\infty$ in
\eqref{E:LIMIT}. The same argument applies to the family $W_{s,t} \in
L^2(\Gamma_{s,t})$. Since $\eta$ was arbitrary, we conclude that
\begin{equation}
  \frac{\partial}{\partial s}
      \Bigg( \int_{\R^2} E(x,y) V_{s,t}(y) \,d\sigma_{s,t}(y) \Bigg)
    = \frac{\partial}{\partial t}
      \Bigg( \int_{\R^2} E(x,y) W_{s,t}(y) \,d\sigma_{s,t}(y) \Bigg)
\label{E:POINTW}
\end{equation}
in the distributional sense in $I^2$, for any $x\in\R^2$ as above; cf.\ Step~2
in \eqref{E:SCHEME}.

\medskip

\textbf{Step~5.} The argument of Step~4 still works if we multiply
\eqref{E:POINTW} by a $\psi \in C^1_c(\R^n)$ and then integrate in $x$ with
respect to the Lebesgue measure $\LEB^n$. $n=2$. Indeed
\[
  y \mapsto \int_{\R^n} \psi(x) E(x,y) \,dx
  \quad\text{for $y \in \R^n$}
\]
is a smooth function that decays like $|y|^{-1}$ as $|y|\to \infty$. Taking a
suitable limit of finite linear combinations of tensor products $\eta_k(s,t)
\psi_k(x)$ with $\eta_k \in C^1_c(I^2)$ and $\psi_k \in C^1_c(\R^2)$ we obtain,
for any $\varphi \in C^1_c(I^2 \times \R^2)$, the identity
\begin{align}
  & \iint_{I^2} \eta(s,t) \int_{\R^2} \varphi_{s,t}(x) \frac{\partial}{\partial s}
      \Bigg( \int_{\R^2} E(x,y) V_{s,t}(y) \,d\sigma_{s,t}(y) \Bigg)\,dx \,ds \,dt
\label{E:ALMO}
\\
  & \qquad
    = \iint_{I^2} \eta(s,t) \int_{\R^2} \varphi_{s,t}(x) \frac{\partial}{\partial t}
      \Bigg( \int_{\R^2} E(x,y) W_{s,t}(y) \,d\sigma_{s,t}(y) \Bigg)\,dx \,ds \,dt.
\nonumber
\end{align}
Choosing a sequence of test functions $\eta_k$ that concentrate on $\{s=0\}$,
then passing to the limit $k\to \infty$, we obtain a statement analogous to
\eqref{E:ALMO}, with $s$ replaced by $0$ and integration only in $(t,x)$. We
will write $V_t \coloneq V_{0,t}$ etc.\ in the following.

Considering another sequence of $\varphi_k \in C^1_c([0,\infty) \times \R^2)$
obtained by a suitable truncation and mollification of $W_t \sigma_t$, then
passing to the limit, we obtain
\begin{align}
  & \int_I \eta(t) \int_{\R^2} W_t(x) \Bigg(
        \bigg( \frac{\partial}{\partial s} SV_t(x) \bigg)_+
          + \bigg( \frac{\partial}{\partial s} SV_t(x) \bigg)_-
      \Bigg) \,d\sigma_t(x) \,dt
\label{E:RES}
\\
    & \qquad
      = 2 \int_I \eta(t) \int_{\R^2} W_t(x)
        \bigg( \frac{\partial}{\partial t} SW_t(x) \bigg) \,d\sigma_t(x) \,dt.
\nonumber
\end{align}
Note that the measures $W_t \sigma_t$ are locally finite because $W_t \in
L^2(\Gamma_t)$ and by item~(ii) of Definition~\ref{D:TS}. The (nontangential)
limits of $\SL V_t$ (cf. \eqref{E:SLP2}) from either side of the interface $\Gamma_t$ give $SV_t$.
But the $s$-derivatives typically jump across $\Gamma_t$, so we must take the
average of the two one-sided limits, here marked by the subscripts $\pm$. Along
the interface, the traces are in $\dot{H}^1(\Gamma_t)$, hence $1/2$-Hölder
continuous. Since $W \in C^1_c([0,\infty) \times \R^2)$ by assumption, no jump
discontinuity occurs on the right-hand side of \eqref{E:RES}.

We now use that the family $\{\Gamma_t\}_t$ is a solution of the Mullins-Sekerka
equation. Then $SV_t$ coincides with the curvature $\kappa_t$. This is Step~3 in
\eqref{E:SCHEME}. We write
\begin{equation}
  \left( \frac{\partial}{\partial s} SV_t \right)_\pm
    = D_s \kappa_t - W_t \Big( \nu_t\cdot\nabla SV_t \Big)_\pm.
\label{E:UN}
\end{equation}
For $n=2$ the convective derivative $D_s \coloneq \partial_s + (W_t \nu_t) \cdot
\nabla$ of the curvature along an evolving family of curves (with ``time''
parameter $s$) is known to be
\begin{equation}
  D_s \kappa_t = -W_t''- \kappa_t^2 W_t,
\label{E:DEUX}
\end{equation}
where $'$ marks differentiation with respect to the arclength parameter; see for
example \cite[p.\ 443, setting $\alpha = 0$]{KimiaTannenbaumZucker1992}. We
insert \eqref{E:UN},\eqref{E:DEUX} into \eqref{E:RES} and integrate by parts,
recalling \eqref{E:DOS}, to obtain the identity
\begin{align}
  & \int_I \eta(t) \int_{\R^2} W_t
      \bigg( \frac{\partial}{\partial t} SW_t \bigg) \,d\sigma_t \,dt
\label{E:GOAL}
\\
  & \qquad
    = \int_I \eta(t) \int_{\R^2} \bigg(
        \Big( |\nabla_\tau W_t|^2 - \kappa_t^2 W_t^2 \Big)
          - \Big( M(V_t) W_t^2 \Big) \bigg) \,d\sigma_t \,dt.
\nonumber
\end{align}

\medskip

\textbf{Step~6.} Consider now a family $\{\Gamma_t\}_t$ as in
Definition~\ref{D:TS} with corresponding normal speeds $V_t$. Let $U, W \in
C^1_c([0,\infty)\times \R^2)$ be given such that
\[
  \int_{\R^2} U(t,x) \,d\sigma_t(x)
    = \int_{\R^2} W(t,x) \,d\sigma_t(x)
    = 0
  \quad\text{for all $t$.}
\]
With $\SL$ the single layer operator from \eqref{E:SLP2} we define
\[
  \varphi(t,x) \coloneq U(t,x) \Big( \SL \big( W(t,\cdot)|_{\Gamma_t} \big) \Big)(x)
  \quad\text{for $(t,x) \in [0,\infty) \times \R^n$.}
\]
This function is $C^\infty$ outside the interfaces $\Gamma_t$, where it is only
continuous. Moreover, the one-sided limits from either side of $\Gamma_t$ exist;
see below. We want to apply formula \eqref{E:CONTI} for generic $0\LS t_1 < t_2$
with $D_t\varphi \coloneq \partial_t + (W_t\nu_t) \cdot \nabla$. We compute
\begin{equation}
  \partial_t \Big( U \, (\SL W) \Big)(t,x)
    = \Big( \partial_t U \, (\SL W) \Big)(t,x)
    + U(t,x) \, \frac{\partial}{\partial t} \SL W(t,x).
\label{E:TPHI}
\end{equation}
For fixed $x \in \R^2 \setminus \Gamma_t$, the map $y \mapsto E(x,y) W(t,y)$,
with $E(x,y) \coloneq E(x-y)-E_*(y)$ the kernel from \eqref{E:SLP2}, is in
$C^1_c([0,\infty) \times \R^2)$. Using \eqref{E:CONTI} we obtain
\begin{align*}
  \frac{\partial}{\partial t} \SL W(t,x)
    & = \int_{\R^2} E(x,y) \Big( D_t W -\kappa V W \Big)(t,y) \,d\sigma_t(y)
\\
    & \qquad + \int_{\R^2} \Big( \nu_t(y) \cdot \nabla_y E(x,y) \Big) V(t,y) W(t,y) \,d\sigma_t(y)
\\
    & = \Big( \SL(D_t W - \kappa V W) \Big)(t,x) + \Big( \DL(VW) \Big)(t,x)
\end{align*}
for a.e.\ $t$, with $\DL$ the double layer potential; see
\cite[(3.28)]{MarinMartellMitreaMitreaMitrea2022}. Next we have
\begin{align}
  \nu_t(x) \cdot \nabla \Big( U \, (\SL W) \Big) (t,x)
    & = \Big( (\nu \cdot \nabla U) \, (\SL W) \Big)(t,x)
\nonumber
\\
    & + U(t,x) \, \int_{\R^2} \nu_t(x) \cdot \nabla_x E(x-y)  W(t,y) \,d\sigma_t(y)
\label{E:SIMMY}
\end{align}
because $E_*(y)$ does not depend on $x$. We now average the nontangential limits
of \eqref{E:TPHI} and \eqref{E:SIMMY} from either side of the interface
$\Gamma_t$. The nontangential limit of the single layer potential $\SL f$ equals
$S f$ for any $f \in L^2(\Gamma)$; see
\cite[(3.47)]{MarinMartellMitreaMitreaMitrea2022}. The two nontangential limits
of $\DL(VW)$ sum up to $2K(VW)$; see
\cite[(3.31)]{MarinMartellMitreaMitreaMitrea2022}. Averaging the nontangential
limits of the integral in \eqref{E:SIMMY} brings up the operator $K^\#(W)$. Note
that $U$ is smooth across $\Gamma_t$. We arrive at
\begin{multline}
  \left. \int_{\R^2} U(SW) \,d\sigma_t \right|^{t_2}_{t_1}
    = \int_{t_1}^{t_2} \int_{\R^2} \bigg(
        \Big( D_t U - \kappa UV \Big) \, SW + UV \, K^\#(W)
\\
    + U \, S\Big( D_t W - \kappa V W \Big) + U \, K(VW) \bigg) \,d\sigma_t \,dt.
\label{E:DOUB}
\end{multline}
The right-hand side of this expression can be simplified. Indeed we have
\begin{align*}
  \int_{\R^2} U V \, K^\#(W) \,d\sigma_t
    & = \int_{\R^2} \int_{\R^2} \Big( \nu_t(x)\cdot E(x-y) \Big) \, U(x) V(x) \, W(y)
      \,d\sigma_t(y) \,d\sigma_t(x)
\\
    & = \int_{\R^2} W \, K(UV) \,d\sigma_t,
\end{align*}
by Fubini's theorem. For the first integral on the right-hand side of
\eqref{E:DOUB} we write
\begin{align}
  & \int_{\R^2} \Big( D_t U - \kappa U V \Big) \, SW \,d\sigma_t
    =  - \left( \int_{\R^2} \Big( D_t U - \kappa V U \Big) \,d\sigma_t \right)
      \left( \int_{\R^2} E_* W \,d\sigma_t \right)
\nonumber
\\
  & \qquad
    + \int_{\R^2} \int_{\R^2} E(x-y) \, \Big( D_t U - \kappa U V \Big)(x) \,
      W(y) \,d\sigma_t(y) \,d\sigma_t(x).
\label{E:OOO}
\end{align}
Using \eqref{E:CONTI} again, we observe that
\[
  \int_{\R^2} \Big( D_t U - \kappa U V \Big)(t,x) \,d\sigma_t(x)
    = \frac{\partial}{\partial t} \bigg( \int_{\R^2} U(t,x) \,d\sigma_t(x) \bigg)
    = 0
\]
because $U(t,\cdot)$ is assumed to have vanishing integral. Therefore the first term
on the right-hand side of \eqref{E:OOO} vanishes. Similarly, we compute that
\begin{align}
  & \int_{\R^2} U \, S\Big( D_t W - \kappa V W \Big) \,d\sigma_t
    =  - \left( \int_{\R^2} U \,d\sigma_t \right)
      \left( \int_{\R^2} E_* \Big( D_t W - \kappa V W \Big) \,d\sigma_t \right)
\nonumber\\
  & \qquad
    + \int_{\R^2} \int_{\R^2} E(x-y) \, U(x) \, \Big( D_t W - \kappa V W \Big)(y)
      \,d\sigma_t(y) \,d\sigma_t(x),
\label{E:TTT}
\end{align}
and again the first term on the right-hand side vanishes because $W$ has zero
integral. So \eqref{E:OOO} and \eqref{E:TTT} coincide if $U=W$, by symmetry of $E$.
This gives
\[
  \left. \int_{\R^2} W(SW) \,d\sigma_t \right|^{t_2}_{t_1}
    = 2 \int_{t_1}^{t_2} \int_{\R^2} W \bigg(
        S\Big( D_t W - \kappa V W \Big) + K(V W) \bigg) \,d\sigma_t \,dt.
\]
Dividing by $t_2-t_1$, then letting $t_2 \to t_1$, we obtain for a.e.\ $t_1 = t$
the identity
\[
  \frac{d}{dt} \int_{\R^2} W (SW) \,d\sigma_t
    = 2 \int_{\R^2} W \bigg(
        S\Big( D_t W - \kappa V W \Big) + K(V W) \bigg) \,d\sigma_t,
\]
which corresponds to the first equality in \eqref{E:SCHEME}. This is equivalent
to
\begin{equation}
 \int_I \eta'(t) \int_{\R^2} W_t (SW_t) \,d\sigma_t \,dt
    = \int_I \eta(t) \int_{\R^2} W_t
      \bigg( \frac{\partial}{\partial t} SW_t \bigg) \,d\sigma_t \,dt
\label{E:SHOT}
\end{equation}
for all $\eta \in C^1_c(I)$. Combining \eqref{E:SHOT}, \eqref{E:GOAL}, we obtain
identity \eqref{E:HES}.
\end{proof}


\section*{Acknowledgements}
We gratefully acknowledge Felix Otto for raising this research question and for
gen\-er\-ous\-ly sharing both thoughts and notes that helped launch and guide
our study. In addition we thank Saša Lukić and Richard Schubert for interesting and beneficial discussions. Moreover we greatly appreciate the feedback and insightful suggestion of the anonymous referee, which contributed significantly to the current form of~Lemma~\ref{L:DECAY}.

We acknowledge funding by the Deutsche Forschungsgemeinschaft (DFG,
German Research Foundation) -- Project number 442047500 through the
Collaborative Research Center ``Sparsity and Singular Structures'' (CRC 1481).

Finally, we acknowledge the hospitality of the Institut des Hautes Études
Scientifiques (IHES), where first thoughts and discussions on this topic took
place.


\printbibliography

\end{document}